\documentclass[a4paper,10pt,reqno]{amsart}
\usepackage{amssymb,latexsym,bbm,amsmath,amsfonts}
\usepackage{amsthm}
\usepackage[mathscr]{eucal}
\usepackage{rotating}
\usepackage{multirow}

\numberwithin{equation}{section}
\newtheorem{theorem}{Theorem}
\newtheorem{lemma}{Lemma}
\newtheorem{proposition}{Proposition}

\theoremstyle{definition}

\theoremstyle{remark}
\newtheorem{remark}{Remark}
\begin{document}
	
	\title[Perturbed Linear ODEs and SDEs]
	{Critical perturbation sizes for preserving decay rates in solutions of perturbed nonlinear differential equations}
	
	\author{John A. D. Appleby}
	\address{School of Mathematical
		Sciences, Dublin City University, Glasnevin, Dublin 9, Ireland}
	\email{john.appleby@dcu.ie} 
	
	\author{Subham Pal}
	\address{School of Mathematical
		Sciences, Dublin City University, Glasnevin, Dublin 9, Ireland}
	\email{subham.pal2@mail.dcu.ie}
	
	\thanks{JA is supported by the RSE Saltire Facilitation Network on Stochastic Differential Equations: Theory, Numerics and Applications (RSE1832).} 
	\date{9 September 2025}
	
	\begin{abstract}
		This paper develops necessary and sufficient conditions for the preservation of asymptotic convergence rates of deterministically and stochastically perturbed ordinary differential equations in which the solutions of the unperturbed equations exhibit have power--law  decay to zero.   
	\end{abstract}
	
	\maketitle
	
	\section{Introduction} 
	This paper develops necessary and sufficient conditions for the preservation of asymptotic convergence rates of deterministically and stochastically perturbed ordinary differential equations in which the solutions of the unperturbed equations exhibit have power--law  decay to zero. 
	
	In this paper we classify the rates of convergence to a limit of the solutions of scalar ordinary and stochastic differential equations
	of the form 
	\begin{equation} \label{eq.odepert}
		x'(t)=-f(x(t))+g(t), \quad t>0; \quad x(0)=\xi,
	\end{equation}
	and 
	\begin{equation} \label{eq.sde}
		dX(t)=-f(X(t))\,dt + \sigma(t)\,dB(t), \quad t\geq 0,
	\end{equation}
	where $B$ is a one--dimensional standard Brownian motion. 
	
	We assume that the unperturbed equation
	\begin{equation} \label{eq.ode}
		y'(t)=-f(y(t)), \quad t>0; \quad y(0)=\zeta
	\end{equation}
	has a unique globally stable equilibrium (which we set to be at zero). This is characterised by the condition
	\begin{equation} \label{eq.fglobalstable}
		xf(x)>0 \quad\text{for $x\neq 0$,} \quad f(0)=0.
	\end{equation}
	In order to ensure that \eqref{eq.ode}, \eqref{eq.odepert} and \eqref{eq.sde} have
	continuous solutions, we assume
	\begin{equation} \label{eq.fgsigmacns}
		f\in C(\mathbb{R};\mathbb{R}), \quad g\in C([0,\infty);\mathbb{R}), \quad \sigma\in C([0,\infty);\mathbb{R}).
	\end{equation}
	The condition \eqref{eq.fglobalstable} ensures that any solution of \eqref{eq.odepert} or \eqref{eq.sde} is \emph{global} i.e., that
	\begin{align*}
		\tau_D:=\inf\{t>0\,:\,x(t)\not\in (-\infty,\infty)\}=+\infty, \\ 
		\tau_S=\inf\{t>0\,:\, X(t)\not\in (-\infty,\infty)\}=+\infty, \quad \text{a.s.}
	\end{align*}
	We also ensure that there is exactly one continuous solution of both \eqref{eq.odepert} and \eqref{eq.ode} by assuming
	\begin{equation} \label{eq.floclip}
		\text{$f$ is locally Lipschitz continuous on $\mathbb{R}$}.
	\end{equation}
	This condition ensures the existence of a unique continuous adapted process which obeys \eqref{eq.sde}. 
	
	In \eqref{eq.ode}, \eqref{eq.odepert} and \eqref{eq.sde}, the assumptions that we impose (which place restrictions on the asymptotic behaviour of solutions of \eqref{eq.ode}) ensures that $f(x)$ does \emph{not} have linear leading order behaviour as $x\to 0$. These conditions 
	also prevent solutions of \eqref{eq.ode} hitting zero in finite time. Since $f$ is continuous, we are free to define
	\begin{equation} \label{def.F}
		F(x)=\int_x^1 \frac{1}{f(u)}\,du, \quad x>0,
	\end{equation}
	and avoiding solutions of \eqref{eq.ode} to hitting zero in finite time forces
	\begin{equation} \label{eq.Ftoinfty}
		\lim_{x\to 0^+} F(x)=+\infty.
	\end{equation}
	We notice that $F:(0,\infty)\to\mathbb{R}$ is a strictly decreasing function, so it has an inverse $F^{-1}$. Clearly, \eqref{eq.Ftoinfty} implies that
	\[
	\lim_{t\to\infty} F^{-1}(t)=0.
	\]
	The significance of the functions $F$ and $F^{-1}$ is that they enable us to determine the rate of convergence of solutions of \eqref{eq.ode} to zero, because
	$F(y(t))-F(\zeta)=t$ for $t\geq 0$ or $y(t)=F^{-1}(t+F(\zeta))$ for $t\geq 0$. It is then of interest to ask whether solutions of \eqref{eq.odepert} or of 
	\eqref{eq.sde} will still converge to zero as $t\to\infty$, and to determine conditions (on $g$ and $\sigma$) under which the rate of decay of the solution of the underlying unperturbed equation \eqref{eq.ode} is preserved by the solutions of \eqref{eq.odepert} and \eqref{eq.sde}. The general question of sharp conditions under which stability of 
	perturbed equations  is preserved is examined by Strauss and Yorke in \cite{SY67,SY67b}. For other literature in this direction on limiting 
	equations, consult the references in~\cite{JAJC:2011szeged}.
	
	In this paper, we will impose hypotheses on $f$ which promote power--law type decay in the solution of \eqref{eq.ode}; then, granted those hypotheses, we will establish necessary and sufficient conditions on the forcing terms  $g$ and $\sigma$ which ensure the asymptotic behaviour of solutions of \eqref{eq.ode} is inherited by the solutions of \eqref{eq.odepert} or \eqref{eq.sde}. These results generalise work in Appleby and Patterson, in which power law decay in solutions of perturbed equations was generated by making the assumption that $f$ is a regularly varying function at zero, with index greater than unity. In a sense, one of the goals of this paper is to see to what degree the hypothesis of regular variation can be relaxed in the nonlinearity, while still obtaining precise rates of decay and characterisation of the situation wherein the asymptotic behaviour of the unperturbed equation is preserved. 
	
	There is a nice literature on power--like dynamics in solutions of SDEs, and we invite the reader to consult 
	those by Mao~\cite{MaoOx,Mao92}, Liu and Mao~\cite{LiuMao98,LiuMao:01a} and in Liu~\cite{Liu01} which deal with highly non--autonomous equations, as well 
	as those of Zhang and Tsoi~\cite{ZhangTsoi96, ZhangTsoi97} and Appleby, Rodkina and Schurz \cite{ARS06} which are concerned with autonomous nonlinear equations. 
	
	The role of regular variation in the asymptotic analysis of the asymptotic behaviour of differential equations is a very active area. An important monograph summarising themes in the research up to the year 2000 is Maric~\cite{Maric2000}. Another important strand of research on the exact asymptotic behaviour of non--autonomous ordinary differential equations (of first and higher order) in which the equations have regularly varying coefficients has been developed. For recent contributions, see for example work of Evtukhov and co--workers (e.g., Evtukhov and Samoilenko~\cite{EvSam:2011}) and Koz\'ma~\cite{Kozma:2012}, as well as the references in these papers. These papers tend to be concerned with non--autonomous features which are \emph{multipliers} of the regularly--varying state dependent terms, in contrast to the presence of the nonautonomous term $g$ in \eqref{eq.odepert}, which might be thought of as \emph{additive}. 
	
	The structure of the paper is as follows; Section 2 introduces the hypotheses on $f$ (and indirectly, on $F^{-1}$) that are used in the paper, as well as stating the main results of the paper. Section 3 gives proofs of a number of auxiliary results needed concern the functions $f$, $F$ and $F^{-1}$; this chapter  does not contain any analysis about perturbed differential equations. In Section 4, we show how to connect the differential equation we wish to study to an auxiliary pertuebed differential equation, and we prove the main results about this  auxiliary perturbed ordinary differential equation in Section 5. In Section 6, an asymptotic characterisation of the solution of \eqref{eq.odepert} is completed. Section 7 deals with the stochastic  differential equation \eqref{eq.sde}, and it turns out that the groundwork laid in Section 5 for the auxiliary differential equation can be reused with very minor modifications from the analysis needed in Chapter 6, once arguments for the stochastic equation are employed on a path--by--path basis.

	\section{Discussion of hypothesis; statement of main results}
	In order to do this with reasonable generality, and in a way which seems sensible for modelling, we will assume throughout that the function $f$ is odd and increasing (at least locally to $0$). Assuming oddness forces solutions to decay at the same rate on both sides of the equilibrium, which absent other information about the system, is a reasonable assumption, and has the advantage of placing essentially no restrictions on the sign or sign changes of forcing terms. These assumptions can be weakened to $f$ being asymptotic to an odd and increasing function, but at the cost of complicating the analysis. 
	
	The main assumption we impose on $f$ is that it is ``asymptotic preserving'' at zero. We say an  increasing function $f:(0,\infty)\to (0,\infty)$ is asymptotic preserving at zero if there exists $\epsilon_0\in (0,1)$ such that for all $\epsilon\in (0,\epsilon_0)$ there exists $\underline{\Phi}_f(\epsilon)>0$ such that 
	\begin{equation} \label{eq.fasypres}
		\liminf_{x\to 0^+} \frac{f((1-\epsilon)x)}{f(x)}= \underline{\Phi}_f(\epsilon)
	\end{equation}
	and $\Phi_f(\epsilon)\to 1$ as $\epsilon\to 0^+$. This hypothesis means that, if $a(t)\to 0$ and $b(t)\to 0$ as $t\to\infty$, and $a(t)/b(t)\to 1$, as $t\to\infty$, then $f(a(t))/f(b(t))\to 1$ as $t\to\infty$. In modelling terms, therefore, a small proportional change in the size of the input $y(t)$ (to $(1-\epsilon)y(t)$, say) or a measurement error in $y$ in \eqref{eq.ode}, does not affect the decay rate very much. Concretely, if $a$ is a continuous function such that  $0<a(t)\to 1$ as $t\to\infty$, $y_a$ satisfies the equation 
	\[
	y_a'(t)=-f(a(t)y_a(t)),
	\]  
	and $f$ is asymptotic preserving, then $y_a'(t)\sim -f(y_a(t))$ as $t\to\infty$, and so by asymptotic integration 
	\[
	\lim_{t\to\infty} \frac{F(y_a(t))}{t}=1.
	\]
	This gives the same rate of decay of $y_a$ and $y$, inasmuch as $F(y)$ and $F(y_a)$ have the same asymptotic behaviour. 
	
	It turns out that imposing the condition \eqref{eq.fasypres} excludes from consideration differential equations where the rate of decay of $y$ to zero is slower than any negative power of $t$. The next hypothesis on $f$ that we introduce rules out faster than power law decay; such  hypotheses are of interest, for example, if we want to compare directly the asymptotic behaviour of $y$ and $y_a$ above. To give motivation, in the example above, we have shown that  $F(y_a(t))\sim t$ and $F(y(t))\sim t$, and wonder whether this also implies that $y_a(t)\sim y(t)$ as $t\to\infty$. This situation would prevail if it were known that, for any functions $a$ and $b$ satisfying 
	$a(t)\sim b(t)\to 0$ as $t\to\infty$, then $F^{-1}{a(t)}\sim F^{-1}(b(t))$ as $t\to\infty$. As in the case of $f$ being asymptotic preserving above, we formulate this condition directly on $F^{-1}$ rather than needing to impose an asymptotic property for arbitrary pairs of functions $a$ and $b$. Precisely, we ask that for all $\epsilon\in (0,\epsilon_0)$ there is $\overline{\Phi}_F(\epsilon)\in (0,1)$ such that 
	\begin{equation} \label{eq.Finvasplimsup}
		\limsup_{t\to\infty} \frac{F^{-1}((1+\epsilon)t)}{F^{-1}(t)}=\overline{\Phi}_F(\epsilon)
	\end{equation}
	where  $\overline{\Phi}_F(\epsilon)\to 1$ as $\epsilon\to 0$ and $\underline{\Phi}_F(\epsilon)\in (0,1)$   such that 
	\begin{equation} \label{eq.Finvaspliminf}
		\liminf_{t\to\infty} \frac{F^{-1}((1+\epsilon)t)}{F^{-1}(t)}=\underline{\Phi}_F(\epsilon)
	\end{equation}
	where  $\underline{\Phi}_F(\epsilon)\to 1$ as $\epsilon\to 0$. 
	
	With these conditions imposed on $F^{-1}$, one is able to deduce from $F(y_a(t))\sim t$ and $F(y(t))\sim t$, firstly that $F(y_a(t))\sim F(y(t))\sim t$, and then that $y_a(t)\sim y(t)\sim F^{-1}(t)$ as $t\to\infty$. Thus the condition that $F^{-1}$ is asymptotic preserving in the sense above facilitates the estimation of direct leading order asymptotic behaviour of solutions of the differential equations. 
	
	The asymptotic conditions on $F^{-1}$ above may seem somewhat opaque, and hard to check, but both of these concerns will be addressed below. It is easy to deduce from the above limits that there exist $0<b<a$ such that 
	\[
	-a\leq \liminf_{t\to\infty} \frac{\log F^{-1}(t)}{\log t}\leq \limsup_{t\to\infty} \frac{\log F^{-1}(t)}{\log t}\leq -b,
	\]
	so that the conditions \eqref{eq.Finvaspliminf} and \eqref{eq.Finvasplimsup} are roughly consistent with power law decay in $y$, since the last limits imply $t^{-a-\epsilon}<y(t)<t^{-b+\epsilon}$ for all $\epsilon>0$ and all $t$ sufficiently large. If one seeks faster than power--law decay in $y$, a (non--asymptotic preserving) condition on $F^{-1}$ such as 
	\begin{equation}
		\label{eq.Finvnotpres}
		\lim_{t\to\infty} \frac{F^{-1}((1+\epsilon)t)}{F^{-1}(t)}=0, \quad \epsilon>0
	\end{equation}
	is more appropriate, while slower than power--law decay in $y$ comes from 
	\begin{equation}
		\label{eq.Finvhyperpres}
		\lim_{t\to\infty} \frac{F^{-1}((1+\epsilon)t)}{F^{-1}(t)}=1, \quad \epsilon>0,
	\end{equation}
	and the study of these cases will be carried out elsewhere. In this sense, we may think of this work as a single part of a (three--part) project to determine necessary and sufficient conditions on forcing terms, so that perturbed differential equations inherit the asymptotic behaviour of the unperturbed ODE, and that the project is sub--divided into parts according to whether $F^{-1}$ obeys \eqref{eq.Finvaspliminf} and \eqref{eq.Finvasplimsup} (this paper), or  \eqref{eq.Finvhyperpres} or \eqref{eq.Finvnotpres}.  
	
	We have already noted that $f$ being asymptotic preserving rules out \eqref{eq.Finvhyperpres}; to rule out the faster than power law decay in \eqref{eq.Finvnotpres}, we will assume another asymptotic condition on $f$, which at the same time will imply the necessary asymptotic preserving properties of $F^{-1}$ in \eqref{eq.Finvaspliminf} and \eqref{eq.Finvasplimsup}. This condition is 
	\begin{equation} \label{eq.fasypresmu}
		\limsup_{x\to 0^+} \frac{f(\mu x)}{f(x)}=:\bar{\Phi}_f(\mu)<\mu, \quad \text{ for all $\mu<1$}.
	\end{equation}
	A good deal of preparatory work is needed to show that this condition generates the right conclusions in general, and is natural in the context of asymptotic preserving $y$. This will be provided in the next section. We also note that this condition (together with \eqref{eq.fasypres} and others) bear similarity to the regular variation property, but are weaker than it. For results concerning regular variation, see \cite{BGT}.
	
	To give some initial intuition that our framework covers ``power law'' type decay, but not faster or slower than power law cases, we consider some simple examples of $f$. 
	
	First, for power law behaviour, we take $f(x)=x^\beta$ for $\beta>1$. Then $f$ is increasing on $(0,\infty)$, evidently obeys \eqref{eq.fasypres} and \eqref{eq.fasypresmu}. Moreover, it is easy to see that $F(x)\sim (\beta-1)^{-1}x^{1-\beta}$ as $x\to 0$, from which is can readily be deduced that      
	\[
	F^{-1}(t)\sim \left(\frac{1}{\beta-1}\right)^{-\frac{1}{\beta-1}} t^{-\frac{1}{\beta-1}}, \quad t\to\infty,
	\]
	and this asymptotic relation for $F^{-1}$ shows that it obeys the asymptotic conditions \eqref{eq.Finvaspliminf} and \eqref{eq.Finvasplimsup}. 
	
	As an example of faster than power law decay, we consider the simple equation where we have exponential decay. For that, take $f(x)=x$; then $f$ is increasing and obeys \eqref{eq.fasypres}, but not \eqref{eq.fasypresmu}; indeed 
	\[
	\limsup_{x\to 0} \frac{f(\mu x)}{f(x)}=\mu=:\bar{\Phi}_f(\mu),
	\] 
	so \eqref{eq.fasypresmu} \textit{just fails} to hold (suggesting that the condition \eqref{eq.fasypresmu} is quite sharp). Also since $F(x)=-\log x$, we have $F^{-1}(x)=e^{-x}$ and hence the properties \eqref{eq.Finvaspliminf} and \eqref{eq.Finvasplimsup} do not hold, while property 
	\eqref{eq.Finvnotpres} does. 
	
	As a final simple example, suppose $f(x)=\exp(-1/x)$ for $x\in (0,1)$ and $f(0)=0$. In this case $f$ is so flat at zero that all its derivatives vanish as $x\to 0^+$, and so we expect the rate of decay of $y$ to be very slow.  Indeed, $f$ does not obey \eqref{eq.fasypres}, but does obey \eqref{eq.fasypresmu}. Making a substitution in the integral, we have that  
	\[
	F(x)=\int_{1}^{1/x} \frac{e^v}{v^2}\,dv \sim x^2 e^{1/x}, \quad x\to 0^+,
	\]
	from which we see as $t\to\infty$ that $F^{-1}(t)^2\exp(1/F^{-1}(t))\sim t$. Since $\log$ preserves the asymptotic order, this implies that  $2\log F^{-1}(t)+\frac{1}{F^{-1}(t)}\sim \log t$ as $t\to\infty$, and since the second term is dominant on the left hand side as $t\to \infty$, this gives
	\[
	F^{-1}(t)\sim \frac{1}{\log t}, \quad t\to\infty.
	\] 
	Thus \eqref{eq.Finvhyperpres} holds. However, the non--unit limits in \eqref{eq.Finvaspliminf} and \eqref{eq.Finvasplimsup} means that $F^{-1}$ does not satisfy these limits. 
	
	With some intuition about hypotheses on $f$ now given, we can discuss the main results of the paper. They are mostly proven (and sometimes restated) in Chapter 6. They give (essentially) necessary and sufficient conditions under which the asymptotic rate of decay of solutions of the perturbed 
	equations are inherited from those of \eqref{eq.ode}. We consider first the deterministic equation \eqref{eq.odepert}. 
	\begin{theorem} \label{thm.expertodeiff}
		Suppose that $f$ is an increasing and odd function and is continuous. Suppose further it obeys the properties \eqref{eq.fasypres} and \eqref{eq.fasypresmu}. Suppose further that $g$ is continuous, and that it is known that $x(t)\to 0$ as $t\to\infty$. 
		Then the following statements are equivalent
		\begin{itemize}
			\item[(a)] The functions $f$ and $g$ obey
			\[
			\lim_{t\to\infty} \int_0^t g(s)\,ds \text{ exists}, \quad \lim_{t\to\infty} \frac{\int_t^\infty g(s)\,ds}{F^{-1}(t)}=0;
			\]
			\item[(b)] There is $\lambda\in \{-1,0,1\}$ such that
			\[
			\lim_{t\to\infty} \frac{x(t)}{F^{-1}(t)}= \lambda.
			\]
		\end{itemize} 
	\end{theorem}
	The cases $\lambda=\pm 1$ reproduce the asymptotic behaviour of the solution $y$ of \eqref{eq.ode} according to whether the initial condition is positive or 
	negative. The case $\lambda=0$ means that solutions of the perturbed equation decay more rapidly to zero than those of the unperturbed equation. We believe that this behaviour is rare, but it can arise for special perturbations. To see 
	this, let $\xi\neq 0$ and define  
	\[ 
	g(t)=-2\xi F^{-1}(t) f(F^{-1}(t)) + f(\xi F^{-t})^2, \quad t\geq 0.
	\]
	Then $x(t)=\xi F^{-1}(t)^2$ for $t\geq 0$ is a solution of \eqref{eq.odepert}. It is not hard to show that the integral of $g$ has a finite limit, and that 
	\begin{equation} \label{eq.lambda0example}
		\int_t^\infty g(s)\,ds = \xi F^{-1}(t)^2 - \int_t^\infty f(\xi F^{-1}(s)^2)\,ds.
	\end{equation}
	Since $f$ is increasing and obeys \eqref{eq.fasypresmu}, it can be shown that the integral is $o(F^{-1}(t))$ as $t\to\infty$, by first using a substitution to write it as 
	\[
	\int_0^{F^{-1}(t)} \frac{f(\xi u^2)}{f(u)}\,du,
	\]
	and then showing that the integrand has zero limit as $u\to 0^+$. 
	Since the first term on the right in \eqref{eq.lambda0example}  is clearly $o(F^{-1}(t))$ as $t\to\infty$, we have in this case that $\int_t^\infty g(s)\,ds = o(F^{-1}(t))$ as $t\to\infty$, showing that the case $\lambda=0$ is possible for a general equation. 
	
	It is easy to generate examples for which $\lambda=\pm 1$, and the sign of $\lambda$ and the initial condition do not need to be the same. Let $\eta(t)=F^{-1}(t)+(\xi-1)F^{-1}(t)^2$, so $\eta(0)=\xi$, and note, since $F^{-1}(t)\to 0$ as $t\to\infty$, and $f$ is asymptotic preserving, that $f(\eta(t))\sim f(F^{-1}(t))$ as $t\to\infty$. Let $g(t)=\eta'(t)+f(\eta(t))$ for $t\geq 0$. Then $x$ is a solution of \eqref{eq.odepert}. Moreover, $x(t)=\eta(t)\sim F^{-1}(t)$ as $t\to\infty$, so we are in the case when $\lambda=+1$, although there is no restriction on the sign of $\xi$. Moreover, we have that $\int_0^t g(s)\,ds$ has a finite limit and that 
	\[
	\int_t^\infty g(s)\,ds = \eta(t)-\int_t^\infty f(\eta(s))\,ds
	\]   
	Since $\eta(t)\sim F^{-1}(t)$, we see the integrand is asymptotic  to $f(F^{-1}(s))$ as $s\to\infty$, and hence 
	\[
	\int_{t}^\infty  f(\eta(s))\,ds\sim \int_t^\infty f(F^{-1}(s))\,ds = F^{-1}(t), \quad t\to \infty. 
	\]
	Due to these asymptotic relations, we have that $\int_t^\infty g(s)\,ds=o(F^{-1}(t))$ as $t\to\infty$, so once again all the hypothesis are fulfilled.
	
	A similar argument with the function $\eta(t)=-F^{-1}(t)+(\xi+1)F^{-1}(t)^2$, and $g(t)=\eta'(t)+f(\eta(t))$ shows that $x(t)=\eta(t)$ is a solution of \eqref{eq.odepert}. In this case $x(t)\sim -F^{-1}(t)$ as $t\to\infty$, so we have generated an example of a perturbed equation where $\lambda=-1$, but there is no restriction on the sign of $\xi$. Arguing as above, one can show that $\int_0^t g(s)\,ds$ has a finite limit, and that $\int_t^\infty g(s)\,ds = o(F^{-1}(t)$ as $t\to\infty$.

	It is notable that Theorem \ref{thm.expertodeiff} does not require any sign or pointwise conditions on the rate of decay of $g$; indeed, it can be shown that $g$ need not be absolutely integrable, a strictly weaker condition than the first part of condition (a). Indeed 
	one can have that $\limsup_{t\to\infty} |g(t)|/\Gamma(t)=1$ for arbitrarily rapidly growing $G$, while solutions still obey condition (b). To see this, let $G$ be positive, continuous and increasing, with $G(t)\to\infty$ as $t\to\infty$, and define $G$ by 
	\[
	g(t)=G(t)\sin\left(\left\{\int_0^t G(s)\,ds\right\}^2\right), \quad t\geq 1,
	\]
	and let $g$ be any continuous function on $[0,1]$ so that the left and right limits of $g(t)$ are equal at $t=1$. Clearly we have $\limsup_{t\to\infty} |g(t)|/G(t)=1$. To see that $\int_0^t g(s)\,ds$ tends to a limit as $t\to\infty$, put $L(t)=(\int_0^t G(s)\,ds)^2$. We have for $t\geq 1$ that 
	\[
	\int_1^t g(s)\,ds = \int_{L(1)}^{L(t)} \frac{1}{u^{1/2}}\sin{u}\,du,
	\]
	by making the substitution $u=L(s)$. Using integration by parts on the righthand side, we see that 
	\[
	\int_{L(1)}^T  \frac{1}{u^{1/2}}\sin(u)\,du = -\frac{1}{\sqrt{T}}\cos(T)+\frac{\cos(L(1))}{\sqrt{L(1)}}+\frac{1}{2}\int_{L(1)}^T u^{-3/2} \cos(u)\,du,
	\]
	from which it is clear that the integral on the lefthand side tends to a finite limit as $T\to\infty$. Since $L(t)\to\infty$ as $t\to\infty$ it therefore follows that $\int_0^t g(s)\,ds$ tends to a finite limit, and so it is legitimate to consider $\int_t^\infty g(s)\,ds$ as an improper integral. Moreover, by making the same substitution as before, we get 
	\[
	\int_t^\infty g(s)\,ds = \int_{L(t)}^\infty \frac{\sin(u)}{\sqrt{u}}\,du.
	\]
	As before, 
	\[
	\int_T^\infty \frac{1}{u^{1/2}}\sin(u)\,du = \frac{1}{\sqrt{T}}\cos(T)+\frac{1}{2}\int_T^\infty u^{-3/2} \cos(u)\,du,
	\]
	so since the last integral grows no faster than $T^{-1/2}$ as $T\to\infty$, the righthand side is $O(T^{-1/2})$ as $T\to\infty$. Since $L(t)\to \infty$ as $t\to\infty$, we see that 
	\[
	\int_t^\infty g(s)\,ds = O(L(t)^{-1/2}) = O\left(\frac{1}{\int_0^t G(s)\,ds} \right), \quad t\to\infty.
	\]
	Moreover, if we allow $G$ to grow so rapidly that $\int_0^t G(s)\,ds$ tends to infinity faster than $1/F^{-1}(t)$ as $t\to\infty$, then $\int_t^\infty g(s)\,ds = o(F^{-1}(t))$ as $t\to\infty$, and the perturbed differential equation will have solution that behaves according to $x(t)\sim \lambda F^{-1}(t)$ as $t\to\infty$, where $\lambda\in \{-1,0,1\}$.   
	
	As mentioned before, the  
	oddness of $f$ is assumed so as to ensure that convergence rates from both sides of the equilibrium are the same.  
	
	If a less precise asymptotic bound on the solution is needed, the conditions on $\int_t^\infty g(s)\,ds$ can be somewhat relaxed. Since, under the conditions on $f$ in the last theorem, we have 
	that $y(t)=O(F^{-1}(t))$ as $t\to\infty$, it is reasonble to ask what is required for $x(t)=O(F^{-1}(t))$ as $t\to\infty$. A similar characterisation of the asymptotic behaviour in that case can now be given. 
	
	\begin{theorem} \label{thm.extpertodebdd}
		Suppose that $f$ is an increasing and odd function and is continuous. Suppose further it obeys the properties \eqref{eq.fasypres} and \eqref{eq.fasypresmu}. Suppose further that $g$ is continuous, and that it is known that $x(t)\to 0$ as $t\to\infty$. 
		Then the following statements are equivalent
		\begin{itemize}
			\item[(a)] The functions $f$ and $g$ obey
			\[
			\lim_{t\to\infty} \int_0^t g(s)\,ds \text{ exists}, \quad \int_t^\infty g(s)\,ds=O(F^{-1}(t)), \quad t\to\infty;
			\]
			\item[(b)] $x(t)=O(F^{-1}(t))$ as $t\to\infty$.
		\end{itemize} 
	\end{theorem}
	
	Once the above results has been established, it is straightforward to characterise conditions under which the solution of \eqref{eq.odepert} and its derivative inherit the asymptotic behaviour of those of \eqref{eq.ode}. In that case, under the same hypotheses as above, we prove that the following statements are equivalent:
	\begin{theorem} \label{thm.deriv} 
		Suppose the hypotheses of Theorem~\ref{thm.expertodeiff} hold. Then the following are equivalent:
		\begin{itemize}
			\item[(c)] The functions $f$ and $g$ obey 
			\[
			\lim_{t\to\infty} \frac{g(t)}{f(F^{-1}(t))}=0; 
			\]
			\item[(d)] There is $\lambda\in \{-1,0,1\}$ such that
			\[
			\lim_{t\to\infty} \frac{x(t)}{F^{-1}(t)}= \lambda, 
			\quad
			\lim_{t\to\infty} \frac{x'(t)}{f(F^{-1}(t))}=-\lambda.
			\]
		\end{itemize}
	\end{theorem}
	We notice that solutions of \eqref{eq.ode} with positive initial condition obey (d) with $\lambda=1$, while those with negative initial condition obey (d) 
	with $\lambda=-1$. The condition (c), in the case of positive $g$ and positive initial condition $\xi$, was employed in Appleby and Patterson~\cite{apppatt} to establish condition (a) (with $\lambda=1$). However, condition (a) shows that such a pointwise condition is merely sufficient, rather than necessary, to preserve the asymptotic behaviour of solutions of \eqref{eq.ode}. 
	
	Corresponding results apply to the stochastic equation \eqref{eq.sde}, which we establish in Section 7. We note first that if $\sigma\not\in L^2([0,\infty);\mathbb{R})$, then 
	\[
	\mathbb{P}\left[\limsup_{t\to\infty} \frac{|X(t)|}{F^{-1}(t)} =+\infty  \right]=1.
	\]
	This corresponds to the necessity of the first part of condition (a) to preserve the rate of decay of solutions of \eqref{eq.ode} in the deterministic case. 
	In the case when $\sigma\in L^2(0,\infty)$, we have a sharp characterisation of situations under which the solution of \eqref{eq.sde} inherits the decay rate of solutions of \eqref{eq.ode}. Define, for sufficiently large $t>0$ the function $\Sigma:[T,\infty)\to (0,\infty)$ by  
	\[
	\Sigma^2(t)=2\int_t^\infty \sigma^2(s)\,ds \log\log\left(\frac{1}{\int_t^\infty \sigma^2(s)\,ds}\right), \quad t\geq T,
	\] 
	and we suppose that 
	\[
	\mu:=\lim_{t\to\infty} \frac{\Sigma(t)}{F^{-1}(t)} \in  [0,\infty].
	\]
	Then $\mu\in (0,\infty]$ implies that 
	\[
	\mathbb{P}\left[\lim_{t\to\infty} \frac{X(t)}{F^{-1}(t)} \in {-1,0,1} \right] =0 
	\]
	while $\mu=0$ implies that there exists a $\mathcal{F}^B(\infty)$--measurable random variable $\lambda$ such that $\mathbb{P}[\lambda\in \{-1,0,1\}]=1$ and 
	\[
	\mathbb{P}\left[\lim_{t\to\infty} \frac{X(t)}{F^{-1}(t)}=\lambda\right]=1.
	\]
	In the case when $\mu=+\infty$, we have once again that 
	\[
	\mathbb{P}\left[\limsup_{t\to\infty} \frac{|X(t)|}{F^{-1}(t)} =+\infty  \right]=1,
	\]
	which shows once again that if the noise intensity decays too slowly, then it is impossible for the stochastically perturbed equation to inherit the rate of the rate of decay of the unperturbed deterministic equation. 
	
	Finally, we note that the function $\Sigma$ above is well-- defined for $\sigma\in L^2(0,\infty)$ in the case when 
	$\int_t^\infty \sigma^2(s)\,ds >0$ for all $t$ sufficiently large, but is not well defined in the case that $\int_t^\infty \sigma^2(s)\,ds=0$ for all $t\geq T'$ and some $T'\geq 0$). In this latter case, $\sigma$ is a.e. zero for a certain point on, and the stochastic differential equation essentially collapses to the unperturbed ODE $X'(t)=-f(X(t))$ for $t\geq T'$. In that case, depending on the value of $X(T')$, we have a situation equivalent to the case in which $\mu=0$.   
	
	A result which is less explicit than the above, but parallel to the main result for \eqref{eq.odepert} is the following equivalence:
	\begin{itemize}
		\item[(e)] $\sigma$ and $f$ obey 
		\[
		\sigma\in L^2(0,\infty), \quad \lim_{t\to\infty} \frac{\int_t^\infty \sigma(s)\,dB(s)}{F^{-1}(t)}=0, \quad\text{a.s.};
		\]
		\item[(f)] There exists a $\mathcal{F}^B(\infty)$--measurable random variable $\lambda$ such that $\mathbb{P}[\lambda\in \{-1,0,1\}]=1$ and 
		\[
		\mathbb{P}\left[\lim_{t\to\infty} \frac{X(t)}{F^{-1}(t)}=\lambda\right]=1.
		\]
	\end{itemize}

	These asymptotic results are proven by constructing appropriate upper and lower solutions to the differential equation \eqref{eq.odepert} as in Appleby and Patterson~\cite{apppatt, apppatt2}. In this paper, we use ideas from \cite{apppatt2} to  prove a result in which the solutions of \eqref{eq.odepert} and \eqref{eq.sde} are related to that of an 
	``internally perturbed'' ordinary differential equation of the form $z'(t)=-f(z(t)+\gamma(t))$. The benefit gained from the added difficulty involved in bringing the perturbation inside the argument of the mean--reverting term is that the function $\gamma$ will typically have good pointwise behaviour (obeying for example 
	$\gamma(t)\to 0$ or $\gamma(t)/F^{-1}(t)\to 0$ as $t\to\infty$), while the original forcing functions $g$ or $\sigma$ in \eqref{eq.odepert} and \eqref{eq.sde} 
	may not have nice pointwise bounds. By means of this reformulation of the problem, we are able to determine a very fine characterisation of the desired asymptotic results. 
	
	As hinted at above, the approach of this paper is also very successful for dealing with highly nonlinear equations of the type \eqref{eq.odepert} or \eqref{eq.sde}
	in the case when $F^{-1}$ obeys \eqref{eq.Finvnotpres}; a demonstration of this can be found in the thesis of Al--ansari~\cite{tahanithesis}.
	
	\section{Properties of $f$, $F$ and $F^{-1}$} 
	
	The asymptotic analysis of the differential equation \eqref{eq.odepert} requires a number of preparatory results. 
	
	We first show that the condition \eqref{eq.fasypres} implies $F(x)\to\infty$ as $x\to 0^+$, and that $f(x)=o(x)$ as $x\to 0^+$. This shows that \eqref{eq.fasypres} deals with the case when the unperturbed equation \eqref{eq.ode} has solutions which decay subexponentially.    
	
	\begin{lemma} \label{lemma.Ftoinftyfderiv0}
		If $f\in C((0,\infty);(0,\infty))$ obeys \eqref{eq.fasypresmu}, then $F$ defined by \eqref{def.F} obeys \eqref{eq.Ftoinfty}. Also 
		$f(x)/x\to 0$ as $x\to 0^+$. 
	\end{lemma}
	\begin{proof}
		For every $\mu\in (0,1)$ and $\epsilon\in (0,1)$ there is an $\tilde{x}_1(\epsilon,\mu)>0$ (we may take $\tilde{x}_1<1$ without loss of generality) such that 
		\[
		f(\mu x)< (\bar{\Phi}_f(\mu)+\epsilon)f(x), \quad x\leq \tilde{x}_1(\epsilon,\mu).
		\]	
		Now define $\epsilon=\epsilon(\mu)=(\mu-\bar{\Phi}_f(\mu))/2$. Then $\epsilon\in (0,1/2)$, and set $x_1(\mu)=\min(\tilde{x}_1(\epsilon(\mu),\mu), \mu)$. Then 
		\begin{equation} \label{eq.fmuepsdel}
			f(\mu x)< (\bar{\Phi}_f(\mu)+\epsilon(\mu))f(x), \quad x\leq x_1(\mu).
		\end{equation}
		Define next 
		\[
		\tilde{F}(x)=\int_x^{x_1(\mu)} \frac{1}{f(u)}\,du, \quad x>0.	
		\]
		Thus for $x\leq x_1(\mu)$, by performing a substitution in the integral, and then splitting the resulting integral, we have 
		\[
		\tilde{F}(\mu x)= \mu\int_{x}^{x_1(\mu)/\mu} \frac{1}{f(\mu v)}\,dv
		=\mu\int_{x}^{x_1(\mu)} \frac{1}{f(v)}\frac{f(v)}{f(\mu v)}\,dv + I(\mu).
		\]
		where $I(\mu)=\int_{\mu x_1(\mu)}^{x_1(\mu)} du/f(u)$. Since $\mu<1$, $I(\mu)>0$; also, using \eqref{eq.fmuepsdel}, we get 
		\[
		\tilde{F}(\mu x)> \frac{\mu}{\bar{\Phi}_f(\mu)+\epsilon(\mu)}\tilde{F}(x) + I(\mu), \quad x\leq x_1(\mu).
		\]  
		Set $\alpha_\mu:=\mu/(\bar{\Phi}_f(\mu)+\epsilon(\mu))=\mu/(\mu-\epsilon(\mu))$. Then $\alpha_\mu>1$. Hence, using the fact that $I(\mu)>0$, we have for all $x\in (0,x_1(\mu)]$ that $\tilde{F}(\mu x)>\alpha_\mu \tilde{F}(x)$. 
		Since $\mu<1$, we can iterate this estimate $n$ times to get 
		\[
		\tilde{F}(\mu^n x)> \alpha_\mu^n \tilde{F}(x), \quad x\in (0,x_1(\mu)].
		\] 
		Now take $x\in [\mu x_1(\mu),x_1(\mu)]=:J(\mu)$. Note that $\underline{F}:=\inf_{x\in J(\mu)} \tilde{F}(x)>0$. Therefore 
		\[
		\inf_{x\in [\mu^{n+1}x_1(\mu),\mu^n x_1(\mu)]} \tilde{F}(x)\geq \alpha_\mu^n \underline{F},
		\]
		and since $\alpha_\mu>1$, the righthand side (and hence the left) tends to infinity as $n\to\infty$. Since $\mu<1$, this ensures that $\tilde{F}(x)\to \infty$ as $x\to 0^+$. Finally, observe that 
		$F(x)=\tilde{F}(x)+F(x_1(\mu))$, so $F(x)\to \infty$ as $x\to 0^+$, proving the claim. 
		
		To show $\delta(x):=f(x)/x\to 0$ as $x\to 0^+$, note that we have defined  for $\mu<1$ the quantity
		$\epsilon(\mu)=(\mu-\bar{\Phi}_f(\mu))/2$. Then \eqref{eq.fmuepsdel} still holds. Divide across this inequality by $\mu x$ to get 
		\begin{equation} \label{eq.delmu}
			\delta(\mu x)=\frac{f(\mu x)}{\mu x} < A_\mu  \frac{f(x)}{x}=A_\mu \delta(x), \quad x\leq x_1(\mu),
		\end{equation}
		where  
		\[
		A_\mu:=\frac{\bar{\Phi}_f(\mu)+\epsilon(\mu)}{\mu}
		=\frac{\bar{\Phi}_f(\mu)+\mu}{2\mu}<1,
		\]
		because $\bar{\Phi}_f(\mu)<\mu$. Since $\mu<1$, we may iterate \eqref{eq.delmu} $n$ times to get 
		\[
		0<\delta(\mu^n x)\leq A_\mu^n \delta(x), \quad x\leq x_1(\mu).
		\]
		Since $\mu<1$, this implies
		\[
		\sup_{x\in [\mu^{n+1}x_1(\mu), \mu^n x_1(\mu) ]} \delta(x)
		\leq A_\mu^n \sup_{x\in [\mu x_1(\mu),x_1(\mu)]}\delta(x),
		\]
		so we have $\sup_{x\in [\mu^{n+1}x_1(\mu), \mu^n x_1(\mu) ]} \delta(x)\to 0$ as $n\to\infty$, which implies $\delta(x)\to 0$ as $x\to 0$, as needed. 
	\end{proof}
	
	Next, in Lemma~\ref{lem.Finvapchar} we show that the asymptotic preserving properties in \eqref{eq.Finvaspliminf} and \eqref{eq.Finvasplimsup} are equivalent to more--readily checked conditions involving $F$ and $f$, thereby bypassing the necessity to obtain precise asymptotic information about the function $F^{-1}$, which is in general difficult to determine in closed--form. 
	
	Once we have done this, the conditions on $F$ and $f$ equivalent to \eqref{eq.Finvaspliminf} and \eqref{eq.Finvasplimsup} will be  shown to result from the asymptotic preserving conditions on $f$ in \eqref{eq.fasypres} and \eqref{eq.fasypresmu}. This is the subject of Lemma~\ref{lem.importfapFinv} and Lemma~\ref{lem.importfmuapFinv}.
	
	We will also see that the conditions \eqref{eq.fasypres} and \eqref{eq.fasypresmu} tend to rule out slower than power--law decay and faster than power--law decay, respectively. 
	
	\begin{lemma} \label{lem.Finvapchar}
		Suppose $f\in C((0,\infty);(0,\infty))$ is an increasing function.
		Suppose $F$ is defined by \eqref{def.F} and obeys \eqref{eq.Ftoinfty}. 
		\begin{itemize}
			\item[(i)] The following are equivalent.
			\begin{enumerate}
				\item[(A)] There exists $L\in (0,\infty)$ such that 
				\[
				\limsup_{x\to 0^+} \frac{F(x)}{x/f(x)}=:L;
				\]
				\item[(B)] $F^{-1}$ obeys \eqref{eq.Finvaspliminf}.
			\end{enumerate}
			Moreover, both imply 
			\begin{equation} \label{eq.PhiL}
				\epsilon L \frac{1}{1+\epsilon(1+L)}\leq 1-\underline{\Phi}_F(\epsilon)\leq \epsilon L. 
			\end{equation}
			\item[(ii)] The following are equivalent.
			\begin{enumerate}
				\item[(A)] There exists $l\in (0,\infty)$ such that 
				\[
				\liminf_{x\to 0^+} \frac{F(x)}{x/f(x)}=:l;
				\]
				\item[(B)] $F^{-1}$ obeys \eqref{eq.Finvasplimsup}.
			\end{enumerate}
			Moreover, both imply 
			\begin{equation} \label{eq.Phil}
				\epsilon l \frac{1}{1+\epsilon(1+l)}\leq 1-\underline{\Phi}_F(\epsilon)\leq \epsilon l. 
			\end{equation}
			\item[(iii)] The following are equivalent
			\begin{enumerate}
				\item[(A)] 
				\[
				\lim_{x\to 0^+} \frac{F(x)}{x/f(x)}=0;
				\]
				\item[(B)] $F^{-1}$ obeys \eqref{eq.Finvhyperpres}. 
			\end{enumerate}
			\item[(iv)] The following are equivalent	
			\begin{enumerate} 
				\item[(A)] 
				\[
				\lim_{x\to 0^+} \frac{F(x)}{x/f(x)}=\infty;
				\]
				\item[(B)] $F^{-1}$ obeys \eqref{eq.Finvnotpres}.  
			\end{enumerate}
		\end{itemize}
	\end{lemma}
	\begin{proof} 
		Since $F(x)\to \infty$ as $x\to 0^+$, we have that $F^{-1}(t)\to 0$ as $t\to\infty$. Define $u(t)=F^{-1}(t)$ for $t\geq 0$. Note that $u'(t)=-f(u(t))$ for all $t\geq 0$ with $u(0)=1$. Let $\epsilon>0$, $t>0$ and integrate the differential equation for $u$ over $[t,(1+\epsilon)t]$. Then 
		\[
		0<u(t)-u((1+\epsilon)t) = \int_t^{(1+\epsilon)t} f(u(s))\,ds.
		\]	
		Since $f$ is increasing, and $u$ is decreasing, $f\circ u$ is decreasing, so by estimating the integral by the integrand at the upper and lower endpoints, we get the inequalities 
		\begin{equation}  \label{eq.fuueps}
			\epsilon t f(u((1+\epsilon)t))\leq u(t)-u((1+\epsilon)t) \leq \epsilon t f(u(t)), \quad t\geq 0.
		\end{equation}
		Take the right inequality in \eqref{eq.fuueps}, and define $\lambda_\epsilon(t)=u((1+\epsilon)t)/u(t)$. Then 
		\begin{equation} \label{eq.lam1}
			1-\lambda_\epsilon(t)\leq \epsilon \frac{tf(u(t))}{u(t)}.
		\end{equation}
		Now take the limsup across this inequality; on the left, we get 
		$1-\underline{\Phi}_F(\epsilon)$. On the right, with $L$ defined as above, we get $\epsilon L$. This follows by making the change of variables $x=u(t)=F^{-1}(t)$ (so $t=F(x)$), and noting that $u(t)\to 0$ as $t\to\infty$. Hence 
		\begin{equation} \label{eq.PhiL1}
			1-\underline{\Phi}_F(\epsilon)\leq \epsilon L.
		\end{equation}
		
		Similarly, take the left inequality in \eqref{eq.fuueps}, divide by $u((1+\epsilon)t)$, and rearrange to get  
		\begin{equation} \label{eq.lam2}
			\frac{1}{\lambda_\epsilon(t)}\geq 	\epsilon \frac{t f(u((1+\epsilon)t))}{u((1+\epsilon)t)}+1.
		\end{equation}
		Next, take the limsup on both sides; on the left, we get $1/\underline{\Phi}_F(\epsilon)$. On the right, we get $\epsilon L/(1+\epsilon) +1$; this follows by making by change of variables $x=u((1+\epsilon)t)=F^{-1}((1+\epsilon)t)$, which gives $F(x)/(1+\epsilon)=t$. Thus 
		\begin{equation} \label{eq.PhiL2}
			\frac{1}{\underline{\Phi}_F(\epsilon)} \geq 1+\frac{\epsilon}{1+\epsilon} L.
		\end{equation}
		Rearrange this inequality to get 
		\[
		1-\underline{\Phi}_F(\epsilon) \geq \epsilon\frac{L}{1+\epsilon(1+ L)},
		\]
		and notice that this, together with \eqref{eq.PhiL1}, yield the inequality \eqref{eq.PhiL}.
		
		Now, we prove part (i). Let (B) hold, so $\underline{\Phi}_F(\epsilon)\in (0,1)$. Then, as $\underline{\Phi}_F(\epsilon)<1$ by \eqref{eq.PhiL1}, we have $L>0$. Next rearrange \eqref{eq.PhiL2} to get  
		$$ 
		L\leq
		\frac{1+\epsilon}{\epsilon}\left( \frac{1}{\underline{\Phi}_F(\epsilon)} -1\right).
		$$
		Since $\underline{\Phi}_F(\epsilon)\in (0,1)$, the righthand side is finite and positive; hence we must have that $L<+\infty$. Thus (B) implies $L\in (0,\infty)$, which is (A). 
		
		On the other hand, if (A) holds, we have $L\in (0,\infty)$. By \eqref{eq.PhiL2}, since $L>0$, we have $	1/\underline{\Phi}_F(\epsilon)>1$, so $\underline{\Phi}_F(\epsilon)<1$. Next, if $L<\infty$, for all $\epsilon>0$ sufficiently small, we have $1-\epsilon L>0$. Rearrange \eqref{eq.PhiL1} to get $\underline{\Phi}_F(\epsilon)\geq 1-\epsilon L>0$. Hence (B) holds, for all $\epsilon\in (0,1/L)$. Since we have already shown, irrespective of the values on the extended real line taken by various limits, that the double inequality \eqref{eq.PhiL} holds, the proof of part (i) is complete. 
		
		The proof of part (ii) is essentially identical to that of (i). First, we note that the inequalities \eqref{eq.lam1} and \eqref{eq.lam2} are still valid. Now, however, we take the liminf across both of these inequalities, rather than the limsup, as in part (i). This generates inequalities identical to  \eqref{eq.PhiL1} and \eqref{eq.PhiL2}, with $L$ replaced by $l$, and $\underline{\Phi}_F(\epsilon)$ replaced by $\overline{\Phi}_F(\epsilon)$. From these inequalities, which are equivalent to the double inequality \eqref{eq.Phil}, we may show that (A) and (B) are equivalent, as in the proof of (i) above.
		
		The proof of part (iii) follows verbatim the proof of part (i). If we assume (A), since $F(x)/(x/f(x))>0$ for $x>0$ sufficiently small, this is the same as having $L=0$ in part (i). We can deduce \eqref{eq.PhiL1} as before, and with $L=0$ it now reads  $1-\underline{\Phi}_F(\epsilon)\leq 0$, while proceeding similarly with \eqref{eq.PhiL2} yields $1/\underline{\Phi}_F(\epsilon)\geq 1$. From these inequalities, we get $\underline{\Phi}_F(\epsilon)=1$, or 
		\[
		\liminf_{t\to\infty} \frac{F^{-1}((1+\epsilon)t)}{F^{-1}(t)}=1.
		\]
		Since $\epsilon>0$ and $F^{-1}$ is decreasing, we have 
		\[
		\limsup_{t\to\infty} \frac{F^{-1}((1+\epsilon)t)}{F^{-1}(t)}\leq 1,
		\]
		and hence combining the limits, we get (B). Conversely (B) implies we have that $\liminf_{t\to\infty}  F^{-1}((1+\epsilon)t)/F^{-1}(t)=1$, so we have $\underline{\Phi}_F(\epsilon)=1$. Then \eqref{eq.PhiL2} yields $1\leq 1+\epsilon L$, or $L\leq 0$, while we must have $L\geq 0$ by the positivity of $F$ and $f$. Hence $L=0$, which yields (A).
		
		To prove part (iv), suppose first (A) holds.
		The inequality \eqref{eq.lam2} is valid. Taking the liminf as $t\to\infty$ on both sides gives 
		\[
		\liminf_{t\to\infty}\frac{1}{\lambda_\epsilon(t)}=+\infty. 
		\]
		Thus 
		\[
		\lim_{t\to\infty} \frac{F^{-1}(t)}{F^{-1}((1+\epsilon)t)}=+\infty, 
		\]
		which is precisely (B). 
		
		Conversely, suppose (B) holds. We still have \eqref{eq.lam1} i.e., 
		\[
		1-\lambda_\epsilon(t)\leq \epsilon t \frac{f(u(t))}{u(t)},
		\]
		and (B) implies $\lambda_\epsilon(t)\to 0$ as $t\to\infty$ for every $\epsilon>0$. Hence 
		\[
		\liminf_{t\to\infty} \epsilon t \frac{f(u(t))}{u(t)}\geq 
		\liminf_{t\to\infty} \{	1-\lambda_\epsilon(t)\}=1-\limsup_{t\to\infty} \lambda_\epsilon(t)=1.
		\]
		Hence 
		\[
		\liminf_{t\to\infty} t \frac{f(u(t))}{u(t)} \geq \frac{1}{\epsilon}.
		\]
		Since $\epsilon>0$ is arbitrary, we may take $\epsilon\to 0^+$, yielding 
		\[
		\lim_{t\to\infty} t \frac{f(u(t))}{u(t)} =+\infty,
		\]
		which yields (A), because $0<u(t)\to 0$ as $t\to 0^+$.
	\end{proof}	
	
	\begin{lemma} \label{lem.importfapFinv}
		Suppose $f\in C((0,\infty);(0,\infty))$ is an increasing function and that \eqref{eq.fasypresmu} holds. Then 
		\[
		\limsup_{t\to\infty} \frac{F(x)}{x/f(x)}=:L<\infty.
		\]
	\end{lemma}
	\begin{proof}
		Fix $\mu\in (0,1)$. Then, by \eqref{eq.fasypresmu}, for every $\eta>0$, there is $\tilde{x}_1(\eta,\mu)<1$ such that 
		\begin{equation} \label{eq.fmuupper}
			f(\mu x)<(\overline{\Phi}_f(\mu)+\eta)f(x), \quad x<\tilde{x}_1(\eta,\mu).
		\end{equation}
		For $x<1$ and $u\in [\mu x,x]$, we have $1/f(u)\leq 1/f(\mu x)$.
		Then 
		\[
		F(\mu x)=F(x)+\int_{\mu x}^x\frac{1}{f(u)}\,du\leq F(x)+\frac{x(1-\mu)}{f(\mu x)}.
		\]	
		Now, define $\rho(x)=F(x)/(x/f(x))$. 
		\[
		\rho(\mu x)=\frac{F(\mu x) f(\mu x)}{\mu x}
		\leq \frac{\{F(x)+x(1-\mu)/f(\mu x)\}  f(\mu x)}{\mu x}
		=\frac{F(x)f(\mu x)}{\mu x}+\frac{1-\mu}{\mu}.
		\]
		Now apply \eqref{eq.fmuupper} to the first term, and using the definition of $\rho$, we get 
		\[
		\rho(\mu x)\leq \frac{1-\mu}{\mu}+\frac{(\overline{\Phi}_f(\mu)+\eta)}{\mu}\rho(x), \quad x<\tilde{x}_1(\eta,\mu).
		\]
		Now take $\tilde{\alpha}(\mu)=\overline{\Phi}_f(\mu)/\mu<1$,  $\eta(\mu)=\mu(1-\tilde{\alpha}(\mu))/2$. Then 
		\[
		\frac{\overline{\Phi}_f(\mu)+\eta(\mu)}{\mu}=\frac{\tilde{\alpha}(\mu)+1}{2}
		=:\alpha_\mu,
		\]
		and $\alpha_\mu\in (0,1)$. Then with $x_1(\mu):=\tilde{x}_1(\eta(\mu),\mu)<1$ we 
		\begin{equation} \label{eq.rhomuxrhox}
			\rho(\mu x)\leq \frac{1-\mu}{\mu}+\alpha_\mu\rho(x), \quad x<x_1(\mu).
		\end{equation}
		Now, by induction it is straightforward to prove for every positive integer $n$ that 
		\[
		\rho(\mu^n x)\leq \frac{1-\mu}{\mu}(1+\alpha_\mu+\alpha_\mu^2+\ldots\alpha_\mu^{n-1})+\alpha_\mu^n \rho(x), \quad x<x_1(\mu).
		\]
		Since $\alpha_\mu<1$, we have the following $n$--independent bound on the right--hand side. Hence 
		\[
		\rho(\mu^n x)\leq \frac{1-\mu}{\mu}\frac{1}{1-\alpha_\mu}+\alpha_\mu \rho(x).
		\]
		Now, take the supremum for $x\in [\mu x_1(\mu),x_1(\mu)]$, so 
		\begin{align*}
			\sup_{y\in [\mu^{n+1}x_1(\mu),\mu^n x_1(\mu)]} \rho(y)
			&=\sup_{x\in [\mu x_1(\mu),x_1(\mu)]} \rho(\mu^n x)\\
			&\leq \frac{1-\mu}{\mu(1-\alpha_\mu)}+\alpha_\mu \sup_{x\in [\mu x_1(\mu),x_1(\mu)]} \rho(x) =: B(\mu)<+\infty,
		\end{align*}
		and from this we see that $L=\limsup_{y\to 0^+} \rho(y)\leq B(\mu)<+\infty$, as required.
	\end{proof}

	\begin{lemma}  \label{lem.importfmuapFinv}
		Suppose $f\in C((0,\infty);(0,\infty))$ is an increasing function and that \eqref{eq.fasypres} holds. Then 
		\[
		\limsup_{t\to\infty} \frac{F(x)}{x/f(x)}=:l>0. 
		\]
	\end{lemma}
	\begin{proof}
		Define $\rho(x)=F(x)/(x/f(x))$ for $x\in (0,1)$. Let $\mu\in (0,1)$. We wish to show that $\liminf_{x\to 0^+} \rho(x)=l>0$. 
		We have 
		\[
		F(\mu x)=F(x)+\int_{\mu x}^x \frac{1}{f(u)}\,du\geq \int_{\mu x}^x \frac{1}{f(u)}\,du. 
		\]
		Thus, using this estimate, the definition of $\rho$, and the monotonicity of $f$, we get  
		\[
		\rho(\mu x) \geq \frac{1}{\mu x}\int_{x\mu}^x \frac{f(\mu x)}{f(u)}\,du\geq  \frac{1}{\mu x}\int_{x\mu}^x \frac{f(\mu x)}{f(x)}\,du = \frac{1-\mu}{\mu}\frac{f(\mu x)}{f(x)}.
		\]
		Since $\mu\in (0,1)$, by taking the liminf in the last estimate, we have, for all $\mu>0$ sufficiently small the estimate 
		\[
		l:=\liminf_{y\to 0^+} \rho(y)=\liminf_{x\to 0^+} \rho(\mu x)
		\geq \frac{1-\mu}{\mu}\underline{\Phi}_f(1-\mu)>0,
		\]
		with the existence of the limit on the right, and the positivity of that limit being a consequence of \eqref{eq.fasypres}. Thus $l>0$. 
	\end{proof}
	
	We next show that the asymptotic preserving property of $f$ is inherited through integration by $F$. 
	
	\begin{lemma} \label{lem.Fap}
		If $f$ obeys \eqref{eq.fasypres}, and $F$ obeys \eqref{eq.Ftoinfty}, then $F$ obeys
		\begin{equation} \label{eq.Fasypres}
			\liminf_{x\to 0^+} \frac{F((1+\epsilon)x)}{F(x)}=:\underline{\Psi}_F(\epsilon)\to 1, \quad \epsilon\to 0^+. 
		\end{equation}
	\end{lemma}
	\begin{proof}
		Since $F$ is decreasing, we have that $F((1+\epsilon)x)<F(x)$, so $\overline{\Psi}_F(\epsilon):=\limsup_{x\to 0^+} F((1+\epsilon)x)/F(x)\leq 1$. Hence, the above limit \eqref{eq.Fasypres} suffices to show that $F$ is asymptotic preserving. 
		
		Since  $f$ obeys \eqref{eq.fasypres}, we have that there exists $\tilde{x}_1(\epsilon')<1$ such that $f((1-\epsilon')x)/f(x)>\underline{\Phi}_f(\epsilon')(1-\epsilon')$ for all $\epsilon'\in (0,1)$ sufficiently small. Now, put $1/(1+\epsilon)=1-\epsilon'$. Then $\epsilon'=\epsilon/(1+\epsilon)$ and we have that there is an $x_1(\epsilon)= \tilde{x}_1(\epsilon/(1+\epsilon))$ such that 
		\[
		\frac{f((1+\epsilon)^{-1}x)}{f(x)}>\underline{\Phi}_f(\epsilon/(1+\epsilon))\frac{1}{1+\epsilon}, \quad x<x_1(\epsilon)
		\]	
		Thus, for $x<x_1(\epsilon)/(1+\epsilon)=:x_2(\epsilon)$ we have 
		\[
		\frac{f(x)}{f((1+\epsilon)x)}>\Phi_f(\epsilon/(1+\epsilon))\frac{1}{1+\epsilon}=:\psi(\epsilon).
		\]
		Note that $x_2(\epsilon)<1/(1+\epsilon)$, so for $x<x_2(\epsilon)$, by making a substitution, we have 
		\[
		F((1+\epsilon)x) = \int_{x}^{1/(1+\epsilon)} \frac{1+\epsilon}{f((1+\epsilon)v)}\,dv
		=(1+\epsilon)\int_x^{x_2(\epsilon)} \frac{1}{f((1+\epsilon)v)}\,dv+ I_1(\epsilon),
		\]
		where 
		\[
		I_1(\epsilon):=(1+\epsilon)\int_{x_2(\epsilon)}^{1/(1+\epsilon)} \frac{1}{f((1+\epsilon)v)}\,dv.
		\]
		Since $x<x_2(\epsilon)$, we have 
		\[
		F((1+\epsilon)x) >(1+\epsilon)\psi(\epsilon)\int_x^{x_2(\epsilon)} \frac{1}{f(v)}\,dv+ I_1(\epsilon)
		= (1+\epsilon)\psi(\epsilon)F(x)+I_2(\epsilon)+I_1(\epsilon),
		\]
		where 
		\[
		I_2(\epsilon)=(1+\epsilon)\psi(\epsilon)\int_1^{x_2(\epsilon)} \frac{1}{f(v)}\,dv.
		\]
		Since $F(x)\to \infty$ as $x\to 0$, this yields
		\[
		\liminf_{x\to 0^+}\frac{F((1+\epsilon)x)}{F(x)}\geq \underline{\Phi}_f(\epsilon/(1+\epsilon)),
		\]
		and since the righthand side tends to 1 as $\epsilon\to 0^+$, the result is proven.
	\end{proof}
	
	One implication of that result is that any function $f$ for which 
	$F(x)f(x)/x\to 0$ as $x\to 0^+$ cannot be asymptotic preserving. 
	
	\begin{remark}
		If $f$ is increasing and $F(x)f(x)/x\to 0$ as $x\to 0^+$, it does not obey \eqref{eq.fasypres}.
	\end{remark}
	\begin{proof}
		Suppose that $f$ is asymptotic preserving. Then by Lemma~\ref{lem.Fap}, we have that 
		\begin{equation} \label{eq.limis1-eps}
			\liminf_{x\to 0^+} \frac{F((1+\epsilon)x)}{F(x)}=\underline{\Psi}_F(\epsilon)\to 1, \quad \epsilon\to 0^+.
		\end{equation}
		On the other hand, by part (iii) of Lemma~\ref{lem.Finvapchar}, we have that $F(x)f(x)/x\to 0$ as $x\to 0^+$ implies 
		\[
		\lim_{t\to\infty} \frac{F^{-1}(\lambda t)}{F^{-1}(t)}=1
		\]
		for all $\lambda>1$. Thus for every $\eta\in (0,1)$ there is a $T(\eta,\lambda)$ such that 
		\[
		(1-\eta)F^{-1}(t)<F^{-1}(\lambda t)< F^{-1}(t), \quad t>T(\eta,\lambda).
		\] 
		Put $x(\eta,\lambda)=F(T(\eta,\lambda))$. If $x=F^{-1}(t)$, we have
		\[
		(1-\eta) x < F^{-1}(\lambda F(x))< x, \quad x<x(\eta,\lambda).
		\] 
		Thus $F((1-\eta)x)>\lambda F(x)$ for all $ x<x(\eta,\lambda)$. Therefore for every $\lambda>1$ and $\eta\in (0,1)$ we have 
		\[
		\liminf_{x\to 0} \frac{F((1-\eta)x)}{F(x)}\geq \lambda.
		\]
		Since $\lambda>1$ is arbitrary, we may let $\lambda\to \infty$ to get 
		\[
		\liminf_{x\to 0} \frac{F((1-\eta)x)}{F(x)}=+\infty.
		\]
		Put $\epsilon=1/(1-\eta)-1>0$. Then $F(x/(1+\epsilon))/F(x)\to\infty$ as $x\to 0^+$, or  $F((1+\epsilon)x)/F(x)\to 0$ as $x\to 0$. But this is inconsistent with \eqref{eq.limis1-eps}, which generates the desired contradiction.
	\end{proof}	
	
	Finally, the function $\overline{\Phi}_f$ in \eqref{eq.fasypresmu} is not merely $O(\mu)$ but $o(\mu)$ as $\mu\to 0^+$.
	
	\begin{lemma} \label{lemma.philogest}
		Suppose $f\in C((0,\infty);(0,\infty))$ is increasing and $f$ obeys \eqref{eq.fasypresmu}. Then there exists $C'>0$ and $\mu_1<1$ such that for all $\mu<\mu_1$ we have $\overline{\Phi}_f(\mu)\leq C'\mu/\log(1/\mu)$. 
	\end{lemma}	
	\begin{proof}
		Define $L(\lambda)=1/\overline{\Phi}_f(1/\lambda)$. Then for all $\lambda>1$ we have 
		\[
		\liminf_{x\to 0^+} \frac{f(\lambda x)}{f(x)}=L(\lambda)>\lambda,
		\]
		the last inequality coming from \eqref{eq.fasypresmu}. Next, let $\lambda_0>1$ be fixed and define $\delta:=L(\lambda_0)-\lambda_0>0$. 
		Now for any arbitrary $\lambda\geq \lambda_0>1$, we have that there exists a positive integer $n=n(\lambda)$ such that $\lambda_0^n \leq \lambda<\lambda_0^{n+1}$. Moreover since $f$ is increasing 
		\[
		L(\lambda)=\liminf_{x\to 0} \frac{f(\lambda x)}{f(x)}
		\geq \liminf_{x\to 0} \frac{f(\lambda_0^n x)}{f(x)}
		=\liminf_{x\to 0} \prod_{j=1}^n \frac{f(\lambda_0^j x)}{f(\lambda_0^{j-1}x)}
		\geq L(\lambda_0)^n.
		\]
		Hence $L(\lambda)\geq L(\lambda_0)^n=(\lambda_0+\delta)^n \geq \lambda_0^n + n \lambda_0^{n-1}\delta$. Since $\lambda_0^{n+1}>\lambda$, this yields
		\[
		L(\lambda)>\frac{\lambda}{\lambda_0}+n \frac{\lambda}{\lambda_0^2}(L(\lambda_0)-\lambda_0).
		\]
		Finally $n+1>\log \lambda/\log\lambda_0$, so 
		\[
		L(\lambda)>\frac{\lambda}{\lambda_0}+ \left(\frac{\log(\lambda)}{\log(\lambda_0)}-1\right) \frac{\lambda}{\lambda_0^2}(L(\lambda_0)-\lambda_0), \quad \lambda\geq \lambda_0.
		\]
		Thus 
		\[
		\liminf_{\lambda\to \infty} \frac{L(\lambda)}{\lambda\log(\lambda)}
		\geq \frac{L(\lambda_0)-\lambda_0}{\log(\lambda_0)\lambda_0^2}=:C>0
		\]
		Hence there exists $\lambda_1>\lambda_0$ such that for all $\lambda>\lambda_1$ we have $L(\lambda)>C/2 \lambda \log\lambda \log(\lambda)$. Put $\mu=1/\lambda$ and $\mu_1=1/\lambda_1<1$. Then for $\mu<\mu_1$ we have 
		\[
		\overline{\Phi}_f(\mu)=\frac{1}{L(\lambda)}<\frac{1}{C/2 \lambda\log(1/\lambda)}=\frac{2}{C}\frac{\mu}{\log(1/\mu)}=:C'\frac{\mu}{\log(1/\mu)},
		\]
		as required. 
	\end{proof}
	
	%We see that the asymptotic preserving properties of $f(x)$ can be induced by power law behaviour in $x$. 
	%
	%\begin{lemma}
	%	Let $f\in C((0,\infty);(0,\infty))$ be increasing. 
	%\begin{itemize}
	%	\item[(i)] If $x\mapsto f(x)/x^{1+\epsilon}$ is increasing for some  $\epsilon>0$, and  $x\mapsto f(x)/x^{1+\eta}$ is decreasing for some $\eta>0$, then $f$ obeys \eqref{eq.fasypres} and \eqref{eq.fasypresmu}. 
	%	\item[(ii)] If $x\mapsto f(x)/x$ is increasing and $x\mapsto f(x)/x^{1+\eta}$ is decreasing for all $\eta>0$, then $f$ obeys \eqref{eq.fasypres} and \eqref{} holds.
	%	\item[(iii)] If 
	%\end{itemize}
	%\end{lemma}
	%\begin{proof}
	%Suppose $\mu<1$. By monotonicity 
	%\[
	%\frac{f(\mu x)}{\mu^{1+\epsilon}}\leq \frac{f(x)}{x^{1+\epsilon}}.
	%\]
	%Then $f(\mu x)/f(x) \leq \mu^{1+\epsilon}$, so taking limits, we see that \eqref{eq.fasypresmu} holds with $\Phi_f(\mu)=\mu^{1+\epsilon}<\mu$ for all $\mu\in (0,1)$. 
	%
	%Let $\epsilon\in (0,1)$, so that $1-\epsilon\in (0,1)$. Then $(1-\epsilon)x<x$, so 
	%\[
	%\frac{f((1-\epsilon)x)}{(1-\epsilon)^\eta x^\eta}\geq \frac{f(x)}{x^\eta}.
	%\]
	%Hence $f((1-\epsilon)x)/f(x) \geq  (1-\epsilon)^\eta$, 
	%and the lefthand side is bounded above by unity, since $f$ is increasing. Thus, taking the liminf as $x\to 0$ yields 
	%\[
	%1\geq \underline{\Phi}_f(\epsilon):=\liminf_{x\to 0^+} \frac{f((1-\epsilon)x)}{f(x)} \geq  (1-\epsilon)^\eta,
	%\]
	%from which \eqref{eq.fasypres} follows. 
	%\end{proof}
	
	The condition \eqref{eq.fasypresmu} is the most ``quantitative'', and therefore is a condition that we speculate may be convenient to impose, or needlessly restrictive. However, if we want to work within the framework of asymptotic preserving $F^{-1}$, even a modest relaxation of \eqref{eq.fasypresmu} forces e.g., \eqref{eq.Finvnotpres} to be satisfied, rather than \eqref{eq.Finvaspliminf} and \eqref{eq.Finvasplimsup}.
	
	A rather natural way to violate \eqref{eq.fasypresmu} is to change it to
	\begin{equation} \label{eq.fnotasypresmu}
		\liminf_{x\to 0^+} \frac{f(\mu x)}{f(x)}=:\underline{\Phi}_f(\mu)\geq \mu, \quad \text{ for some $\mu<1$}.
	\end{equation}
	
	\begin{lemma}
		Suppose $f\in C((0,\infty);(0,\infty))$ is increasing and obeys 
		\eqref{eq.fnotasypresmu}. Then
		\[
		\lim_{x\to 0} \frac{F(x)}{x/f(x)}=+\infty,
		\] 	
		and therefore $F^{-1}$ obeys \eqref{eq.Finvnotpres}.
	\end{lemma}
	\begin{proof}
		As before, define $\rho(x)=F(x)f(x)/x$, and for any $\mu\in (0,1)$ write 
		\[
		F(\mu x)=F(x)+\int_{\mu x}^x \frac{1}{f(u)}\,du,
		\]
		so 
		\[
		\rho(\mu x)=\rho(x)\cdot\frac{f(\mu x)}{\mu f(x)} + \frac{1}{\mu x}\int_{\mu x}^x \frac{f(\mu x)}{f(u)}\,du.
		\]
		Since $f$ is increasing, in the integral we have $1/f(u)>1/f(x)$ for $u>x$, so 
		\begin{equation} \label{eq.rholwrbd1}
			\rho(\mu x)\geq \rho(x)\cdot\frac{f(\mu x)}{\mu f(x)} + (1-\mu)\frac{f(\mu x)}{\mu f(x)}.
		\end{equation}
		Now fix a value of $\mu$ for which \eqref{eq.fnotasypresmu} holds. Defining $\phi_\mu:=\underline{\Phi}_f(\mu)/\mu\geq 1$, we see for every $\epsilon\in (0,1)$ sufficiently small, and $\mu\in (0,1)$ we have that there exists $\tilde{x}_1(\epsilon,\mu)<1$ such that 
		\[
		\frac{f(\mu x)}{\mu f(x)} > \phi_\mu-\epsilon, \quad x\leq \tilde{x}_1(\epsilon,\mu).
		\]
		Hence for all $\epsilon\in (0,1)$ sufficiently small, there is an $\tilde{x}_1(\epsilon,\mu)<1$ such that 
		\begin{equation} \label{eq.rholwrbd2}
			\rho(\mu x)\geq \rho(x) (\phi_\mu-\epsilon) + (1-\mu)(\phi_\mu-\epsilon), \quad x<\tilde{x}_1(\epsilon,\mu).
		\end{equation}
		Letting $a=\phi_\mu-\epsilon$ and $b=(1-\mu)(\phi_\mu-\epsilon)$, and noting that $\mu<1$, by iterating the inequality, we get  
		\[
		\rho(\mu^n x) \geq (\phi_\mu-\epsilon)^n\rho(x)+(1+a+a^2+\cdots+a^{n-1})(1-\mu)(\phi_\mu-\epsilon), \quad x<\tilde{x}_1(\mu,\epsilon)
		\] 
		We consider now two cases. 
		
		\textbf{Case 1: $\phi_\mu=1$.} If this is the case, then $a=1-\epsilon$, and 
		the geometric series sums to 
		\[
		\frac{1-(1-\epsilon)^n}{\epsilon}.
		\]
		Since $\rho(x)>0$, we have 
		\[
		\rho(\mu^n x) \geq\frac{1-(1-\epsilon)^n}{\epsilon}(1-\mu)(1-\epsilon), \quad x<\tilde{x}_1(\mu,\epsilon)
		\]
		Thus 
		\[
		\inf_{x\in [\mu^{n+1}\tilde{x}_1(\mu,\epsilon),\mu^n \tilde{x}_1(\mu,\epsilon)]} \rho(x)
		\geq \frac{1-(1-\epsilon)^n}{\epsilon}(1-\mu)(1-\epsilon).
		\]
		Since $\mu<1$, and $1-\epsilon\in (0,1)$, by taking the limit as $n\to\infty$, we have that 
		\[
		\liminf_{y\to 0^+} \rho(y)\geq \frac{1}{\epsilon}(1-\mu)(1-\epsilon).
		\]
		Finally, since $\epsilon\in (0,1)$ is arbitrary, we may let it tend to zero, giving $\liminf_{y\to 0^+} \rho(y)=+\infty$, as needed.
		
		\textbf{Case 2: There exists $\mu\in (0,1)$ such that $\phi_\mu>1$.} 
		In this case, take $\epsilon=\epsilon_\mu>0$ so small that $a=\phi_\mu-\epsilon>1$. Take $x_1(\mu)=\tilde{x}_1(\epsilon_\mu,\mu)$. 
		Then 
		\[
		\rho(\mu^n x) \geq \frac{a(a^n-1)}{a-1} (1-\mu), \quad x\leq x_1(\mu).
		\]
		Thus 
		\[
		\inf_{x\in[\mu^{n+1}x_1(\mu),\mu^n x_1(\mu)]} \rho(x)
		\geq \frac{a(a^n-1)}{a-1} (1-\mu).
		\]
		Since $a>1$, we have $\inf_{x\in[\mu^{n+1}x_1(\mu),\mu^n x_1(\mu)]} \rho(x)\to\infty$ as $n\to\infty$. Therefore, as $\mu<1$, we have 
		$\liminf_{y\to 0} \rho(y)=+\infty$, as required. 
	\end{proof}
	
	Finally, we may ask what happens in the case that $f$ is not asymptotically preserving in the sense that 
	\begin{equation} \label{eq.fnotap}
		\lim_{x\to 0+} \frac{f(\mu x)}{f(x)}=0, \quad \mu<1.
	\end{equation}
	In this case, $F$ is not asymptotic preserving, and $F^{-1}$ obeys \eqref{eq.Finvhyperpres}. 
	
	\begin{lemma} \label{lem.fnotapFinvhyper}
		Let $f\in C((0,\infty);(0,\infty))$ be increasing. If $f$ obeys \eqref{eq.fnotap} 	then 
		\[
		\lim_{x\to 0^+} \frac{F(\mu x)}{F(x)}=+\infty, \quad \mu<1.
		\]
		and $F^{-1}$ obeys \eqref{eq.Finvhyperpres}.
	\end{lemma}
	\begin{proof}
		By the limit and Lemma \ref{lemma.Ftoinftyfderiv0}, we have that $F(x)\to \infty$ as $x\to 0^+$.Also by the limit, for every $\mu<1$ and $\epsilon\in (0,1)$ there is $\tilde{x}_1(\mu,\epsilon)<1$ such that $x\leq \tilde{x}_1(\mu,\epsilon)$ implies $f(\mu x)<\epsilon f(x)$. Let $x<\tilde{x}_1(\mu,\epsilon)<1$ and write 
		\[
		F(\mu x) = \int_{\mu x}^1 \frac{1}{f(u)}\,du = \mu \int_{x}^{1/\mu} \frac{1}{f(\mu v)}\,dv,
		\]
		so with $I(\mu,\epsilon)= \mu \int_{\tilde{x}_1}^{1/\mu} \frac{1}{f(\mu v)}\,dv$, we get  
		\[
		F(\mu x) = \mu \int_{x}^{\tilde{x}_1} \frac{1}{f(\mu v)}\,dv + I(\mu,\epsilon) 
		> \frac{\mu}{\epsilon} \int_{x}^{\tilde{x}_1} \frac{1}{f(v)}\,dv +I(\mu,\epsilon). 
		\]
		Since $F(x)\to\infty$ as $x\to 0^+$, divide across by $F(x)$ and let $x\to 0^+$ to get 
		\[
		\liminf_{x\to 0^+} \frac{F(\mu x)}{F(x)}\geq \frac{\mu}{\epsilon}.
		\]
		Now, since $\epsilon>0$ is arbitrary, let $\epsilon\to 0^+$, giving 
		$F(\mu x)/F(x)\to +\infty$ as $x\to 0^+$, for any $\mu<1$, or alternatively $F(\lambda x)/F(x)\to 0$ as $x\to 0^+$ for all $\lambda>1$. Thus, for every $\epsilon\in (0,1)$ we have $x_2(\lambda,\epsilon)<1$ such that $x<x_2(\lambda,\epsilon)$ implies $F(\lambda x)<\epsilon F(x)$. Now, let 
		$T_2(\lambda,\epsilon)=F(x_2(\lambda,\epsilon))$. Then for $t>T_2(\lambda,\epsilon)$ we have $F(\lambda F^{-1}(t))<\epsilon t$, so since $F^{-1}$ is decreasing, this becomes $\lambda F^{-1}(t)>F^{-1}(\epsilon t)$ for all $t\geq T_2$. Thus, as $\epsilon\in (0,1)$, we have  
		\[
		1<\frac{F^{-1}(\epsilon t)}{F^{-1}(t)}<\lambda, \quad t\geq T_2(\epsilon,\lambda).
		\]  
		Therefore we have 
		\[
		1\leq \liminf_{t\to\infty} \frac{F^{-1}(\epsilon t)}{F^{-1}(t)}\leq 
		\limsup_{t\to\infty} \frac{F^{-1}(\epsilon t)}{F^{-1}(t)}\leq\lambda.
		\]
		Since $\lambda>1$ is arbitrary, let $\lambda\to 1^-$ to get 
		\[
		\lim_{t\to\infty} \frac{F^{-1}(\epsilon t)}{F^{-1}(t)}=1.
		\]
		But this is true for all $\epsilon\in (0,1)$, so it follows that $F^{-1}$ obeys \eqref{eq.Finvhyperpres}.
	\end{proof}
	
	If the parameterised family of functions $x\mapsto f(x)/x^{1+\eta}$ for various values of $\eta\geq 0$ have certain clearcut monotonicity conditions, it is possible to give a classification of whether $F^{-1}$ obeys \eqref{eq.Finvaspliminf} and \eqref{eq.Finvasplimsup} (the case covered in this paper), \eqref{eq.Finvnotpres} (which subdivides into the case where we have faster or slower than exponential decay), or \eqref{eq.Finvhyperpres} (in which case we have slower than power law decay). The properties of $F^{-1}$ in each case stem from the fact that the monontonicity hypothesis forces $f$ to enjoy the appropriate asymptotic preserving (or non--preserving) property of $f$ (viz., \eqref{eq.fasypres}, \eqref{eq.fasypresmu}, \eqref{eq.fnotap}, \eqref{eq.fnotasypresmu}).  
	
	\begin{proposition}
		Let $f\in C((0,\infty);(0,\infty))$ be increasing. 
		\begin{itemize}
			\item[(i)] If $x\mapsto f(x)/x^{1+\epsilon}$ is increasing for some  $\epsilon>0$, and  $x\mapsto f(x)/x^{1+\eta}$ is decreasing for some $\eta>0$, then $f$ obeys \eqref{eq.fasypres} and \eqref{eq.fasypresmu}. Consequently, $F^{-1}$ obeys \eqref{eq.Finvaspliminf} and \eqref{eq.Finvasplimsup}. 
			\item[(ii)] If $x\mapsto f(x)/x$ is increasing and $x\mapsto f(x)/x^{1+\eta}$ is decreasing for all $\eta>0$, then $f$ obeys \eqref{eq.fasypres} and \eqref{eq.fnotasypresmu}. Consequently, $F^{-1}$ obeys \eqref{eq.Finvnotpres}. 
			\item[(iii)] If $x\mapsto f(x)/x$ is decreasing, then $f$ obeys \eqref{eq.fasypres} and \eqref{eq.fnotasypresmu}. Consequently, $F^{-1}$ obeys \eqref{eq.Finvnotpres}. 
			\item[(iv)] If $x\mapsto f(x)/x^{1+\epsilon}$ is increasing for all $\epsilon>0$, then $f$ obeys \eqref{eq.fnotap}. Consequently, 
			$F^{-1}$ obeys \eqref{eq.Finvhyperpres}. 
		\end{itemize}
	\end{proposition}
	\begin{proof}
		Part (i). Suppose $\mu<1$. By monotonicity 
		\[
		\frac{f(\mu x)}{\mu^{1+\epsilon}}\leq \frac{f(x)}{x^{1+\epsilon}}.
		\]
		Then $f(\mu x)/f(x) \leq \mu^{1+\epsilon}$, so taking limits, we see that \eqref{eq.fasypresmu} holds with $\Phi_f(\mu)=\mu^{1+\epsilon}<\mu$ for all $\mu\in (0,1)$. 
		Let $\epsilon\in (0,1)$, so that $1-\epsilon\in (0,1)$. Then $(1-\epsilon)x<x$, so 
		\[
		\frac{f((1-\epsilon)x)}{(1-\epsilon)^\eta x^\eta}\geq \frac{f(x)}{x^\eta}.
		\]
		Hence $f((1-\epsilon)x)/f(x) \geq  (1-\epsilon)^\eta$, 
		and the lefthand side is bounded above by unity, since $f$ is increasing. Thus, taking the liminf as $x\to 0$ yields 
		\[
		1\geq \underline{\Phi}_f(\epsilon):=\liminf_{x\to 0^+} \frac{f((1-\epsilon)x)}{f(x)} \geq  (1-\epsilon)^\eta,
		\]
		from which \eqref{eq.fasypres} follows. 
		
		Part (ii). Let $\mu<1$. Since $x\mapsto f(x)/x$ is increasing, we have $f(\mu x)/f(x)\leq \mu$, so 
		\[
		\limsup_{x\to 0^+} \frac{f(\mu x)}{f(x)}\leq \mu.
		\] 
		On the other hand, $x\mapsto f(x)/x^{1+\eta}$ is decreasing, we have for $\mu <1$, $f(\mu x)/f(x)\geq \mu^{1+\eta}$. Hence 
		\[
		\liminf_{x\to 0^+} \frac{f(\mu x)}{f(x)}\geq \mu^{1+\eta}.
		\] 
		Since this is true for all $\eta>0$, we may let $\eta\to 0^+$ to give 
		$f(\mu x)/f(x)\to \mu$ as $x\to 0$ for all $\mu<1$. This means that both \eqref{eq.fasypres} and \eqref{eq.fnotasypresmu} hold, and hence that \eqref{eq.Finvnotpres} holds. 
		
		Part (iii). If $x\mapsto f(x)/x$ is decreasing, then for $\mu<1$ we have $f(\mu x)/f(x)>\mu$, so taking limits as $x\to 0^+$ we have that 
		\[
		\underline{\Phi}_f(\mu):=\liminf_{x\to 0^+} \frac{f(\mu x)}{f(x)}\geq \mu,
		\]
		for all $\mu>1$. Thus, we have that $f$ obeys \eqref{eq.fasypres}. 
		Also, we have $\underline{\Phi}_f(\mu)\geq \mu$ for all $\mu<1$, so \eqref{eq.fnotasypresmu} holds also. As a consequence of this last fact, $F^{-1}$ obeys \eqref{eq.Finvnotpres}.
		
		Part (iv). Since $x\mapsto f(x)/x^{1+\epsilon}$ is increasing for all $\epsilon>0$, for $\mu<1$ we have 
		\[
		\frac{f(\mu x)}{f(x)}\leq \mu^{1+\epsilon}.
		\]
		This implies for all $\mu<1$ and all $\epsilon>0$ that 
		\[
		0\leq \limsup_{x\to 0^+} \frac{f(\mu x)}{f(x)}\leq \mu^{1+\epsilon}.
		\]
		Now, let $\epsilon\to \infty$; since $\mu\in (0,1)$, the righthand side tends to zero, and we have for $\mu\in (0,1)$
		\[
		\limsup_{x\to 0^+} \frac{f(\mu x)}{f(x)}=0,
		\]
		which is exactly \eqref{eq.fnotap}. By Lemma \ref{lem.fnotapFinvhyper}, this implies that $F^{-1}$ obeys \eqref{eq.Finvhyperpres}.
	\end{proof}

	\section{Internally Perturbed ODEs and Connection with Externally Perturbed Equations}
	As suggested in the introduction, the asymptotic analysis of the ``externally'' perturbed differential equation 
	\[
	x'(t)=-f(x(t))+g(t), \quad t>0; \quad x(0)=\zeta,
	\]
	is facilitated by considering the related ``internally'' perturbed ordinary differential equation 
	\begin{equation} \label{eq.z}
		z'(t)=-f(z(t)+\Gamma(t)), \quad t>0; \quad z(0)=\xi.
	\end{equation}
	This section is devoted to the analysis of \eqref{eq.z}. The connection between the equations is the following: if 
	\begin{equation} \label{eq.g2}
		\lim_{t\to\infty} \int_0^t g(s)\,ds \text{ exists and is finite}
	\end{equation}
	holds, then the function 
	\begin{equation} \label{eq:Gamma}
		\Gamma(t):=-\int_t^\infty g(s)\,ds, \quad t\geq 0
	\end{equation}
	is well--defined, as is the function $z(t)=x(t)-\Gamma(t)$, $t\geq 0$ where $x$ is the solution of the externally perturbed ODE. It is immediate that $z$ is a solution of \eqref{eq.z}. In fact, under \eqref{eq.g2}, it is true that $x(t)\to 0$ as $t\to\infty$ (and also $z(t)\to 0$ as $t\to\infty$), under mild conditions on $f$. Note that the condition \eqref{eq.g2} is strictly weaker than $g\in L^1(0,\infty)$.
	\begin{theorem} \label{thm.xto0zto0}
		Suppose that $f\in C(\mathbb{R};\mathbb{R})$ obeys \eqref{eq.fglobalstable}. 
		\begin{itemize}
			\item[(a)] Suppose $\Gamma$ is a continuous function obeying $\lim_{t \to \infty}\Gamma(t) = 0$. Then there exists a continuous solution $z$ of \eqref{eq.z} on $[0,\infty)$. Moreover, any such continuous solution is uniformly bounded on $[0,\infty)$ and obeys $z(t)\to 0$ as $t\to\infty$. 
			\item[(b)] Suppose $g$ is continuous and obeys \eqref{eq.g2}. 
			Then there exists a continuous solution $x$ of \eqref{eq.odepert} on $[0,\infty)$. Moreover, any such continuous solution is uniformly bounded on $[0,\infty)$ and obeys $x(t)\to 0$ as $t\to\infty$. 
		\end{itemize} 
	\end{theorem}   
	A proof of these results can be found in Chapter 3 of \cite{tahanithesis}.
	
	One potential advantage in studying \eqref{eq.z} rather than directly attacking the original ODE is that pointwise conditions on $g$ may no longer be needed to get decay properties, contingent on $\Gamma$ being well--defined. On the other hand, bringing the perturbation $\Gamma$ into the argument of $f$ in \eqref{eq.z} is what motivates the introduction of the asymptotic preservation property \eqref{eq.fasypres}. In a sense, therefore, we are making a trade--off between requiring extra control on $f$, set against a new ability to mollify potentially bad pointwise behaviour of the perturbation $g$ by capturing the effect of the perturbation via $\Gamma$.
	
	We demonstrate for \eqref{eq.z} that when the ``internal'' perturbation $\Gamma$ decays to zero so rapidly that 
	\begin{equation} \label{eq.gammaZ}
		\lim_{t\to\infty} \frac{\Gamma(t)}{F^{-1}(t)}=0,
	\end{equation}
	and the solution of \eqref{eq.z} tends to zero as $t\to\infty$, the asymptotic behaviour of  is preserved in the following sense.
	\begin{theorem} \label{thm.ThZ}
		Suppose that $f\in C(\mathbb{R};\mathbb{R})$ is increasing, odd and obeys \eqref{eq.fglobalstable}, \eqref{eq.fasypres}, and \eqref{eq.fasypresmu}. Suppose $F$ is defined by \eqref{def.F}.   Let $\Gamma$ be continuous and $z$ be the continuous solution of \eqref{eq.z}. 
		If $\Gamma$ obeys \eqref{eq.gammaZ}, then  
		\begin{align*}
			\lim_{t \to \infty} \frac{z(t)}{F^{-1}(t)} =\lambda\in \{-1,0,+1\}.
		\end{align*}
	\end{theorem}
	We now explore how Theorem~\ref{thm.ThZ} can be applied to determine sufficient conditions for certain asymptotic decay in \eqref{eq.odepert}. 
	\subsection{Application of Theorem~\ref{thm.ThZ} to \eqref{eq.odepert}} \label{applyThm1}
	Consider the solution $x$ of \eqref{eq.odepert} which we suppose obeys $x(t)\to 0$ as $t\to\infty$. Introduce the function $u(t)=\int_0^t g(s)\,ds$ and assume that it tends to a finite limit as $t\to\infty$, which we call $u(\infty)$. We are therefore free to define $\Gamma(t)=u(t)-u(\infty)$ for $t\geq 0$. Clearly, $\Gamma$ is continuous and obeys $\Gamma(t)\to 0$ as $t\to\infty$. Of course, $u'(t)=g(t)$. Consider now $z(t)=x(t)-u(t)+u(\infty)=x(t)-\Gamma(t)$ for $t\geq 0$. Then $z$ is in $C^1((0,\infty);\mathbb{R})$ and we have that $z(t)\to 0$ as $t\to\infty$. Then $z(0)=\xi+\int_0^\infty g(s)\,ds=:\xi'$ and 
	\[
	z'(t)=x'(t)-u'(t)=-f(x(t))=-f(z(t)+\Gamma(t)), \quad t\geq 0.
	\]
	Therefore, we see that if $\Gamma(t)=\int_t^\infty g(s)\,ds$ obeys \eqref{eq.gammaZ}, we can apply Theorem~\ref{thm.ThZ} to $z$ to obtain 
	\[
	\lim_{t\to\infty} \frac{z(t)}{F^{-1}(t)}=\lambda\in \{-1,0,1\}. 
	\] 
	Then, as $\Gamma$ obeys \eqref{eq.gammaZ}, and $x=z+\Gamma$, we have
	\[
	\lim_{t\to\infty} \frac{x(t)}{F^{-1}(t)}
	=\lambda+0=\lambda\in \{-1,0,1\}. 
	\]  
	Therefore we have established the following result.
	\begin{theorem}\label{thm.T1ch3}
		Suppose that $f\in C(\mathbb{R};\mathbb{R})$ is increasing, odd and obeys \eqref{eq.fglobalstable}, \eqref{eq.fasypres}, and \eqref{eq.fasypresmu}. Suppose $F$ is defined by \eqref{def.F}.  Let $g$ be a continuous function such that \eqref{eq.g2} holds and let $\Gamma$ be defined by \eqref{eq:Gamma}. 
		If $\Gamma$ obeys \eqref{eq.gammaZ}, then the continuous solution of \eqref{eq.odepert}  obeys 
		\begin{align*}
			\lim_{t \to \infty} \frac{x(t)}{F^{-1}(t)} =\lambda\in \{-1,0,+1\}.
		\end{align*}
	\end{theorem}
	In the case that the perturbation $\Gamma$ is of the same order of magnitude as the solution of the unperturbed equation \eqref{eq.ode}, namely
	\begin{equation} \label{eq.GammaOFinv}
		\Gamma(t)=O(F^{-1}(t)), \quad t\to\infty,
	\end{equation}
	the following result enables us to establish an upper bound on the rate of decay of $z$. 
	\begin{theorem} \label{ThZbdd}
		Suppose that $f\in C(\mathbb{R};\mathbb{R})$ is increasing, odd and obeys \eqref{eq.fglobalstable}, \eqref{eq.fasypres}, and \eqref{eq.fasypresmu}. Suppose $F$ is defined by \eqref{def.F}.   Let $\Gamma$ be continuous and $z$ be the continuous solution of \eqref{eq.z}. 
		If $\Gamma$ obeys \eqref{eq.GammaOFinv}, then  
		\begin{align*}
			z(t)=O(F^{-1}(t)), \quad t\to\infty.
		\end{align*}
	\end{theorem}
	Since $x=z+\Gamma$, this immediately translates to the give a bound on the decay rate of solutions of \eqref{eq.odepert}.
	\begin{theorem}\label{Txbdd}
		Suppose that $f\in C(\mathbb{R};\mathbb{R})$ is increasing, odd and obeys \eqref{eq.fglobalstable}, \eqref{eq.fasypres}, and \eqref{eq.fasypresmu}. Suppose $F$ is defined by \eqref{def.F}.  Let $g$ be a continuous function such that \eqref{eq.g2} holds and let $\Gamma$ be defined by \eqref{eq:Gamma}. 
		If $\Gamma$ obeys \eqref{eq.GammaOFinv}, then the continuous solution of \eqref{eq.odepert}  obeys 
		\begin{align*}
			x(t)=O(F^{-1}(t)), \quad t\to\infty. 
		\end{align*}
	\end{theorem}
	
	\section{Proof of Results for Internally Perturbed Equation}
	\subsection{Proof of Theorem~\ref{thm.ThZ}} 
	Theorem~\ref{thm.ThZ} is the key underlying result of this paper, and its proof relies on careful asymptotic analysis, and a number of interlinked intermediate results. Accordingly, we take a moment to summarise the structure of the proof. 
	
	The proof involves a successive ``ratcheting'' of the asymptotic results: the last two steps of the proof in particular rely on constructing functions 
	that are guaranteed to majorise and minorise $z$ for sufficiently large $t$. Both majorisation and minorisation results rely on a comparison principle based on the differential equation for $t\mapsto z(t)$. Good initial conditions for these comparison arguments (for large values of the time argument) are generated by developing sharp a priori bounds on the liminf and limsup of $t\mapsto |z(t)|/F^{-1}(t)$. In particular, we prove the result through the following steps:
	\begin{itemize}
		\item[STEP 1a:] $\liminf_{t\to\infty} |z(t)|/F^{-1}(t)<+\infty$;
		\item[STEP 1b:] $\liminf_{t\to\infty} |z(t)|/F^{-1}(t)\leq 1$;
		\item[STEP 2:] $\limsup_{t\to\infty} |z(t)|/F^{-1}(t)\leq 1$;
		\item[STEP 3:] If $\limsup_{t\to\infty} |z(t)|/F^{-1}(t)>0$, then $\limsup_{t\to\infty} |z(t)|/F^{-1}(t)=1$; therefore $\limsup_{t\to\infty} |z(t)|/F^{-1}(t)=0$ or $1$;
		\item[STEP 4:] If $\limsup_{t\to\infty} |z(t)|/F^{-1}(t)=1$, then $\liminf_{t\to\infty} |z(t)|/F^{-1}(t)=1$.
	\end{itemize}
	Of course, it can be seen that STEPs 3 and 4 together imply that the limit of $t\mapsto z(t)/F^{-1}(t)$ must exist and be 0, -1 or 1, which is the desired result.
	
	In each of the following results, we assume that $f$ is increasing, odd, in $C((0,\infty);(0,\infty))$ and obeys \eqref{eq.fasypres} and \eqref{eq.fasypresmu}; as a consequence $F$ and $F^{-1}$ are both asymptotic preserving. We do not restate these hypotheses on $f$ and $F$ in the forthcoming results in this section. Likewise, $\Gamma$ is tacitly assumed to obey \eqref{eq.gammaZ} in this subsection (in the next subsection, when we prove Theorem~\ref{thm.T1ch3} we will assume $\Gamma$ obeys \eqref{eq.GammaOFinv}).
	
	In this lemma, we prove STEP 1a and STEP 1b of the above programme.  
	\begin{lemma} \label{lemma.step1zsimFinv}
		Let $z$ be a continuous solution of \eqref{eq.z}. Then
		\begin{itemize}
			\item[(i)] 
			\[
			\liminf_{t\to\infty} \frac{|z(t)|}{F^{-1}(t)}<+\infty.
			\] 
			\item[(ii)] 
			\[
			\liminf_{t\to\infty} \frac{|z(t)|}{F^{-1}(t)}\leq 1.
			\] 
		\end{itemize}	
	\end{lemma}
	\begin{proof}
		Part (i). Suppose to the contrary that $\liminf_{t\to\infty} |z(t)|/F^{-1}(t)=+\infty$. Then either $z(t)>0$ for all $t$ sufficiently large or $z(t)<0$ for all $t$ sufficiently large. Moreover, as $\Gamma$ obeys \eqref{eq.gammaZ}, we have that $\Gamma(t)=o(|z(t)|)$ as $t\to\infty$. Since $f$ is asymptotic preserving, it therefore follows that $f(z(t)+\Gamma(t))\sim f(z(t))$ as $t\to\infty$ (since $f$ is odd, it does not matter which side of zero the argument of $f$ is). Hence from \eqref{eq.z} we get
		\begin{equation} \label{eq.zprfzto-1}
			\lim_{t\to\infty} \frac{z'(t)}{f(z(t))}=-1.
		\end{equation}
		
		In the case that $z(t)>0$ for all $t$ sufficiently large, by asymptotic integration, we obtain 
		\[
		\lim_{t\to\infty} \frac{F(z(t))}{t}=1.
		\]
		Since $F^{-1}$ is asymptotic preserving, it follows that $z(t)\sim F^{-1}(t)$ as $t\to\infty$, but this contradicts the supposition that $z(t)/F^{-1}(t)\to \infty$ as $t\to\infty$.
		
		In the case when $z(t)<0$ for all $t$ sufficiently large, let $z_-(t)=-z(t)$. Then $z_-(t)>0$ for all $t$ sufficiently large, and by the fact that $f$ is odd, we get 
		\[
		\frac{z_-'(t)}{f(z_-(t))}=\frac{-z'(t)}{f(-z(t))}\to -1, \quad t\to\infty,
		\]   
		using \eqref{eq.zprfzto-1}. Asymptotic integration gives once again 
		$F(z_-(t))\sim t$ as $t\to\infty$, and the asymptotic preserving property of $F^{-1}$ yields $z_-(t)\sim F^{-1}(t)$ as $t\to\infty$. Hence $z(t)\sim -F^{-1}(t)$ as $t\to\infty$, which contradicts the supposition that 
		$z(t)/F^{-1}(t)\to -\infty$ as $t\to\infty$.
		
		Hence, both cases possible when $\liminf_{t\to\infty} |z(t)|/F^{-1}(t)=+\infty$ lead to contradictions, and so the liminf must be finite; this completes the proof of part (i).
		
		Part (ii). Assume by way of contradiction that  $\liminf_{t\to\infty} |z(t)|/F^{-1}(t)=:\lambda>1$. Once again this implies that either $z(t)>0$ for all $t$ sufficiently large or $z(t)<0$ for all $t$ sufficiently large. Moreover, as $\Gamma$ obeys \eqref{eq.gammaZ}, we have that $\Gamma(t)=o(|z(t)|)$ as $t\to\infty$.
		Hence we are in the same situation as in the proof of part (i), and can show that $z'(t)\sim -f(z(t))$ as $t\to\infty$. From this asymptotic relation, we can show, in a manner identical to part (i), that $|z(t)|\sim F^{-1}(t)$ as $t\to\infty$ by considering the (only possible) cases when $z(t)$ is positive or negative for sufficiently large $t$. But the conclusion $\lim_{t\to\infty} |z(t)|/F^{-1}(t)=1$ is incompatible with the supposition that $\lambda>1$, and so this initial supposition must be false. Hence $\lambda=\liminf_{t\to\infty} |z(t)|/F^{-1}(t)\leq 1$ as desired. 
	\end{proof}
	
	In this lemma we prove STEP 2 of the programme.
	\begin{lemma} \label{lemma.step2zsimFinv}
		Let $z$ be a continuous solution of \eqref{eq.z}. Then
		\[
		\Lambda:=\limsup_{t\to\infty} \frac{|z(t)|}{F^{-1}(t)}
		\] 
		obeys $\Lambda\leq 1$.   	
	\end{lemma}
	\begin{proof}
		Since $f$ obeys \eqref{eq.fasypres}, for every $\epsilon\in (0,1)$ sufficiently small, there exists $x_1(\epsilon)<1$ such that 
		\begin{equation}   \label{eq.epsfasypres}
			\frac{f((1-\epsilon)x)}{f(x)}> \underline{\Phi}_f(\epsilon)(1-\epsilon), \quad x<x_1(\epsilon).
		\end{equation}
		Now, choose 
		\[
		\lambda_\epsilon=\min(\underline{\Phi}_f(\epsilon)(1-\epsilon),1/(1+\epsilon)).
		\]
		Notice that $\lambda_\epsilon<1$ for all $\epsilon\in (0,1)$ sufficiently small, and that $\lambda_\epsilon\to 1$ as $\epsilon\to 0^+$. 
		
		Let $T_1(\epsilon)>0$ be so small so that 
		\begin{equation} \label{eq.T1}
			F^{-1}(\lambda_\epsilon t)<x_1(\epsilon), \quad t\geq T_1(\epsilon).
		\end{equation}
		Since $\Gamma(t)=o(F^{-1}(t))$ as $t\to\infty$, and $F^{-1}$ is asymptotic preserving, for every $\epsilon\in (0,1)$ sufficiently small, there exists $T_2(\epsilon)>0$ such that
		\begin{equation} \label{eq.Gammabound}
			|\Gamma(t)|<\epsilon F^{-1}(\lambda_\epsilon t), \quad t\geq T_2(\epsilon).
		\end{equation}
		
		Since $F^{-1}$ obeys \eqref{eq.Finvasplimsup}, 
		%$\limsup_{t\to\infty} F^{-1}((1+\epsilon)t)/F^{-1}(t)=\overline{\Phi}_F(\epsilon)$, 
		for every 
		$\epsilon>0$ and $\eta\in (0,1)$ there exists $\tilde{T}_3(\eta,\epsilon)>0$ such that 
		\[
		\frac{F^{-1}((1+\epsilon)t)}{F^{-1}(t)}<\overline{\Phi}_F(\epsilon)(1+\eta), \quad t\geq \tilde{T}_3(\eta,\epsilon).
		\]
		Therefore, setting $\tilde{T}_4(\eta,\epsilon)=(1+\epsilon) \tilde{T}_3(\eta,\epsilon)$, we have 
		\[
		\frac{F^{-1}(\frac{1}{1+\epsilon}t)}{F^{-1}(t)}>\frac{1}{\overline{\Phi}_F(\epsilon)}\frac{1}{1+\eta}, \quad t\geq \tilde{T}_4(\eta,\epsilon).
		\]
		Note that $\overline{\Phi}_F(\epsilon)<1$ by \eqref{eq.Finvasplimsup}, so set 
		\[
		\eta=\eta(\epsilon)=\min\left(\frac{1}{2}\left(\frac{1}{\overline{\Phi}_F(\epsilon)}-1\right),\frac{\epsilon}{2} \right),
		\]
		and let $T_4(\epsilon)=\tilde{T}_4(\eta(\epsilon),\epsilon)$. Since   $\lambda_\epsilon\leq 1/(1+\epsilon)$, we have for $t\geq T_4(\epsilon)$  that
		\begin{equation} \label{eq.boundforic}
			\frac{F^{-1}(\lambda_\epsilon t)}{F^{-1}(t)}\geq 
			\frac{F^{-1}(\frac{1}{1+\epsilon}t)}{F^{-1}(t)}>
			\frac{1}{\overline{\Phi}_F(\epsilon)}\frac{1}{1+\eta(\epsilon)}=: \pi(\epsilon).
		\end{equation}
		The construction of $\eta(\epsilon)$ ensures that $\pi(\epsilon)>1$. However, since $\overline{\Phi}_F(\epsilon)\to 1$ as $\epsilon\to 0^+$, we have that $\eta(\epsilon)\to 0$ as $\epsilon\to 0$, so $\pi(\epsilon)\to 1$ as $\epsilon\to 0^+$.  
		
		Next, we have $\liminf_{t\to\infty} |z(t)|/F^{-1}(t)\leq 1$. Therefore, as $\pi(\epsilon)>1$ and $\pi(\epsilon)\to 1$ as $\epsilon\to 0^+$, for every $\epsilon>0$ sufficiently small, there exists a sequence $(t_n)$ with $t_n\nearrow\infty$ such that 
		\begin{equation} \label{eq.zliminfseq}
			|z(t_n)|<\pi(\epsilon) F^{-1}(t_n).
		\end{equation}
		Next, let $T(\epsilon)=\min\{t_n: t_n>\max(T_1,T_2,T_4)\}$ and define 
		\begin{equation} \label{eq.icforlimsup}
			z_+(t)=F^{-1}(\lambda_\epsilon t), \quad t\geq T.
		\end{equation}
		Then, since $T=t_n$ for some $n$, we have by  \eqref{eq.zliminfseq} and 
		\eqref{eq.boundforic} that 
		\[
		|z(T)|<\pi(\epsilon)F^{-1}(T)<F^{-1}(\lambda_\epsilon T)=z_+(T).
		\]
		For $t\geq T$, we have by \eqref{eq.Gammabound} and the definition of $z_+$
		\begin{align*}
			f(z_+(t)+\Gamma(t))>f(z_+(t)-\epsilon F^{-1}(\lambda_\epsilon t))
			=f((1-\epsilon)z_+(t)),
		\end{align*}
		so by \eqref{eq.T1} and \eqref{eq.epsfasypres}, we have 
		\[
		f(z_+(t)+\Gamma(t))>f((1-\epsilon)z_+(t))>\underline{\Phi}_f(\epsilon)(1-\epsilon)f(z_+(t))\geq \lambda_\epsilon f(z_+(t)),
		\]
		the last inequality coming from the definition of $\lambda_\epsilon$. 
		Now, since $F(z_+(t))=\lambda_\epsilon t$, we have that $z_+'(t)=-\lambda_\epsilon f(z_+(t))$. Hence 
		\[
		z_+'(t)+f(z_+(t)+\Gamma(t))> -\lambda_\epsilon f(z_+(t)) +\lambda_\epsilon f(z_+(t))=0.
		\]
		Hence $z_+'(t)>-f(z_+(t)+\Gamma(t))$ for all $t\geq T(\epsilon)$. Using this, $|z(T)|<z_+(T)$ and \eqref{eq.z}, we see that $z(t)< z_+(t)$ for all $t\geq T$. Therefore
		\[
		\frac{z(t)}{F^{-1}(t)}\leq \frac{F^{-1}(\lambda_\epsilon t)}{F^{-1}(t)}, \quad t\geq T.
		\] 
		Since $\lambda_\epsilon<1$ and $\lambda_\epsilon\to 1$ as $\epsilon\to 0^+$, we have that 
		\begin{align*}
			\limsup_{t\to\infty}\frac{z(t)}{F^{-1}(t)}&\leq \limsup_{t\to\infty} \frac{F^{-1}(\lambda_\epsilon t)}{F^{-1}(t)}\\
			&=\frac{1}{\liminf_{t\to\infty} \frac{F^{-1}(t)}{F^{-1}(\lambda_\epsilon t)}}\\
			&=\frac{1}{\liminf_{t\to\infty} \frac{F^{-1}(\lambda_\epsilon^{-1} t)}{F^{-1}(t)}}
			=\frac{1}{\underline{\Phi}_F(1/\lambda_\epsilon-1)},
		\end{align*}
		where we used \eqref{eq.Finvaspliminf} at the last step. By \eqref{eq.Finvaspliminf}, $\underline{\Phi}_F(1/\lambda_\epsilon-1)\to 1$ 
		as $\epsilon\to 0^+$, so we have 
		\begin{equation} \label{eq.limsupzlt1}
			\limsup_{t\to\infty}\frac{z(t)}{F^{-1}(t)}\leq 1.
		\end{equation}
		
		Now, consider $z_-(t)=-z_+(t)$ for $t\geq T$. By \eqref{eq.icforlimsup}, we have $z(T)>z_-(T)$. Also, for $t\geq T$, we get 
		\begin{align*}
			f(z_-(t)+\Gamma(t))&<f(z_-(t)+\epsilon F^{-1}(\lambda_\epsilon t))\\
			&=f(-z_+(t)+\epsilon F^{-1}(\lambda_\epsilon t))\\
			&=f(-z_+(t)+\epsilon z_+(t))=-f((1-\epsilon)z_+(t)).
		\end{align*}
		Therefore
		\begin{align*}
			z_-'(t)+f(z_-(t)+\Gamma(t))&=-z_+'(t)+f(z_-(t)+\Gamma(t))\\
			&=\lambda_\epsilon f(z_+(t))+f(z_-(t)+\Gamma(t))\\
			&<\lambda_\epsilon f(z_+(t))-f((1-\epsilon)z_+(t))\\
			&\leq \underline{\Phi}_f(\epsilon)(1-\epsilon) f(z_+(t)) - f((1-\epsilon)z_+(t))< 0.
		\end{align*}
		Hence $z_-'(t)<-f(z_-(t)+\Gamma(t)))$ for $t\geq T$, $z(T)>z_-(T)$ and $z$ obeys \eqref{eq.z}. Thus $z(t)>z_-(t)=-z_+(t)$ for $t\geq T$. Therefore, we have 
		\[
		\liminf_{t\to\infty} \frac{z(t)}{F^{-1}(t)}
		\geq 
		\liminf_{t\to\infty} \frac{-F^{-1}(\lambda_\epsilon t)}{F^{-1}(t)}
		=-\frac{1}{\underline{\Phi}_F(1/\lambda_\epsilon-1)},
		\]
		and letting $\epsilon\to 0^+$ yields 
		\[
		\liminf_{t\to\infty} \frac{z(t)}{F^{-1}(t)}\geq -1.
		\]
		Thus, using this limit and \eqref{eq.limsupzlt1} gives $\limsup_{t\to\infty} |z(t)|/F^{-1}(t)\leq 1$, as required. 
	\end{proof}
	
	In this lemma, we prove STEP 3 of the programme.  
	\begin{lemma} \label{lemma.step3zsimFinv}
		Let $z$ be a continuous solution of \eqref{eq.z}. Then
		\[
		\Lambda:=\limsup_{t\to\infty} \frac{|z(t)|}{F^{-1}(t)}<+\infty.
		\] 
		obeys $\Lambda=0$ or $\Lambda\geq 1$.  	
	\end{lemma}
	\begin{proof}
		The quantity $\Lambda$ obey $\Lambda\in [0,1]$ by Lemma~\ref{lemma.step2zsimFinv}. Assume, by way of contradiction to the claim that $\Lambda\in (0,1)$. Then we may take $0<\epsilon<1$ sufficiently small that $\Lambda+2\epsilon/3<1$. By the property \eqref{eq.fasypresmu}, we have for every $\eta>0$ (and $\epsilon$) that there is an $\tilde{x}_1(\eta,\epsilon)<1$ such that 
		\[
		\frac{f(\Lambda+2\epsilon/3)/x}{f(x)}<\bar{\Phi}_f(\Lambda+2\epsilon/3)+\eta, \quad x\leq \tilde{x}_1(\eta,\epsilon).
		\]  
		Now let $\eta=\epsilon/6$ and $x_1(\epsilon):=\tilde{x}_1(\epsilon/6,\epsilon)$; then 
		\[
		\frac{f(\Lambda+2\epsilon/3)/x}{f(x)}<\bar{\Phi}_f(\Lambda+2\epsilon/3)+\frac{\epsilon}{6}, \quad x\leq x_1(\epsilon).
		\] 
		
		Next, the limsup on $z$ and $\Gamma$ imply that for every $\epsilon\in (0,1)$ sufficiently small, there exist $T_1(\epsilon)>0$ and $T_2(\epsilon)>0$ such that 
		\[
		|z(t)|<(\Lambda+\epsilon/3)F^{-1}(t), \quad t\geq T_1(\epsilon); \quad
		|\Gamma(t)|< \frac{\epsilon}{3} F^{-1}(t), \quad t\geq T_2(\epsilon).
		\]
		Since $z'(t)=-f(z(t)+\Gamma(t))$, noting that $z(t)\to 0$ as $t\to\infty$, and that the bounds on $z$ and $\Gamma$ imply the absolute integrabilty of $t\mapsto f(z(t)+\Gamma(t))$, we have that 
		\[
		z(t)=\int_t^\infty f(z(s)+\Gamma(s))\,ds, \quad t\geq 0.
		\] 
		Finally, since $F^{-1}(t)\to 0$ as $t\to\infty$, we have that there exists $T_3(\epsilon)>0$ such that $F^{-1}(t)<x_1(\epsilon)$ for all $t\geq T_3(\epsilon)$. Now set $T(\epsilon)=\max(T_1,T_2,T_3)$, and let $t\geq T(\epsilon)$. Then 
		\begin{align*}
			|z(t)|&=\left|\int_t^\infty f(z(s)+\Gamma(s))\,ds\right|\\
			&\leq \int_t^\infty |f(z(s)+\Gamma(s))|\,ds\\
			&= \int_t^\infty f(|z(s)+\Gamma(s)|)\,ds\\
			&\leq \int_t^\infty f(|z(s)|+|\Gamma(s)|)\,ds\\
			&\leq \int_t^\infty f((\Lambda+2\epsilon/3)F^{-1}(s))\,ds\\
			&=\int_0^{F^{-1}(t)} \frac{f((\Lambda+2\epsilon/3)u)}{f(u)}\,du\\
			&\leq \left\{\bar{\Phi}_f(\Lambda+2\epsilon/3)+\frac{\epsilon}{6}\right\}F^{-1}(t).
		\end{align*}
		Hence for all $\epsilon>0$ so small that $\Lambda+2\epsilon/3<1$, we have 
		\begin{equation} \label{eq.zFinvbdd}
			\limsup_{t\to\infty} \frac{|z(t)|}{F^{-1}(t)}\leq \bar{\Phi}_f(\Lambda+2\epsilon/3)+\frac{\epsilon}{6}.
		\end{equation}
		The rest of the proof involves taking limits in \eqref{eq.zFinvbdd}
		as $\epsilon\to 0^+$. This is delicate, as we have not assumed continuity of $\bar{\Phi}_f$; however, we can connect $\bar{\Phi}_f$ to $\underline{\Phi}_f$, which does obey a limit property as its argument tends to zero. 
		
		Next, we have  
		\begin{align*}
			\bar{\Phi}_f(\Lambda+2\epsilon/3)&=
			\limsup_{x\to 0^+} \frac{f((\Lambda+2\epsilon/3)x)}{f(\Lambda x)}
			\cdot \frac{f(\Lambda x)}{f(x)}\\
			&\leq \limsup_{x\to 0^+} \frac{f((1+2\epsilon/(3\Lambda) )x)}{f(x)} \bar{\Phi}_f(\Lambda)\\
			&= \bar{\Phi}_f(1+2\epsilon/(3\Lambda))\bar{\Phi}_f(\Lambda),
		\end{align*}
		or summarising, we have 
		\begin{equation} \label{eq.Phi1}
			\bar{\Phi}_f(\Lambda+2\epsilon/3)\leq \bar{\Phi}_f(1+2\epsilon/(3\Lambda))\bar{\Phi}_f(\Lambda).
		\end{equation}
		Now, note that \eqref{eq.fasypres} implies that $\underline{\Phi}_f(\epsilon')\to 1$ as $\epsilon'\to 0^+$ where  
		\[
		\underline{\Phi}_f(\epsilon')=\liminf_{x\to 0^+} \frac{f((1-\epsilon')x)}{f(x)}.
		\]
		Hence
		\[
		\frac{1}{\underline{\Phi}_f(\epsilon')} %= \frac{1}{\liminf_{x\to 0^+} \frac{f((1-\epsilon')x)}{f(x)}}
		=\limsup_{x\to 0^+} \frac{f(x)}{f((1-\epsilon')x)}
		=\limsup_{y\to 0^+} \frac{f(y/(1-\epsilon'))}{f(y)}
		=\bar{\Phi}_f\left(\frac{1}{1-\epsilon'}\right).
		\] 
		Now, choose $\epsilon'\in (0,1)$ such that $1+2\epsilon/(3\Lambda)=1/(1-\epsilon')$; this fixes 
		\[
		\epsilon':=1-\frac{1}{1+\frac{2\epsilon}{3\Lambda}},
		\]
		and we see that 
		\begin{equation} \label{eq.Phi2}
			\frac{1}{\underline{\Phi}_f\left(1-\frac{1}{1+\frac{2\epsilon}{3\Lambda}}\right)}=\bar{\Phi}_f\left(1+2\epsilon/(3\Lambda)\right).
		\end{equation}
		Finally, combining \eqref{eq.Phi1}, \eqref{eq.Phi2} and 
		\eqref{eq.zFinvbdd} yields
		\begin{align*}
			\limsup_{t\to\infty} \frac{|z(t)|}{F^{-1}(t)}&\leq \bar{\Phi}_f(\Lambda+2\epsilon/3)+\frac{\epsilon}{6}\\
			&\leq \bar{\Phi}_f(1+2\epsilon/(3\Lambda))\bar{\Phi}_f(\Lambda)  + \frac{\epsilon}{6}\\
			&= \frac{1}{\underline{\Phi}_f\left(1-\frac{1}{1+\frac{2\epsilon}{3\Lambda}}\right)} \cdot \bar{\Phi}_f(\Lambda)  + \frac{\epsilon}{6}.
		\end{align*}
		The argument in $\underline{\Phi}_f$ in the denominator tends to zero as $\epsilon\to 0^+$; thus, by \eqref{eq.fasypres}, the denominator tends to unity as $\epsilon\to 0^+$. Since $\epsilon>0$ is arbitrary, we may take $\epsilon\to 0^+$ across the inequality to get 
		\[
		\limsup_{t\to\infty} \frac{|z(t)|}{F^{-1}(t)} \leq \bar{\Phi}_f(\Lambda).
		\]
		By the definition of $\Lambda\in (0,1)$, this reads $\Lambda\leq   \bar{\Phi}_f(\Lambda)$. But for any $\Lambda\in (0,1)$, \eqref{eq.fasypresmu} implies $\bar{\Phi}_f(\Lambda)<\Lambda$, a contradiction. Hence the original supposition that $\Lambda\in (0,1)$ must be false, and we have either $\Lambda=0$ or $\Lambda\geq 1$, as claimed. 
	\end{proof}
	
	\subsection{Proof of Theorem~\ref{thm.ThZ}} 
	Since 
	\[
	\liminf_{x\to 0^+} \frac{f((1-\epsilon)x)}{f(x)}=\underline{\Phi}_f(\epsilon)
	\]
	we have that 
	\[
	\overline{\Phi}_f(1+\epsilon):=\limsup_{x\to 0^+} \frac{f((1+\epsilon)x)}{f(x)}
	=\frac{1}{\underline{\Phi}_f(\epsilon/(1+\epsilon))}.
	\]
	Thus 
	$\overline{\Phi}_f(1+\epsilon)\to 1$ 
	as $\epsilon\to 0^+$. Thus for every $\eta>0$ sufficiently small there is an $\tilde{x}_1(\eta, \epsilon)>0$ such that 
	\[
	\frac{f((1+\epsilon)x)}{f(x)} < 	\overline{\Phi}_f(1+\epsilon)+\eta, \quad x<\tilde{x}_1(\eta, \epsilon).
	\]
	Choose $\eta=\epsilon$ and let $x_1(\epsilon)=\tilde{x}_1(\epsilon,\epsilon)$. Then 
	\begin{equation} \label{eq.fdd}
		\frac{f((1+\epsilon)x)}{f(x)} < 	\overline{\Phi}_f(1+\epsilon)+\epsilon, \quad x<x_1( \epsilon).
	\end{equation}	
	Next define $\lambda_\epsilon=\max(\overline{\Phi}_f(1+\epsilon)+\epsilon,1+\epsilon)$. Since $\Gamma(t)=o(F^{-1}(t))$ as $t\to\infty$ and $F^{-1}$ is asymptotic preserving, for every $\epsilon\in (0,1)$ sufficiently small, there is a $T_1(\epsilon)>0$ such that $|\Gamma(t)|<\epsilon F^{-1}(\lambda_\epsilon t)$ for $t\geq T_1(\epsilon)$ and a $T_2(\epsilon)>0$ such that $F^{-1}(\lambda_\epsilon t)<x_1(\epsilon)$ for all $t\geq T_2(\epsilon)$. Next recall that 
	\[
	\limsup_{t\to\infty} \frac{F^{-1}((1+\epsilon)t)}{F^{-1}(t)}
	=\overline{\Phi}_F(\epsilon)
	\] 
	where $\overline{\Phi}_F(\epsilon)<1$ and $\overline{\Phi}_F(\epsilon)\to 1$ as $\epsilon\to 0^+$. Thus for all $\eta>0$ (and $\epsilon$ sufficiently small) there is a $\tilde{T}_3(\eta,\epsilon)>0$ such that 
	\[
	\frac{F^{-1}((1+\epsilon)t)}{F^{-1}(t)}
	<\overline{\Phi}_F(\epsilon)+\eta, \quad t\geq \tilde{T}_3(\eta,\epsilon).
	\] 
	Now, choose $\eta=\eta(\epsilon)-(1-\overline{\Phi}_F(\epsilon))/2$. Define $T_3(\epsilon)=\tilde{T}_3(\epsilon,\eta(\epsilon))$, so for all $t\geq T_3(\epsilon)$ we have
	\begin{equation} \label{eq.Finvdot}
		\frac{F^{-1}((1+\epsilon)t)}{F^{-1}(t)}
		<\overline{\Phi}_F(\epsilon)+\eta(\epsilon)=\frac{1+\overline{\Phi}_F(\epsilon)}{2}=:\pi(\epsilon), \quad t\geq T_3(\epsilon).
	\end{equation}	 
	Note, since $\overline{\Phi}_F(\epsilon)<1$ and $\overline{\Phi}_F(\epsilon)\to 1$ as $\epsilon\to 0^+$, we have that $\pi(\epsilon)<1$ and $\pi(\epsilon)\to 1$ as $\epsilon\to 0^+$. 
	
	We have already shown that $\limsup_{t\to\infty} |z(t)|/F^{-1}(t)=0$ or $1$. Putting aside the former case, in which $z(t)=o(F^{-1}(t))$ as $t\to\infty$, we have either 
	\begin{itemize}
		\item[(a)]  For every $\epsilon>0$ sufficiently small, there is $t_n\uparrow \infty$ such that $z(t_n)>\pi(\epsilon)F^{-1}(t_n)$ (such a sequence is guaranteed to exist because $\pi(\epsilon)<1$ and $\pi(\epsilon)\to 1$ as $\epsilon\to 0^+$); or
		\item[(b)] For every $\epsilon>0$ sufficiently small, there is $\tau_n\uparrow \infty$ such that $z(\tau_n)<-\pi(\epsilon)F^{-1}(t_n)$ (such a sequence is guaranteed to exist because $\pi(\epsilon)<1$ and $\pi(\epsilon)\to 1$ as $\epsilon\to 0^+$); or
	\end{itemize}
	In case (a), define 
	\[
	T(\epsilon)=\min\{t_n:t_n>T_4(\epsilon)=\max(T_1(\epsilon),T_2(\epsilon),T_3(\epsilon))\}
	\]
	and also 
	\[
	z_-(t)=F^{-1}(\lambda_\epsilon t), \quad t\geq T(\epsilon).
	\]	 
	Then by the construction of $T$ from $t_n$, \eqref{eq.Finvdot} and the definition of $\lambda_\epsilon$, we have 
	\[
	z(T)>\pi(\epsilon)F^{-1}(T)>F^{-1}((1+\epsilon)T)\geq F^{-1}(\lambda_\epsilon T)=z_-(T).
	\]	
	Next for $t\geq T(\epsilon)$, using the definition of $T_1$ and $T_2$ alongside \eqref{eq.fdd}, and the definition of $\lambda_\epsilon$, we get the estimate 
	\begin{align*}
		f(z_-(t)+\Gamma(t))&< f(z_-(t)+\epsilon F^{-1}(\lambda_\epsilon t))\\
		&=f((1+\epsilon)F^{-1}(\lambda_\epsilon t))\\
		&< (\overline{\Phi}_f(1+\epsilon)+\epsilon)f(F^{-1}(\lambda_\epsilon t))\\
		&\leq \lambda_\epsilon f(z_-(t))=-z_-'(t).
	\end{align*}
	Therefore
	\[
	z_-'(t)<-f(z_-(t)+\Gamma(t)), \quad t\geq T(\epsilon).
	\]
	From this and the fact that $z(T)>z_-(T)$, we have $z(t)>z_-(t)$ for all $t\geq T(\epsilon)$. Therefore
	\begin{equation} \label{eq.zFinvddd}
		\liminf_{t\to\infty} \frac{z(t)}{F^{-1}(t)}\geq 
		\liminf_{t\to\infty} \frac{F^{-1}(\lambda_\epsilon t)}{F^{-1}(t)}
		=\underline{\Phi}_F(\lambda_\epsilon-1).
	\end{equation}
	Since $\lambda_\epsilon\to 1$ as $\epsilon\to 0^+$, we have $\underline{\Phi}_F(\lambda_\epsilon-1)\to 1$ as $\epsilon\to 0^+$. 
	Hence taking the limit as $\epsilon \to 0^+$ in \eqref{eq.zFinvddd}, we get 
	\begin{equation*} 
		\liminf_{t\to\infty} \frac{z(t)}{F^{-1}(t)}\geq 
		%\liminf_{t\to\infty} \frac{F^{-1}(\lambda_\epsilon t)}{F^{-1}(t)}
		1.
	\end{equation*}
	In this case, we have that $z(t)>0$ for all $t>0$ sufficiently large and therefore 
	\[
	\limsup_{t\to\infty} \frac{z(t)}{F^{-1}(t)}=
	\limsup_{t\to\infty} \frac{|z(t)|}{F^{-1}(t)}=1,
	\]
	so combining this with the liminf, we have that $z(t)/F^{-1}(t)\to 1$ as $t\to\infty$ in case (a).
	
	Now, suppose we are in case (b). 
	Define 
	\[
	\tau(\epsilon)=\min\{\tau_n:\tau_n>T_4(\epsilon)\},
	\]
	and define also $z_+(t)=-F^{-1}(\lambda_\epsilon t)$ for $t\geq \tau(\epsilon)$. Now
	\[
	-z(\tau)>\pi(\epsilon)F^{-1}(\tau)>F^{-1}((1+\epsilon)\tau)\geq F^{-1}(\lambda_\epsilon \tau)=-z_+(\tau),
	\]
	or $z_+(\tau)>z(\tau)$. Next for $t\geq \tau$ we estimate
	\begin{align*}
		f(z_+(t)+\Gamma(t))&> f(z_+(t)-\epsilon F^{-1}(\lambda_\epsilon t))\\
		&=f(-(1+\epsilon)F^{-1}(\lambda_\epsilon t))\\
		&=-f((1+\epsilon)F^{-1}(\lambda_\epsilon t)),
	\end{align*}
	so $-f(z_+(t)+\Gamma(t))<f((1+\epsilon)F^{-1}(\lambda_\epsilon t))<\lambda_\epsilon f(-z_+(t))=z_+'(t)$, since $f$ is odd. 
	Hence 
	\[
	z_+'(t)>-f(z_+(t)+\Gamma(t)), \quad t\geq \tau.
	\]
	As a result $z_+(t)>z(t)$ for all $t\geq \tau$. Hence
	\[
	\frac{z(t)}{F^{-1}(t)}<-\frac{F^{-1}(\lambda_\epsilon t)}{F^{-1}(t)}, \quad t\geq \tau.
	\]
	Now, proceeding as in case (a), we get 
	\[
	\limsup_{t\to\infty} \frac{z(t)}{F^{-1}(t)}\leq -1,
	\]
	and we see that $z(t)<0$ for all $t$ sufficiently large. Therefore
	\[
	1=\limsup_{t\to\infty} \frac{|z(t)|}{F^{-1}(t)}=\limsup_{t\to\infty} \frac{-z(t)}{F^{-1}(t)},
	\]
	so $\liminf_{t\to\infty} z(t)/F^{-1}(t)=-1$. Combining this with the limsup above, we have that $z(t)/F^{-1}(t)\to -1$ as $t\to\infty$. 
	
	Thus, we see that if $\limsup_{t\to\infty} |z(t)|/F^{-1}(t)=1$ we must have either $z(t)/F^{-1}(t)\to \pm1$ as $t\to\infty$. The only other possibilty is if $\limsup_{t\to\infty} |z(t)|/F^{-1}(t)=0$, in which case $z(t)/F^{-1}(t)\to 0$ as $t\to\infty$. Therefore, we must have 
	\[
	\lim_{t\to\infty} \frac{z(t)}{F^{-1}(t)}=\lambda\in \{-1,0,+1\},
	\]
	as required.
	
	\section{Proofs of Results for Externally Perturbed Equation}
	In this section, we complete proofs of results for the ordinary differential equation $x'=-f(x)+g$ which characterise certain types of decay rate. 
	
	To a substantial degree, results from which properties of solutions are deduced from the asymptotic order of the forcing term are more or less straightforward corollaries of corresponding results for the internally perturbed equation. 
	
	Converse results, which demonstrate that these size restrictions on the forcing term are essentially optimal, mostly involve writing $g$ or $\Gamma$ in terms of $x$ by a straight rearrangement of the differential equation. Once this is done, the asymptotic estimates on $x$ (or functionals of $x$) derived from sufficiency results are then used to determine the appropriate asymptotic restrictions on $g$ or $\Gamma$.  
	
	\subsection{Proof of Theorem \ref{thm.expertodeiff}}
	The fact that (a) implies (b) is the subject of Theorem \ref{thm.T1ch3}. 
	
	We need to show that (b) implies (a). By hypothesis, we have that 
	$x(t)\sim \lambda F^{-1}(t)$ as $t\to\infty$, where $\lambda=0,\pm 1$. 
	
	We take first the case where $\lambda=1$. Since $f$ obeys \eqref{eq.fasypres}, we have that $f(x(t))\sim f(F^{-1}(t))$ as $t\to\infty$. Therefore, since $(f\circ F^{-1})(t)>0$ for all $t\geq 0$, we have  
	\[
	\int_0^t |f(F^{-1}(s))|\,ds = \int_0^t f(F^{-1}(s))\,ds= 1-F^{-1}(t)
	\] 
	so $f\circ F^{-1}$ is in $L^1(0,\infty)$, and therefore $t\mapsto f(x(t))\in L^1(0,\infty)$. Integrating the differential equation yields
	\begin{equation} \label{eq.inteqn}
		x(t)=\xi-\int_0^t f(x(s))\,ds + \int_0^t g(s)\,ds, 
	\end{equation}
	so because $x(t)\to 0$ as $t\to\infty$, we have that $\lim_{t\to\infty} \int_0^t g(s)\,ds$ is finite. Call this limit $L$. Then we may define 
	\[
	\Gamma(t):=\int_0^t g(s)\,ds - L =: -\int_t^\infty g(s)\,ds, \quad t\geq 0.
	\]
	Taking the limit as $t\to\infty$ in \eqref{eq.inteqn}, we get 
	\begin{equation*} 
		0=\xi-\int_0^\infty f(x(s))\,ds + L. 
	\end{equation*}
	Subtracting this from \eqref{eq.inteqn}, and using the definition of $\Gamma$, we arrive at 
	\begin{equation} \label{eq.Gammaintermsofx}
		\Gamma(t)=x(t)-\int_t^\infty f(x(s))\,ds, \quad t\geq 0.
	\end{equation}
	Since $f(x(t))\sim f(F^{-1}(t))$ as $t\to\infty$,
	\[
	\int_t^\infty f(x(s))\,ds \sim \int_t^\infty f(F^{-1}(s))\,ds = F^{-1}(t), \quad t\to\infty.
	\] 
	Since $x(t)\sim F^{-1}(t)$ as $t\to\infty$, from \eqref{eq.Gammaintermsofx} and the estimate for $\int_t^\infty f(x(s))\,ds$, we get $\Gamma(t)=o(F^{-1}(t))$ as $t\to\infty$, as required. This shows that (b) implies (a) in the case $\lambda=1$.
	
	In the case $\lambda=-1$, since $f$ is odd and obeys \eqref{eq.fasypres}, we have that $x(t)\sim -F^{-1}(t)$ as $t\to\infty$ implies $-x(t)\sim F^{-1}(t)$, so $f(x(t))=-f(-x(t))\sim -f(F^{-1}(t))$ as $t\to\infty$. Thus, $-f(x(t))>0$ for all $t$ sufficiently large, so 
	$f(x(t))\in L^1(0,\infty)$ because $f(F^{-1}(t))\in L^1(0,\infty)$. Integrating the differential equation again yields \eqref{eq.inteqn}, and since $x(t)\to 0$ as $t\to\infty$ and $t\mapsto f(x(t))\in L^1(0,\infty)$, it follows once again that $\int_0^t g(s)\,ds$ tends to a finite limit $L$ as $t\to\infty$, and so $\Gamma$ is well--defined and obeys the relation \eqref{eq.Gammaintermsofx}.
	
	Now, on the righthand side of \eqref{eq.Gammaintermsofx}, by the above considerations, we have that 
	\[
	\int_t^\infty f(x(s))\,ds \sim \int_t^\infty -f(F^{-1}(s))\,ds =-F^{-1}(t), \quad t\to\infty,
	\]    
	while $x(t)\sim -F^{-1}(t)$ as $t\to\infty$. Now, using \eqref{eq.Gammaintermsofx} and the above estimates, we get $\Gamma(t)=o(F^{-1}(t))$ as $t\to\infty$, as before. Thus we have shown that (b) implies (a) in the case $\lambda=-1$.
	
	It remains to deal with the case when $\lambda=0$. In this situation, we have for every $\epsilon\in (0,1)$ that there is $T(\epsilon)>0$ such that 
	$|x(t)|<\epsilon F^{-1}(t)$ for all $t\geq T(\epsilon)$. Therefore using the fact that $f$ is odd and increasing, we have for $t\geq T(\epsilon)$ that 
	\[
	|f(x(t))|=f(|x(t)|)\leq f(\epsilon F^{-1}(t)).
	\]
	Since $f$ is increasing, $|f(x(t))|\leq f( F^{-1}(t))$ for all $t\geq T(\epsilon)$, and since $t\mapsto f(F^{-1}(t))$ is in $L^1(0,\infty)$, it follows that $t\mapsto f(x(t))$ is also in $L^1(0,\infty)$. Since $x(t)\to 0$ as $t\to\infty$, we have from \eqref{eq.inteqn} that $\int_0^t g(s)\,ds$ once again tends to a finite limit (call it $L$), and therefore, once again $\Gamma$ is well--defined and obeys \eqref{eq.Gammaintermsofx}. 
	
	We now estimate more carefully the integral term in \eqref{eq.Gammaintermsofx}. For $t\geq T(\epsilon)$ we have 
	\[
	\frac{|f(x(t))|}{f(F^{-1}(t))}\leq \frac{f(\epsilon F^{-1}(t))}{f(F^{-1}(t))},
	\]
	so for $\epsilon \in (0,1)$, using \eqref{eq.fasypresmu} we have 
	\[
	\limsup_{t\to\infty} \frac{|f(x(t))|}{f(F^{-1}(t))}\leq 
	\limsup_{x\to 0^+} \frac{f(\epsilon x)}{f(x)}=\bar{\Phi}_f(\epsilon)<\epsilon.
	\]
	Since $\epsilon\in (0,1)$ is arbitrary, we have that $f(x(t))=o(f(F^{-1}(t)))$ as $t\to\infty$, and this in turn implies that 
	\[
	\int_{t}^\infty f(x(s))\,ds = o\left(\int_t^\infty f(F^{-1}(s))\,ds\right)=o(F^{-1}(t)), \quad t\to\infty.
	\]
	Since $x(t)=o(F^{-1}(t))$ as $t\to\infty$, both terms on the righthand side of \eqref{eq.Gammaintermsofx} are $o(F^{-1}(t))$, and hence $\Gamma(t)=o(F^{-1}(t))$ as $t\to\infty$, as claimed. Thus (b) implies (a) in the case that $\lambda=0$, proving the implication (b) implies (a) in all cases. Hence (a) and (b) are equivalent, as claimed. 
	
	\subsection{Proof of Theorem \ref{thm.deriv}}
	We show first that (c) implies (d). Note first that (c) implies $\Gamma(t)=o(F^{-1}(t))$ as $t\to\infty$, because
	\[
	\Gamma(t)=-\int_t^\infty g(s)\,ds, \quad F^{-1}(t)=\int_t^\infty f(F^{-1}(s))\,ds.
	\] 
	Therefore, by Theorem~\ref{thm.expertodeiff}, $x(t)\sim \lambda F^{-1}(t)$ as $t\to\infty$ for $\lambda=0, \pm1$. In the case that $\lambda=\pm 1$, since $f$ obeys \eqref{eq.fasypres}, we have in the case $\lambda=1$ the estimate 
	\[
	f(x(t))\sim f(F^{-1}(t)), \quad t\to\infty,
	\]
	while in the case $\lambda=-1$, using first the fact that $f$ is odd, we have 
	\[
	f(x(t))=-f(-x(t))\sim -f(F^{-1}(t)), \quad t\to\infty.
	\]
	Hence, in these cases $x(t)\sim \lambda F^{-1}(t)$ as $t\to\infty$ implies $f(x(t))\sim \lambda f(F^{-1}(t))$ as $t\to\infty$. Thus, since $g(t)=o(F^{-1}(t))$ as $t\to\infty$, we have in the case $\lambda=\pm1$ that 
	\[
	x'(t)=-f(x(t))+g(t)\sim -\lambda f(F^{-1}(t)), \quad t\to\infty,
	\]
	and so (c) implies (d) in the cases where $\lambda\neq 0$. 
	
	To tackle the final case, where $\lambda=0$, we can argue as in the proof of Theorem~\ref{thm.expertodeiff} above that $x(t)=o(F^{-1}(t))$ as $t\to\infty$ implies that $f(x(t))=o(f(F^{-1}(t)))$ as $t\to\infty$. 
	Since $g(t)=o(f(F^{-1}(t)))$ as $t\to\infty$, we have in the case when $\lambda=0$ that 
	\[
	x'(t)=-f(x(t))+g(t)= o(f(F^{-1}(t))), \quad t\to\infty,
	\]
	and so once again $x'(t)\sim -\lambda f(F^{-1}(t))$ as $t\to\infty$. Hence (c) implies (d) in all situations. 
	
	To show (d) implies (c), rearrange the differential equation to make  $g$ its subject, and then divide by $f(F^{-1}(t))$. This gives 
	\[
	\frac{g(t)}{f(F^{-1}(t))} = \frac{x'(t)}{f(F^{-1}(t))}+\frac{f(x(t))}{f(F^{-1}(t))},
	\]  
	so taking limits as $t\to\infty$, and noting, as shown above, that 
	$x(t)\sim \lambda F^{-1}(t)$ as $t\to\infty$ implies $f(x(t))\sim \lambda f(F^{-1}(t)))$ as $t\to\infty$, we see that the first term on the right hand side has limit $-\lambda$, by (d), and the second has limit $\lambda$, by the argument just given. Therefore, the limit of the righthand side overall is zero, so $g(t)=o(f(F^{-1}(t)))$ as $t\to\infty$, proving (c).
	
	\subsection{Asymptotic Characterisation for Positive $g$}
	In the case when $g$ is positive, we can give a very sharp characterisation of the size of $g$ that is permissible, while preserving the rate of decay of the unperturbed equation. One nice feature of these results is that we can impose weaker conditions on $f$. 
	\begin{theorem}
		Let $g$ be non--negative and continuous, and let $f$ be continuous and increasing. Consider
		\[
		x'(t) = -f(x(t)) + g(t), \quad t \geq 0, \quad x(0)=\xi>0
		\]
		If $f'(x)\to 0$ as $x\to 0$ or $f\circ F^{-1}$ is asymptotic preserving, and 
		\begin{equation} \label{eq.mixedcondn}
			\frac{1}{t} \int_0^t \frac{g(s)}{f(F^{-1}(s))} \, ds \;\to\; 0, 
			\quad \text{as } t \to \infty,
		\end{equation}
		then
		\[
		\frac{F(x(t))}{t} \;\to\; 1, \quad t \to \infty.
		\]
	\end{theorem}
	\begin{proof}
		Define $y'(t) = -f(y(t)), \; y(0) = \xi/2$.  
		Then $x(t)\geq  y(t) > 0$ and
		\[
		F(y(t)) - F(\tfrac{\xi}{2}) = t,
		\]
		Since $f$ is increasing, we have $f(x(t)) \geq f(y(t)) > 0$, so 
		\[
		\frac{1}{f(x(t))} \leq  \frac{1}{f(y(t))}.
		\]
		Hence, as $g(t)\geq 0$, we have 
		\[
		\frac{x'(t)}{f(x(t))} 
		= -1 + \frac{g(t)}{f(x(t))}
		\leq  -1 + \frac{g(t)}{f(y(t))}.
		\]
		Integrating yields
		\[
		\int_0^t \frac{x'(s)}{f(x(s))}\, ds
		\;\leq\; -t + \int_0^t \frac{g(s)}{f(y(s))}\, ds.
		\]
		Integrating by substitution on the left hand side and notice that $y(t) = F^{-1}(F(\xi/2) + t)$, we obtain
		\[
		\int_{x(t)}^{\xi} \frac{du}{f(u)} \geq t - \int_{0}^{t} \frac{g(s)}{(f \circ F^{-1})(F(\xi/2)+s)} \, ds,
		\]
		so
		\[
		F(x(t)) - F(\xi) \geq t - \int_{0}^{t} 
		\frac{g(s)}{(f \circ F^{-1})(F(\xi/2)+s)} \, ds, \quad t\geq 0.
		\]
		On the other hand, 
		\[
		F(x(t)) \leq F(y(t)) = F(\xi/2) + t, \quad t\geq 0.
		\]
		Let $\varphi(t) = (f \circ F^{-1})(t)$; if $\varphi$ is asymptotic preserving, we have 
		\[
		(f \circ F^{-1})(F(\xi/2)+t) \sim (f \circ F^{-1})(t), 
		\quad t \to \infty.
		\]
		Now, we check when $f'(x)\to 0$ as $x\to 0$ that $\varphi$ is subexponential. We have 
		\[
		\varphi'(t) = f'\!\left(F^{-1}(t)\right)\cdot (F^{-1})'(t) 
		= f'\!\left(F^{-1}(t)\right)\cdot {-f(F^{-1}(t))}
		= {- f'(F^{-1}(t)).} \, \varphi(t),
		\]
		So $\varphi'(t)/\varphi(t)\to 0$ as $t\to\infty$ provided $f'(x) \to 0$ as $x \to 0^+$, and so 
		\[
		\varphi(F(\xi/2)+t) \sim \varphi(t), \quad t \to \infty,
		\]
		and so 	$(f \circ F^{-1})(F(\xi/2)+t) \sim (f \circ F^{-1})(t)$ , 
		as $t\to\infty$ in this case too. 
		
		Thus, in the cases where either $f'(x)\to 0$ as $x\to\infty$, or when $f\circ F^{-1}$ is asymptotic preserving, for every $\varepsilon \in (0,1)$ there is a $T(\varepsilon)>0$ such that  
		\[
		(1-\varepsilon)\varphi(t) \leq \varphi(F(\xi/2)+t) \leq (1+\varepsilon)\varphi(t),
		\quad  t \geq T(\varepsilon).
		\]
		Thus
		\[
		\frac{1}{(1-\varepsilon)t}\int_{T}^{t} \frac{g(s)}{\varphi(s)} \, ds
		\geq \frac{1}{t}\int_{T}^{t} 
		\frac{g(s)}{\varphi(F(\xi/2)+s)} \, ds
		\geq \frac{1}{(1+\varepsilon)t}\int_{T}^{t} \frac{g(s)}{\varphi(s)} \, ds.
		\]
		By hypothesis, we have 
		\[
		\frac{1}{t}\int_{0}^{t} \frac{g(s)}{\varphi(s)} \, ds \;\to\; 0, 
		\quad t \to \infty,
		\]
		so
		\[
		\frac{1}{t}\int_{0}^{t} 
		\frac{g(s)}{(f \circ F^{-1})(F(\xi/2)+s)} \, ds \;\to\; 0, 
		\quad t \to \infty.
		\]
		Hence, using the estimate 
		\[
		\frac{F(x(t))}{t} > 1 + \frac{F(\xi)}{t} - \frac{1}{t} \int_0^t \frac{g(s)}{(f \circ F^{-1})(s + F(\frac{\xi}{2}))} \, ds,
		\]
		we get $\liminf_{t \to \infty} F(x(t)/t \geq 1$. On the other hand, 
		\[
		\limsup_{t \to \infty} \frac{F(x(t))}{t} \leq \limsup_{t\to\infty} \frac{F(y(t))}{t}=1,
		\]
		so combining these limits gives $F(x(t))/t\to 1$ as $t\to\infty$, as claimed. 
	\end{proof}
	We now show that the smallness condition on $g$ is necessary, provided $f\circ F^{-1}$ is asymptotic preserving.	
	\begin{theorem}
		Suppose that $g$ is positive and continuous, $f\circ F^{-1}$ is asymptotic preserving, and $F(x(t))/t\to 1$ as $t\to\infty$. Then $g$ obeys  \eqref{eq.mixedcondn}.
	\end{theorem}	
	\begin{proof} 
		If $F(x(t)) \sim t$, then $x(t)>0$ for all $t$ sufficiently large ($t>T$, say). Since $f \circ F^{-1}$ is asymptotic preserving, 
		we have that 
		\[
		f(x(t))=(f\circ F^{-1})(F(x(t)))\sim (f\circ F^{-1})(t), \quad t\to\infty.
		\]
		Thus, for all $t>T$ we have 
		\[
		\frac{x'(t)}{f(x(t))} +1 = \frac{g(t)}{f(x(t))}.
		\]
		Now, integrating yields 
		\[
		\int_T^t \frac{g(s)}{f(x(s))} \, ds = t - T + \int_T^t \frac{x'(s)}{f(x(s))} \, ds
		= t - T - [F(x(t)) - F(x(T))],
		\]
		so, since $F(x(t))/t\to 1$ as $t\to\infty$, we have 		
		\[
		\frac{1}{t} \int_T^t \frac{g(s)}{f(x(s))} \, ds \to 0, \quad t\to\infty. 
		\]
		Next, for every $\epsilon\in (0,1)$ there is a $T_1(\epsilon)>0$ such that for $t \geq T_{1}(\epsilon)$, we have
		\[
		(1-\epsilon)(f \circ F^{-1})(t) 
		< f(x(t)) 
		< (1+\epsilon)(f \circ F^{-1})(t).
		\]
		Hence
		\[
		\frac{g(t)}{(1-\epsilon)(f \circ f^{-1})(t)} 
		\geq \frac{g(t)}{f(x(t))} 
		\geq  \frac{1}{1+\epsilon} \, \frac{g(t)}{(f \circ f^{-1})(t)},
		\qquad t > T_{1}(\epsilon),
		\]
		since $g$ is non--negative. Hence			
		\[
		0 \leq \frac{1}{t} \int_{T_1(\epsilon)}^{t} \frac{g(s)}{(1+\epsilon)(f \circ F^{-1})(s)} \, ds
		\leq \frac{1}{t} \int_{T(\epsilon)}^{t} \frac{g(s)}{f(x(s))} \, ds \longrightarrow 0,
		\quad \text{as } t \to \infty.
		\]
		Hence \eqref{eq.mixedcondn} holds, as claimed. 
	\end{proof}
	We may combine the last two results to give a characterisation of when the asymptotic behaviour of the unperturbed equation is preserved in the case of non--negative $g$.
	\begin{theorem}
		Suppose that $f$ is continuous, increasing and that $f\circ F^{-1}$ is asymptotic preserving. Let $g$ be non--negative and continuous. Then the following are equivalent:
		\begin{itemize}
			\item[(A)] $$\lim_{t\to\infty} \frac{1}{t}\int_{0}^t \frac{g(s)}{(f\circ F^{-1})(s)}\,ds \to 0, \quad t\to\infty;$$
			\item[(B)] $$\lim_{t\to\infty} \frac{F(x(t))}{t}=1.$$
		\end{itemize}
	\end{theorem}		
	%	\textbf{Theorem:} 
	%	
	%	If $\varphi$ is asymptotic preserving and monotone at $+\infty$, it is also subexponential.
	%				
	%				\medskip
	%				
	%				\begin{proof} Take up $\varphi$ is decreasing, i.e.,
		%				\[
		%				\varphi(t(1+\varepsilon)) < \varphi(t).
		%				\]
		%				
		%				Let $c > 0$. For $t < t + c < t + \varepsilon t$, provided $t > \tfrac{c}{\varepsilon} = T_{1}(\varepsilon)$.
		%				
		%				Then,
		%				\[
		%				\varphi(t) > \varphi(t+\varepsilon t) > \varphi\bigl(t(1+\varepsilon)\bigr), \qquad t \geq T_{1}(\varepsilon).
		%				\]
		%				
		%				So,
		%				\[
		%				1 > \frac{\varphi(t+\varepsilon t)}{\varphi(t)} > \frac{\varphi(t(1+\varepsilon))}{\varphi(t)}, \qquad t \geq T_{1}(\varepsilon).
		%				\]
		%				
		%				Hence,
		%				\[
		%				1 \geq \limsup_{t \to \infty} \frac{\varphi(t+\varepsilon t)}{\varphi(t)}.
		%				\]
		%				Likewise,
		%				\[
		%				\liminf_{t \to \infty} \frac{\varphi(t+c)}{\varphi(t)} \geq \liminf_{t \to \infty} \frac{\varphi(t(1+\varepsilon))}{\varphi(t)}
		%				= \underline{\Phi}\,{\varphi}(\varepsilon)
		%				\]
		%				
		%				Now, let $\varepsilon \to 0^+$. So,
		%				\[
		%				\lim_{\varepsilon \to 0^+} \underline{\Phi}\,{\varphi}(\varepsilon) \longrightarrow 1
		%				\]
		%				
		%				So \quad $\displaystyle \lim_{t \to \infty} \frac{\varphi(t+c)}{\varphi(t)} \geq 1.$
		%				
		%				Hence, for any $c > 0$,
		%				\[
		%				\lim_{t \to \infty} \frac{\varphi(t+c)}{\varphi(t)} = 1.
		%				\]
		%					\end{proof}
	\subsection{Characterisation of bounds}
	In this section, we prove Theorem \ref{thm.extpertodebdd} which characterises the situation in which $x(t)=O(F^{-1}(t))$ as $t\to\infty$. 
	
	\begin{theorem}
		Suppose that $f\in C(\mathbb{R};\mathbb{R})$ is increasing, odd and obeys \eqref{eq.fglobalstable}, \eqref{eq.fasypres}, and \eqref{eq.fasypresmu}. Suppose $F$ is defined by \eqref{def.F}.  Let $g$ be a continuous function such that \eqref{eq.g2} holds and let $\Gamma$ be defined by \eqref{eq:Gamma}. 
		If $\Gamma$ obeys \eqref{eq.GammaOFinv}, then the continuous solution of \eqref{eq.odepert}  obeys 
		\begin{align*}
			x(t)=O(F^{-1}(t)), \quad t\to\infty. 
		\end{align*}
	\end{theorem}
	
	%	\[
	%	\limsup_{x \to 0^+} \frac{f(\mu x)}{f(x)} =: L(\mu) < \mu, 
	%	\quad \forall \, \mu \in (0,1).  
	%	\]										
	\begin{proof}
		For definiteness, suppose that $|\Gamma(t)|\leq DF^{-1}(t)$ for all $t\geq 0$. 
		Define $z(t) = x(t) - \Gamma(t)$ for $t\geq 0$, so that $z'(t)=-f(z(t)+\Gamma(t))$ for $t\geq 0$. If we can show that there exists $K'>0$ such that $|z(t)| \leq K' F^{-1}(t)$ for all $t\geq 0$, the result is secured, since 
		\[
		|x(t)| \leq |z(t)| + |\Gamma(t)| \leq (K' + D) F^{-1}(t),						
		\]
		
		\textbf{Step 1:} We show first that $z$ obeys 
		$$\lambda := \liminf_{t \to \infty}\frac{|z(t)|}{F^{-1}(t)} < \infty.$$
		Suppose not: then 
		$|z(t)|/F^{-1}(t) \to \infty$ as $t\to\infty$. 
		Hence						
		\[
		\frac{\Gamma(t)}{|z(t)|} 						= \frac{\Gamma(t)}{F^{-1}(t)} \cdot \frac{F^{-1}(t)}{|z(t)|}
		\to 0 \quad \text{as } t \to \infty,
		\]
		since $\Gamma(t)/F^{-1}(t)$ is in $[-D,D]$. 					
		Therefore, as $f$ is asymptotic preserving, we have  
		\[						
		f(z(t) + \Gamma(t))= f\!\left(z(t)\left(1 + \tfrac{\Gamma(t)}{z(t)}\right)\right) 
		\sim f(z(t)), \quad t \to \infty.
		\]						
		Thus $z'(t) = - f(z(t) + \Gamma(t)) \sim - f(z(t))$ as $t\to\infty$. By our supposition, we either have $z(t) > 0$ for all sufficiently large $t$, or $z(t) < 0$ for all sufficiently large $t$. In the first case, we have
		\[
		\lim_{t \to \infty} \frac{z'(t)}{f(z(t))} = -1.
		\]
		Thus for every $\varepsilon \in (0,1)$ there exists $T_{1}(\varepsilon) > 0$ such that
		\[
		(-1 - \varepsilon) < \frac{z'(t)}{f(z(t))} < (-1 + \varepsilon), 
		\quad  t \geq T_{1}.
		\]
		So, for every $\varepsilon \in (0,1)$, we have 
		\[
		(-1 - \varepsilon)(t - T_{1}) 
		< \int_{T_{1}(\varepsilon)}^{t} \frac{z'(s)}{f(z(s))}\, ds 
		< (-1 + \varepsilon)(t - T_{1}).
		\]
		Performing a substitution, we get 
		\[
		-(1+\varepsilon)(t - T_{1}) 
		< \int_{z(T_{1})}^{z(t)} \frac{du}{f(u)} 
		< -(1-\varepsilon)(t - T_{1}), \quad t\geq T_1,
		\]
		which yields
		\[
		(1+\varepsilon)(t - T_{1}) 
		> \int_{z(t)}^{z(T_{1})} \frac{du}{f(u)} 
		> (1-\varepsilon)(t - T_{1}).
		\]
		Therefore, we get
		\[
		1 + \varepsilon \;\geq\; \limsup_{t \to \infty} \frac{F(z(t))}{t} 
		\;\geq\; \liminf_{t \to \infty} \frac{F(z(t))}{t} \;\geq\; 1 - \varepsilon.
		\]
		Let $\varepsilon \to 0$; then $F(z(t)) \sim t$ as $t \to \infty$.
		But since $f$ obeys \eqref{eq.fasypresmu}, $F^{-1}$ is asymptotic preserving, so 
		\[
		z(t) \sim F^{-1}(t), \quad t\to\infty,
		\]
		which contradicts the supposition that $z(t)/F^{-1}(t)\to + \infty$ as $t\to\infty$; this rules out $z(t)/F^{-1}(t)\to\infty$ as $t\to\infty$. 
		
		Likewise, in order to rule out the possibility that $z(t)/F^{-1}(t)\to -\infty$ as $t\to\infty$, we assume by way of contradiction that it holds. In this case, $z(t) < 0$ for all sufficiently large $t$, and we introduce the notation $z\_(t) := -z(t)$. Hence $z\_(t) > 0$ for all sufficiently large $t$, and $z'\_(t) = - z'(t)=f(z(t)+\Gamma(t))$. 
		Since $f$ is odd, we have 
		\[
		z'\_(t) =f(z(t)+\Gamma(t))=-f(-z(t)-\Gamma(t))=-f(z\_(t)-\Gamma(t))
		\] 
		Arguing as above, we have $\Gamma(t)/z\_(t)\to 0$ as $t\to\infty$, so since $f$ is asymptotic preserving we have  
		\[
		z'\_(t) \sim -f(z\_(t)), \quad t\to\infty.
		\]
		So by the same argument as above, we have that 
		\[
		F(z\_(t)) \sim t, \quad t \to \infty,
		\]
		and since $F^{-1}$ is asymptotic preserving, this implies 
		Thus $z\_(t) \sim F^{-1}(t)$ as $t \to \infty$. 
		Hence, $z(t) \sim -F^{-1}(t)$ as $t \to \infty$, which contradicts the supposition that $z(t)/F^{-1}(t)\to -\infty$. This completes the proof of Step 1.
		
		%	\textbf{}		
		%		\textbf{Proof.} Suppose not, so $|z(t)| \gg F^{-1}(t)$ as $t \to \infty$.
		%		
		%		Then
		%		\[
		%		\frac{|\Gamma(t)|}{|z(t)|} \to 0, \quad t \to \infty \quad \text{by (.)}.
		%		\]
		%		
		%		Thus, as $z(t) \neq 0$ for all sufficiently large $t$, we have
		%		\[
		%		\frac{z'(t)}{-f(z(t))} 
		%		= \frac{f(z(t) + \Gamma(t))}{f(z(t))} 
		%		= f\!\left(z(t)\left(1 + \frac{\Gamma(t)}{z(t)}\right)\right)\!\Big/f(z(t)) \to 1,
		%		\]
		%		as $t \to \infty$, using the fact that $f$ is AP and odd.
		%		
		%		\medskip
		%		So, irrespective of the case $z(t) > 0$ or $z(t) < 0$ separately, we get
		%		\[
		%		\frac{F(z(t))}{t} \to 1 \quad \text{as } t \to \infty.
		%		\]
		%		
		%		\((+)\) makes $F^{-1}$ AP, so $\frac{|z(t)|}{F^{-1}(t)} \to 1$ as $t \to \infty$,  
		%		which contradicts the assumption. Therefore,
		%		\[
		%		\liminf_{t \to \infty} \frac{|z(t)|}{F^{-1}(t)} < +\infty,
		%		\]
		%		as claimed.
		
		\textbf{Step 2:} We prove that 
		\[
		\limsup_{t \to \infty} \frac{|z(t)|}{F^{-1}(t)} < +\infty.
		\]
		
		By Step 1, there is $\lambda_1 > 0$ and $t_n \to \infty$ such that
		$|z(t_n)| < \lambda_1 F^{-1}(t_n)$. We also have that there exist $D>0$, $T_1>0$ such that
		$|\Gamma(t)| < D F^{-1}(t)$ for all $t \geq T_1$. 
		By Lemma~\ref{lemma.philogest}, there exists $c''>0$ and $\mu_1<1/2$ such that for all $\mu<\mu_1$ we have 
		$$\overline{\Phi}_f(\mu)\leq c''\mu/\log(1/\mu).$$ 
		Thus,
		we have $\overline{\Phi}_f(\mu) = o(\mu)$ as $\mu \to 0^+$.
		Hence, there is a $\mu_2 = \mu_2(D)$ such that			
		\[
		\overline{\Phi}_f(\mu) < \mu \cdot \frac{1 + D (1-\mu)}{(1+\mu D)(2+D)}, 
		\quad \mu < \mu_2(D) < \tfrac{1}{2}			
		\]	
		Put $\varepsilon = \varepsilon(\mu,D) = \frac{\mu}{2+D} < \tfrac{1}{4}$. Then,
		\[
		\overline{\Phi}_f(\mu) < \mu \cdot \frac{1 + D(1-\mu)}{(1+\mu D)(2+D)} 
		= \frac{\mu}{1+\mu D} - \frac{\mu}{2+D}.
		\]	
		Therefore
		\begin{equation} \label{eq.dd}
			\overline{\Phi}_f(\mu)+ \varepsilon(\mu,D) < \frac{\mu}{1+\mu D}, \quad \mu < \mu_2(D).
		\end{equation}
		Also for every $\varepsilon>0$ and $\mu < \tfrac{1}{2}$, there is an 
		$\tilde{x}_1(\mu,\varepsilon) > 0$ such that
		$
		f(\mu x) < \big(	\overline{\Phi}_f(\mu) + \varepsilon \big) f(x)$ for $x < \tilde{x}_1(\mu,\varepsilon)$. Hence 
		\[
		f(x) < \big(	\overline{\Phi}_f(\mu) + \varepsilon\big) f\!\Big(\frac{x}{\mu}\Big),
		\quad x \leq \mu \tilde{x}_1(\mu,\varepsilon).
		\]
		Let $\varepsilon = \varepsilon(\mu,D)$ as above; then with
		$x_1(\mu,D) := \mu \tilde{x}_1(\mu,\varepsilon(\mu,D))$, we have
		\begin{equation} \label{eq.ddd}
			f(x) < \Big(	\overline{\Phi}_f(\mu) + \frac{\mu}{2+D}\Big) f\!\Big(\frac{x}{\mu}\Big), 
			\quad x < x_1(\mu,D), 
		\end{equation}
		and, for any $\mu < \mu_2(D) < \tfrac{1}{2}$, \eqref{eq.dd} and \eqref{eq.ddd} yield 
		\begin{equation} \label{IV}
			f(x) < \frac{\mu}{1 + \mu D} f\!\left(\tfrac{1}{\mu}x\right), 
			\quad x < x_1(\mu,D), \; \mu < \mu_2(D) < \tfrac{1}{2}. 
		\end{equation}
		Now, take  
		\[
		\lambda_+ = D + \frac{1}{\mu}, 
		\quad \text{for some } \mu < \mu_2(D).
		\]
		We immediately have $\lambda_+ > D+2$; additionally, choose $\mu>0$ so small that $\lambda_+ > \lambda_1$. 
		Then
		\begin{equation} \label{III}
			\frac{1}{\lambda_+} = \frac{1}{D+1/\mu} = \frac{\mu}{1+\mu D}. 
		\end{equation}
		Also, there exists $T_2 > 0$ such that $F^{-1}(t) < x_1(\mu,D)$ for all $t \geq T_2$. Let 
		\[
		T_3 = \inf \left\{ t_j > {T_1 \lor T_2}: \; |z(t_j)| < \lambda_1 F^{-1}(t_j) \right\},
		\]
		and define
		\[
		z_+(t) = \lambda_+ F^{-1}(t), \quad t \geq T_3.
		\]
		Then
		\[
		z_+(T_3) = \lambda_+ F^{-1}(T_3) = \lambda_+ F^{-1}(t_{j^*}) > \lambda_1 F^{-1}(t_{j^*}) > |z(t_{j^*})| = |z(T_3)|.
		\]
		Also, for $t \geq T_3$, using the monotonicity of $F$ and the bound on $|\Gamma(t)|\leq DF^{-1}(t)$ which holds for $t\geq T_3\geq T_1$, we have
		\begin{equation} \label{I}
			f(z_+(t) + \Gamma(t)) \geq f\!\big(z_+(t) - D F^{-1}(t)\big)
			=	f\!\big((\lambda_+ - D)F^{-1}(t)\big)= f\!\left(\tfrac{1}{\mu} F^{-1}(t)\right),
		\end{equation}
		where we used the definition of $z_+$ and $\lambda_+$ in the last two identities.
		%Next 
		%Since $f \uparrow$ and $|\Gamma(t)| \leq D F^{-1}(t), \; T_3 \gg T_1$, we obtain
		%\[
		%f\!\big((\lambda_+ - D)F^{-1}(t)\big) = f\!\left(\tfrac{1}{\mu} %F^{-1}(t)\right).
		%\]
		By \eqref{IV}, since $t \geq T_3 \geq T_2$, we have $F^{-1}(t) < x_1(\mu,D)$ for $\mu < \mu_2(D)$, so
		\begin{equation} \label{II}
			f\!\left(\tfrac{1}{\mu}F^{-1}(t)\right) > \frac{1+\mu D}{\mu} f(F^{-1}(t)). 
		\end{equation} 
		Thus, using \eqref{I}, \eqref{II} and \eqref{III}, we have 
		\[
		f\big(z_+(t) + \Gamma(t)\big) > \frac{1+\mu D}{\mu} f(F^{-1}(t)) 
		= \lambda_+ f(F^{-1}(t)), \quad t\geq T_3. 
		\]
		On the other hand, for $t\geq T_3$, $-z_{+}'(t) = \lambda_{+}f(F^{-1}(t))$, so 
		\[
		f\big(z_{+}(t) + \Gamma(t)\big) > - z_{+}'(t), \quad t \geq T_{3}.
		\]
		Since this inequality holds, $z_+(T_3)>z(T_3)$ and $z'(t)=-f(z(t)+\Gamma(t))$ for $t\geq T_3$, we have 
		\[
		z(t) < z_{+}(t)=\lambda_+ F^{-1}(t), \quad t \geq T_{3}.
		\]
		This gives an upper bound on $z$. Now we seek a corresponding lower bound. To this end, define
		\[
		z_{-}(t) = - z_{+}(t), \quad t \geq T_{3}.
		\]
		We established above that $|z(T_3)|<z_+(T_3)$; therefore $-z_+(T_3)<z(T_3)$, and so we have 
		\[
		z_{-}(T_{3}) = - z_{+}(T_{3}) < z(T_{3}),
		\]
		Note that
		$z_{-}'(t) = - z_{+}'(t) = \lambda_{+} f(F^{-1}(t))$ for $t \geq T_{3}$, and for $t \geq T_{3}$, we have 
		\[
		- f\big(z_{-}(t) + \Gamma(t)\big) 
		= - f\big(- z_{+}(t) + \Gamma(t)\big)=f(z_+(t)-\Gamma(t)),
		\]
		where we used the fact that $f$ is odd at the last step. Since $f$ is increasing and $\Gamma(t) \leq D F^{-1}(t)$ for $t \geq T_{3}$, we get
		\[
		- f\big(z_{-}(t) + \Gamma(t)\big) = f(z_+(t)-\Gamma(t)) \geq f(z_+(t)-DF^{-1}(t)).
		\]
		Now revisiting \eqref{I}, \eqref{II}, and \eqref{III} for $t\geq T_3$, we have that $f(z_+(t)-DF^{-1}(t))>\lambda_+f(F^{-1}(t))$.
		%		\[
		%		- f\big(z_{-}(t) + \Gamma(t)\big) 
		%		= - f\big(- z_{+}(t) + \Gamma(t)\big) 
		%		\geq - f\big((- z_{+}(t)) + D F^{-1}(t)\big).
		%		\]
		Therefore, for $t\geq T_3$ we have 
		\[
		- f\big(z_{-}(t) + \Gamma(t)\big)\geq f(z_+(t)-DF^{-1}(t))>\lambda_+f(F^{-1}(t))=z_-'(t).
		\]
		%
		%		That is,
		%		\[
		%		- f\big(z_{-}(t) + \Gamma(t)\big) 
		%		\geq f\big( (\lambda_{+} - D) F^{-1}(t) \big) 
		%		= f\!\left( \tfrac{1}{\mu} F^{-1}(t)\right).
		%		\]
		%		By the definition of $\lambda_{+}$ and from $(II)$, since $t \geq T_{3}$,
		%		\[
		%		f\!\left( \tfrac{1}{\mu} F^{-1}(t)\right) 
		%		> \frac{1+\mu D}{\mu} f(F^{-1}(t)).
		%		\]
		%		Therefore,
		%		\[
		%		- f\big(z_{-}(t) + \Gamma(t)\big) 
		%		> \frac{1+\mu D}{\mu} f(F^{-1}(t)) 
		%		= \lambda_{+} f(F^{-1}(t)) 
		%		= z_{-}'(t).
		%		\]
		Thus, we have that $z_-'(t)<	- f\big(z_{-}(t) + \Gamma(t)\big)$ for $t\geq T_3$, $z_-(T_3)<z(T_3)$ and  $z'(t)=	- f(z(t) + \Gamma(t))$ for $t\geq T_3$. 
		%		\[
		%		- f\big(z_{-}(t) + \Gamma(t)\big) > z_{-}'(t).
		%		\]
		%		So, we have
		%		\[
		%		\begin{cases}
			%			z_{-}'(t) < - f\big(- z_{-}(t) + \Gamma(t)\big), & t \geq T_{3}, \\[6pt]
			%			z_{-}(T_{3}) < z(T_{3}), \\[6pt]
			%			z'(t) = - f\big(z(t) + \Gamma(t)\big), & t \geq T_{3}.
			%		\end{cases}
		%		\]
		Hence
		\[
		z(t) > z_{-}(t) = - z_{+}(t) = - \lambda_{+} F^{-1}(t), 
		\quad t \geq T_{3}.
		\]
		Since we have already shown that $z(t) < \lambda_{+} F^{-1}(t)$ for $t \geq T_{3}$, we have 
		\[
		\frac{|z(t)|}{F^{-1}(t)} < \lambda_{+}, 
		\quad  t \geq T_{3},
		\]	
		and so 
		\[
		\limsup_{t \to \infty} \frac{|z(t)|}{F^{-1}(t)} <  +\infty
		\]
		as claimed.
	\end{proof}
	Now we prove the converse of the last result.
	\begin{theorem} \label{thm.bigOconverse}
		Suppose that $f\in C(\mathbb{R};\mathbb{R})$ is increasing, odd and obeys \eqref{eq.fglobalstable}, \eqref{eq.fasypres}, and \eqref{eq.fasypresmu}. Suppose $F$ is defined by \eqref{def.F}.  Let $g$ be a continuous function. If the continuous solution of \eqref{eq.odepert}  obeys 
		\begin{align*}
			x(t)=O(F^{-1}(t)), \quad t\to\infty,
		\end{align*}	
		then $g$ is such that \eqref{eq.g2} holds and $\Gamma$ be defined by \eqref{eq:Gamma} obeys \eqref{eq.GammaOFinv}.  
		%and define
		%\[
		%F(x) := \int_{x}^{1} \frac{1}{f(u)} \, du, 
		%\qquad x>0 .
		%\]
		%Then $F$ is decreasing and invertible, with derivative
		%\[
		%F'(u) = -\frac{1}{f(u)} .
		%\]
	\end{theorem}	
	\begin{proof}
		By hypothesis we have that there is $L>0$ such that 
		$|x(t)| \le L\,F^{-1}(t)$ for $t\geq 0$. Since $f$ is odd and increasing, we obtain
		\[
		|f(x(t))| = f(|x(t)|) \le f\!\big(L\,F^{-1}(t)\big),
		\qquad t\geq 0.
		\]
		%Using the substitution $u=F^{-1}(s)$
		%	\[
		%	u = F^{-1}(s) \quad \text{so that} \quad ds = F'(u)\,du = \frac{-du}{f(u)},
		%	\] 
		Next we estimate 
		\begin{align*}
			\int_{0}^{t} |f(x(s))|\,ds
			&\le \int_{0}^{t} f\!\big(L\,F^{-1}(s)\big)\,ds \\[6pt]
			%	&= \int_{u=F^{-1}(t)}^{u=F^{-1}(0)}
			%	f(Lu)\,\bigl(-F'(u)\bigr)\,du \\[6pt]
			&= \int_{F^{-1}(t)}^{1} \frac{f(Lu)}{f(u)}\,du \\[6pt]
			&\leq  \int_{0}^{1} \frac{f(Lu)}{f(u)}\,du .
		\end{align*}
		\noindent\textbf{Case $L\le 1$.}  
		Since $f$ is increasing, $f(Lu)\le f(u)$ for $u>0$, so
		\[
		\int_{0}^{t} |f(x(s))|\,ds
		\;\le\; \int_{F^{-1}(t)}^{1} 1\,du
		\;\le\; 1 .
		\]
		Thus the integral is finite, and $t\mapsto f(x(t)) \in L^1(0,\infty)$.
		
		\noindent\textbf{Case $L>1$.}  
		Since $f$ is asymptotic preserving, we have that there is an $\epsilon_0\in (0,1)$ such that 
		\[
		\limsup_{x\to 0^+} \frac{f(1+\epsilon)x)}{f(x)} = \overline{\Phi}_f(1+\epsilon).
		\]
		%	Note that 
		%	\[
		%	\frac{1}{\underline{\Phi}_f(\epsilon/(1+\epsilon))}=\overline{\Phi}_f(1+\epsilon).
		%	\]
		Now, fix $\epsilon\in (0,\epsilon_0)$. Since $L>1$, there exists and integer $N=N(L)\geq 1$ 	such that $(1+\epsilon)^{N-1}< L \leq (1+\epsilon)^N$. Since $f$ is increasing, we have 
		\[
		\frac{f(Lx)}{f(x)}\leq \frac{f((1+\epsilon)^Nx)}{f(x)}
		=\prod_{j=1}^{N} \frac{f((1+\epsilon)^jx)}{f((1+\epsilon)^{j-1}x)}.
		\]
		Therefore 
		\[
		\limsup_{x\to 0^+} 	\frac{f(Lx)}{f(x)} 
		\leq \prod_{j=1}^{N} \limsup_{x\to 0^+}\frac{f((1+\epsilon)^jx)}{f((1+\epsilon)^{j-1}x)}
		\leq \overline{\Phi}_f(1+\epsilon)^N.
		\]
		Now $(N-1)\log(1+\epsilon)<\log L$, so $N<1+\log L/\log(1+\epsilon)$, so therefore
		\[
		\limsup_{x\to 0^+} 	\frac{f(Lx)}{f(x)}\leq \overline{\Phi}_f(1+\epsilon) \cdot \overline{\Phi}_f(1+\epsilon)^{\log L/\log(1+\epsilon)}
		= \overline{\Phi}_f(1+\epsilon) \cdot L^{\log \overline{\Phi}_f(1+\epsilon)/\log(1+\epsilon)}.
		\]
		Now, since we can fix $\epsilon\in (0,\epsilon_0)$, we have that there exist $c>1$ and $a>0$ (both independent of $L>1$) such that 
		\[
		\limsup_{x\to 0^+} 	\frac{f(Lx)}{f(x)}\leq c L^a.
		\] 
		Then, there exists $x_{1}=x_{1}(L)$ such that for all $x<x_{1}(L) \land 1$,
		implies $f(Lx) \le (c+1)\,L^{a}\,f(x)$. 
		Therefore, for all $t \geq 0$,
		\[
		\int_0^t |f(x(s))|\, ds \le 
		\int_0^{x_1(L)} (c+1)L^a \, du + \int_{x_1(L)}^{1} \frac{f(Lu)}{f(u)} \, du < +\infty.
		\]
		Hence, $t \mapsto f(x(t)) \in L^1(0,\infty)$, regardless of the value of $L$. Also, since $x(t) \to 0$ as $t \to \infty$, and 
		\[
		x(t)=x(0)-\int_0^t f(x(s))\,ds + \int_0^t g(s)\,ds,
		\]
		we see that $\int_0^t g(s)\,ds$ tends to a finite limit as $t\to\infty$, and therefore $\Gamma$ is well defined and obeys  
		\[
		\Gamma(t)=x(t)-\int_t^\infty f(x(s))\,ds, \quad t\geq 0.
		\]
		Now, we re--estimate the righthand side, using the same line as before
		\[
		\left| \int_t^\infty f(x(s)) \, ds \right| \le \int_t^\infty |f(x(s))| \, ds
		= \int_t^\infty f(|x(s)|) \, ds \le \int_t^\infty f(L F^{-1}(s)) \, ds
		= \int_0^{F^{-1}(t)} \frac{f(Lu)}{f(u)} \, du.
		\]
		Therefore
		% Substituting $u = F^{-1}(s)$, we get 
		%\quad F(u) = s, \quad ds = \frac{-du}{f(u)},
		%	\]
		%	Thus,
		\[
		\limsup_{t \to \infty} \frac{\left|\int_t^\infty f(x(s))\, ds \right|}{F^{-1}(t)}
		\leq \limsup_{t \to \infty} \frac{1}{F^{-1}(t)} \int_0^{F^{-1}(t)} \frac{f(Lu)}{f(u)} \, du
		=\limsup_{x\to 0^+} \frac{1}{x}\int_0^x \frac{f(Lu)}{f(u)}\,du.	
		\]
		As before, we consider the cases when $L\leq 1$ and $L>1$. In the former case
		we have 
		\[
		\frac{1}{x}\int_0^x \frac{f(Lu)}{f(u)}\,du\leq \frac{1}{x}\int_0^x 1\,du =1,
		\]
		so $L\leq 1$ implies 
		\[
		\limsup_{t \to \infty} \frac{\left|\int_t^\infty f(x(s))\, ds \right|}{F^{-1}(t)}\leq 1.
		\]
		If $L>1$, we have that there exists $x_1=x_1(L)>0$ such that $f(Lx)/f(x)\leq (c+1)L^a$ for all $x<x_1(L)\wedge 1$. Therefore for $x<x_1(L)\wedge 1$ we have
		\[
		\frac{1}{x}\int_0^x \frac{f(Lu)}{f(u)}\,du \leq \frac{1}{x}\int_0^x (c+1)L^a\,du = (1+c)L^a.
		\]
		Thus, if $L>1$ we have 
		\[
		\limsup_{t \to \infty} \frac{\left|\int_t^\infty f(x(s))\, ds \right|}{F^{-1}(t)}\leq (1+c)L^a.
		\]
		Therefore, irrespective of whether $L \in (0,1]$ or $L > 1$,
		\[
		\limsup_{t \to \infty} \frac{\left| \int_t^\infty f(x(s)) \, ds \right|}{F^{-1}(t)} < +\infty.
		\]
		Thus,
		\[
		\frac{|\Gamma(t)|}{F^{-1}(t)} \leq \frac{|x(t)|}{F^{-1}(t)} + \frac{\left| \int_t^\infty f(x(s)) \, ds \right|}{F^{-1}(t)}
		\leq L+\frac{\left| \int_t^\infty f(x(s)) \, ds \right|}{F^{-1}(t)},
		\]
		so since the last term has a finite limsup, we have 
		\[
		\limsup_{t \to \infty} \frac{|\Gamma(t)|}{F^{-1}(t)} \leq L+\max(1,(1+c)L^a).
		\]
		and so $\Gamma(t)=O(F^{-1}(t))$ as $t\to\infty$, as required.
	\end{proof}
	Combining the last two Theorems, we have proven Theorem~\ref{thm.extpertodebdd}.

	\section{Stochastic Differential Equations}
	In this section, we consider the asymptotic behaviour of the stochastic differential equation 
	\begin{equation}  \label{eq.sde7}
		dX(t) = -f\big(X(t)\big)\,dt + \sigma(t)\,dB(t), \quad t\geq 0.
	\end{equation}
	We let $X(0)=\xi$, where $\xi$ is deterministic, and $B$ is a standard one dimensional Brownian motion. We also assume that $f$ is continuous with $f(0)=0$, $xf(x)>0$ for $x\neq 0$, and that $\sigma$ is continuous. If we assume that $f$ is locally Lipschitz continuous, then there is a unique continuous, adapted (to the natural filtration generated by $B$) solution to the differential equation. Moreover, the solution is defined on $[0,\infty)$ a.s. We work on the triple $(\Omega,\mathcal{F},\mathbb{P})$ and use the notation 
	\[
	\mathcal{F}^B(t):=\sigma\{ B(s): 0\leq s\leq t\}, \quad t\geq 0
	\] 
	for the members of the natural filtration $(\mathcal{F}^B(t))_{t\geq 0}$ generated by $B$. 
	
	In the case that $\sigma\in L^2(0,\infty)$ it is known that $X(t)\to 0$ as $t\to\infty$ a.s. In that case, we have by the martingale convergence theorem that 
	\[
	\lim_{t\to\infty} \int_0^t \sigma(s)\,dB(s) \quad \text{ exists and is finite, a.s.}
	\]
	Let $\Omega^\ast$ be the event of all outcomes for which this convergence takes place, and for which the solution $X$ of the SDE is well--defined on $[0,\infty)$. We have that $\Omega^\ast$ is an almost sure event. For each 
	$\omega\in \Omega^\ast$, we may define the function 
	\[
	\Gamma(t,\omega) = -\int_{t}^{\infty} \sigma(s)\, dB(s),\quad t\geq 0
	\]
	i.e., for each $\omega\in \Omega^\ast$ we have 
	\[
	\Gamma(t,\omega) = -\left\{ \lim_{T \to \infty} \int_{0}^{T} \sigma(s)\, dB(s) - \int_{0}^{t} \sigma(s)\, dB(s) \right\}(\omega), \quad t\geq 0.
	\]
	Notice that, for each $\omega\in \Omega^\ast$, we have $t\mapsto \Gamma(t,\omega)$ tends to zero as $t\to\infty$, and that $t\mapsto \Gamma(t,\omega)$ is continuous. Notice moreover that $\Gamma$ cannot be defined on an event of positive probability if $\sigma\not\in L^2(0,\infty)$, again by the martingale convergence theorem. 
	
	Our first result gives necessary and sufficient conditions for the preservation of the asymptotic behaviour of the unperturbed ordinary differential equation. However, the asymptotic condition on the noise term is not directly checkable, so in a subsequent result, we replace it with a deterministic condition on the noise that can be checked directly.
	
	\begin{theorem} \label{thm.Xwithtailmart}
		Suppose that $f\in C(\mathbb{R};\mathbb{R})$ is increasing, odd and obeys \eqref{eq.fglobalstable}, \eqref{eq.fasypres}, and \eqref{eq.fasypresmu}. Suppose $F$ is defined by \eqref{def.F}.  Let $\sigma$ be a continuous function, and let $X$ be the unique continuous adapted process satisfying \eqref{eq.sde7}.
		Then the following are equivalent:
		\begin{enumerate}
			\item[(A)] $\sigma \in L^2(0,\infty)$, \quad $\lim_{t \to \infty} \dfrac{\int_{t}^{\infty} \sigma(s)\, dB(s)}{F^{-1}(t)} = 0$, \; a.s.
			\item[(B)] There is a random variable $\Lambda$, which takes only the values $1,0,-1$, such that
			\[
			\lim_{t \to \infty} \frac{X(t)}{F^{-1}(t)} = \Lambda, \quad \text{a.s.}
			\]
		\end{enumerate}
	\end{theorem}
	\begin{proof} We start by showing (A) $\Rightarrow$ (B). By hypothesis, there is a unique (in the sense of uniqueness up to indistinguishibility) continuous adpated process $X$ which is a solution of the SDE \eqref{eq.sde7}. We work on the a.s. event $\Omega_2$ on which $X$ is pathwise unique. Define 
		\[
		U(t) = \int_{0}^{t} \sigma(s)\, dB(s), \quad t\geq0,
		\]
		and let  $\Omega_3$ be the a.s. event
		\begin{align*}
			\Omega_3 &= \{\omega : U \text{ is a well--defined continuous adapted process}\}. 
		\end{align*}
		Note that since $\sigma \in L^2(0,\infty)$, we have 
		\[
		U(t) \xrightarrow[t \to \infty]{} U(\infty) = \int_{0}^{\infty} \sigma(s)\, dB(s), \quad \text{a.s. on $\Omega_3$}.
		\]
		For each $\omega$ in the a.s. event $\Omega_{1}$ where the limit is finite, define for $t \geq 0$ 
		\[
		\Gamma(t,\omega) = -\left\{ \lim_{T \to \infty}\left( \int_{0}^{T} \sigma(s)\, dB(s)\right)(\omega) 
		- \left(\int_{0}^{t} \sigma(s)\, dB(s)\right)(\omega) \right\}.
		\]
		Hence $\Gamma(t) = -(U(\infty) - U(t))$ for $t \geq 0$, and is well--defined on $\Omega_1\cap \Omega_3$.	Consider next
		the process $V$, which obeys  	
		\[
		V(t) := X(t) - U(t), \quad t \geq 0.
		\]
		Then $V$ is well--defined on $[0,\infty)$ on the a.s. event $\Omega_2\cap \Omega_3$. Moreover since $X$ and $U$ have the same It\^o integral in common,
		\[
		V(t) = V(0)-\int_0^t f(X(s)\,ds, \quad t\geq 0.
		\]
		Thus, since $f$ is continuous and the paths of $X$ are continuous on $\Omega_2$, $t \mapsto V(t)$ is a process with $C^1(0,\infty)$ paths a.s., and we have on $\Omega_2\cap \Omega_3$ that
		\[
		V'(t) = -f(X(t)), \quad t \geq 0.
		\]
		Now, define for all $\omega \in \Omega_4 := \Omega_1 \cap \Omega_2\cap \Omega_3$
		\[
		Z(t,\omega) := V(t,\omega) + U(\infty,\omega), \quad t \geq 0.
		\]
		Note that $\Omega_4$ is an a.s. event, and that $V(\omega)$ is in $C^1(0,\infty)$. Thus, for $\omega \in \Omega_4$, we have for $t\geq 0$
		\begin{align*}
			Z'(t,\omega) &= V'(t,\omega) = -f(X(t,\omega)) \\
			&= -f\big(V(t,\omega) + U(t,\omega)\big)\\
			&= - f\big( Z(t,\omega) - U(\infty,\omega) + U(t,\omega) \big)\\
			&=- f\big( Z(t,\omega) + \Gamma(t,\omega) \big).
		\end{align*}
		If (A) holds, we have that
		\[
		\frac{\Gamma(t,\omega)}{F^{-1}(t)} \;\longrightarrow\; 0, 
		\quad \text{as } t \to \infty \;\; \text{for all } \omega \in \Omega_1,
		\]
		where $\Omega_1$ is an almost sure event.  
		For each $\omega \in \Omega_4$ we have
		\[
		\left\{
		\begin{aligned}
			Z'(t,\omega) &= - f\big( Z(t,\omega) + \Gamma(t,\omega) \big), \quad t \geq 0, \\
			\Gamma(t,\omega) &= o\!\big(F^{-1}(t)\big), \quad t \to \infty, \\
			X(t,\omega) &= Z(t,\omega) + \Gamma(t,\omega), \quad t \geq 0.
		\end{aligned}
		\right.
		\]
		Therefore, applying Theorem~\ref{thm.ThZ} to $Z(\omega)$ for each $\omega\in \Omega_4$ we have that 
		\[
		\lim_{t \to \infty} \frac{Z(t,\omega)}{F^{-1}(t)} 
		= \lambda(\omega), \quad \omega \in \Omega_4,
		\]
		where $\lambda(\omega) \in \{-1,0,1\}$. Clearly  
		since $\Gamma(t,\omega) = o(F^{-1}(t))$ for all $\omega \in \Omega_4$, we have that
		\[
		\frac{X(t,\omega)}{F^{-1}(t)} \;\longrightarrow\; \lambda(\omega), 
		\quad t \to \infty, \;\; \forall \omega \in \Omega_4.
		\]
		This gives (B), and we see that $\lambda$ is a $\mathcal{F}^B(\infty)$--measurable random variable, which takes values in the set $\{-1,0,1\}$ with probability one.
		
		We now need to show that (B) implies (A). Write \eqref{eq.sde7} in integral form 
		\[
		X(t)=\xi-\int_0^t f(X(s))\,ds + \int_0^t \sigma(s)\,dB(s), \quad t\geq 0.
		\]
		Let $\Omega^\ast$ be the almost sure event on which there is a unique continuous adapted solution $X$, and on which $X$ additionally obeys $X(t)\sim \lambda F^{-1}(t)$ as $t\to\infty$, where $\lambda\in \{-1,0,1\}$. By a recapitulation of deterministic arguments, we can show for each $\omega\in \Omega^\ast$ that $t\mapsto f(X(t,\omega))$ is in $L^1(0,\infty)$, so therefore, since $X(t,\omega)\to 0$ as $t\to\infty$ for all $\omega\in \Omega^\ast$, we see for $\omega\in \Omega^\ast$ that 
		\[
		\lim_{t\to\infty} \left(\int_0^t \sigma(s)\,dB(s)\right)(\omega) \text{ exists and is finite}.
		\] 
		By the martingale convergence theorem, it must follow that $\sigma$ is in $L^2(0,\infty)$, which is the first part of (A). Moreover, for all $\omega\in\Omega^\ast$ and $t\geq 0$, we have that $\Gamma(t,\omega)$ is well--defined, and using the fact that \[
		0=\xi-\int_0^t f(X(s))\,ds + \int_0^t \sigma(s)\,dB(s), \quad t\geq 0.
		\] 
		on $\Omega^\ast$, we have for every $\omega\in \Omega^\ast$ that 
		\[
		\Gamma(t,\omega)=X(t,\omega)-\int_{t}^\infty f(X(s,\omega))\,ds, \quad t\geq 0.
		\]
		Using the fact that $X(t,\omega)/F^{-1}(t)\to \lambda(\omega)$ as $t\to\infty$, where $\lambda(\omega)$ is in $\{1,0,-1\}$, we may again reuse deterministic arguments to show that 
		\[
		\lim_{t\to\infty} \frac{\int_t^\infty f(X(s,\omega))\,ds}{F^{-1}(t)}=\lambda(\omega),
		\]
		and so $\Gamma(t,\omega)/F^{-1}(t)\to 0$ as $t\to\infty$, for each $\omega\in \Omega^\ast$, which proves the second part of (A).
	\end{proof} 
	
	In practice, the last result is not easily used, because the condition 
	\[
	\lim_{t\to\infty} \frac{\int_t^\infty \sigma(s)\,dB(s)}{F^{-1}(t)}=0
	\]
	cannot be checked a priori (since the integral depends on the future of the path). However, the asymptotic behaviour of this integral can be determined in terms of a well--defined deterministic function, and once such information is available, it is possible to arrive at conditions which can be checked a priori. The next lemma estimates the decay rate on $\Gamma(t)$ as $t\to\infty$.
	
	\begin{lemma} \label{lemma.tailmartasy} Assume $\sigma$ is continuous, $\sigma \in L^2(0, \infty)$ and obeys
		\begin{equation} \label{d}
			\int_t^\infty \sigma^2(s) \, ds > 0, \quad \forall t \geq 0.
		\end{equation}
		Then,
		\begin{multline*}
			\limsup_{t \to \infty} \frac{\int_t^\infty \sigma(s) \, dB(s)}{\sqrt{2 \int_t^\infty \sigma^2(s) \, ds \log_2 \left( \frac{1}{\int_t^\infty \sigma^2(s) \, ds} \right)}} = 1\\
			= -\liminf_{t \to \infty} \frac{ \int_t^\infty \sigma(s) \, dB(s)}{\sqrt{2 \int_t^\infty \sigma^2(s) \, ds \log_2 \left( \frac{1}{\int_t^\infty \sigma^2(s) \, ds} \right)}}, \quad \text{a.s.}
		\end{multline*}
		where
		\[
		\log_2 x := \log(\log(x)).
		\]
	\end{lemma}
	\begin{proof}
		Since $\sigma$ is in $L^2(0,\infty)$ we may define 
		\[
		T^* = \int_0^\infty \sigma^2(s) \, ds, 
		\]
		Let
		\[
		M(t) = \int_0^t \sigma(s) \, dB(s), \quad t \geq 0 \text{ is a local martingale.}
		\]
		Then $M$ has square variation given by
		\[
		\langle M \rangle (t) = \int_0^t \sigma^2(s) \, ds \xrightarrow[t \to \infty]{} T^*.
		\]
		Then, by the martingale time change theorem, there is a BM $\widetilde{B}$ such that
		\[
		M(t) = \widetilde{B}(\langle M \rangle (t)), \quad 0 \leq t < \infty.
		\]
		Using this we can write 
		\begin{align*}
			\limsup_{t \to \infty} \frac{\int_t^\infty \sigma(s) \, dB(s)}{\sqrt{2 \int_t^\infty \sigma^2(s) \, ds \cdot \log \log \left( \frac{1}{\int_t^\infty \sigma^2(s) \, ds} \right) }}
			&= \limsup_{t \to \infty} 
			\frac{M(\infty) - M(t)}
			{\sqrt{2\bigl(T^* - \langle M \rangle (t)\bigr)\log\log\frac{1}{(T^* - \langle M \rangle (t))}}}\\
			&= \limsup_{t \to \infty} 
			\frac{\tilde{B}\bigl(\langle M \rangle (\infty)\bigr) - \tilde{B}\bigl(\langle M \rangle (t)\bigr)}
			{\sqrt{2\bigl(T^* - \langle M \rangle (t)\bigr)\log\log \frac{1}{(T^* - \langle M \rangle (t))}}}.
		\end{align*}
		If $t \to \infty$, and $s := \langle M \rangle (t)$, then $s \to T^*$. Making this change of variable, we get  
		\[
		\limsup_{t \to \infty} \frac{\int_t^\infty \sigma(s) \, dB(s)}{\sqrt{2 \int_t^\infty \sigma^2(s) \, ds \cdot \log \log \left( \frac{1}{\int_t^\infty \sigma^2(s) \, ds} \right) }}	
		= \limsup_{s \uparrow T^*} 
		\frac{\tilde{B}(T^*) - \tilde{B}(s)}
		{\sqrt{2(T^* - s)\log\log \frac{1}{T^* - s}}}.
		\]
		Let $u = T^* - s$, so that we now have
		\[
		\limsup_{t \to \infty} \frac{\int_t^\infty \sigma(s) \, dB(s)}{\sqrt{2 \int_t^\infty \sigma^2(s) \, ds \cdot \log \log \left( \frac{1}{\int_t^\infty \sigma^2(s) \, ds} \right) }}
		= \limsup_{u \to 0^+} 
		\frac{\tilde{B}(T^*) - \tilde{B}(T^* - u)}
		{\sqrt{2u \log\log (1/u)}}.
		\]
		Now define 
		\[
		\bar{B}(u) := \tilde{B}(T^*) - \tilde{B}(T^* - u), 
		\quad 0 \leq u \leq T^*.
		\]
		Then $\bar{B}$ is also a standard Brownian motion, and we have that
		\begin{align*}
			\limsup_{t \to \infty}
			\frac{\int_t^\infty \sigma(s)\, dB(s)}
			{\sqrt{2 \int_t^\infty \sigma^2(s)\, ds \; \log\log \left( \frac{1}{\int_t^\infty \sigma^2(s) \, ds} \right) }}
			&= \limsup_{u \to 0^+}
			\frac{\bar{B}(u)}
			{\sqrt{2u \log\log (1/u)}}.
		\end{align*}
		Now, by the law of the iterated logarithm for standard Brownian motion, the last $\limsup$ is $1$ with probability $1$, proving the first limit. 
		Likewise, taking the $\liminf$ (and observing that $-\tilde{B}$ is a standard Brownian motion also), we get
		\[
		\liminf_{t \to \infty}
		\frac{\int_t^\infty \sigma(s)\, dB(s)}
		{\sqrt{2 \int_t^\infty \sigma^2(s)\, ds \;\log\log \left( \frac{1}{\int_t^\infty \sigma^2(s) \, ds} \right)}}
		= \liminf_{u \to 0^+} \frac{\bar{B}(u)}{\sqrt{2u \log\log (1/u)}}
		= -1, \; \text{a.s.}
		\]
		as needed.
	\end{proof}
	Using Lemma \ref{lemma.tailmartasy} and Theorem \ref{thm.Xwithtailmart}, we get the following result, which characterises sharply a rate of decay of the noise at which there is a transition from rates of decay like the unperturbed differential equation $y'=-f(y)$ to a slower rate of decay. 
	As we explain, the critical rate is  
	\[
	\Sigma^2(t):=\frac{F^{-1}(t)^2}{ \log\log 1/F^{-1}(t)^2}, \quad t\geq e^e.
	\] 
	More concretely, if the rate of decay of the noise intensity is sufficiently fast so that 
	\[
	\lim_{t\to\infty} \frac{\int_t^\infty \sigma^2(s)\,ds}{\Sigma^2(t)} = 0, 
	\]
	then $X(t)\sim \lambda F^{-1}(t)$ as $t\to\infty$ a.s. (where $\lambda\in \{1,0,-1\}$), while if the noise intensity decays sufficiently slowly so that 
	\[
	\lim_{t\to\infty} \frac{\int_t^\infty \sigma^2(s)\,ds}{\Sigma^2(t)} = +\infty, 
	\] 
	then the rate of decay of the solution is strictly slower than in the deterministic case, because we can conclude that 
	\[
	\limsup_{t\to\infty} \frac{|X(t)|}{F^{-1}(t)}=+\infty, \quad \text{a.s.},
	\]
	
	Before we prove the result, we show that the existence of a limit 
	\[
	\lim_{t\to\infty} \frac{\int_t^\infty \sigma^2(s)\,ds}{\Sigma^2(t)}=\mu \in [0,\infty]
	\]
	is equivalent to the limit
	\[
	\lim_{t\to\infty} \frac{\int_t^\infty \sigma^2(s)\,ds\log\log 1/\int_t^\infty \sigma^2(s)\,ds}{F^{-1}(t)^2}=\mu \in [0,\infty].
	\]
	To see this, notice that if $\varphi(x)=x\log\log(1/x)$ for $x\in 0,e^{-e})$, we have that $\varphi$ is increasing, and moreover that its inverse obeys 
	\[
	\varphi^{-1}(x)\sim \frac{x}{\log\log(1/x)}, \quad x\to 0^+.
	\]
	so in terms of $\varphi$ the limit statements are equivalent to 
	\[
	\lim_{t\to\infty} \frac{\int_t^\infty \sigma^2(s)\,ds}{\varphi^{-1}(F^{-1}(t)^2)}=\mu \in [0,\infty],
	\]
	and
	\[
	\lim_{t\to\infty} \frac{\varphi(\int_t^\infty \sigma^2(s)\,ds)}{F^{-1}(t)^2}=\mu \in [0,\infty],
	\] 
	respectively. The limits are equal by virtue of the fact that 
	if $\phi$ is $\varphi$ or $\varphi^{-1}$, then 
	$\phi(\mu x)/\phi(x)\to \mu$ as $x\to 0^+$ for all $\mu>0$. The  the cases when $\mu=0$ and $\mu=\infty$ can are dealt similarly. We show now how to proceed in the case $\mu=0$. Suppose $\mu=0$ in the first limit. Then for every $\epsilon>0$ there is a $T(\epsilon)>0$ such that 
	\[
	\frac{\int_t^\infty \sigma^2(s)\,ds}{\varphi^{-1}(F^{-1}(t)^2)}<\epsilon, \quad t\geq T(\epsilon). 
	\]
	Thus 
	\[
	\varphi\left(\int_t^\infty \sigma^2(s)\,ds\right)< \varphi(\epsilon \varphi^{-1}(F^{-1}(t)^2)), \quad t\geq T(\epsilon).
	\]
	Hence 
	\[
	\frac{\varphi\left(\int_t^\infty \sigma^2(s)\,ds\right)}{F^{-1}(t)^2}
	< \frac{\varphi(\epsilon \varphi^{-1}(F^{-1}(t)^2))}{\varphi( \varphi^{-1}(F^{-1}(t)^2))}\to \epsilon, \quad t\to\infty,	
	\]
	so sending $\epsilon\to 0^+$, we see that $\mu=0$ in the first statement implies $\mu=0$ in the second. The proof that $\mu=0$ in the second statement implies $\mu=0$ in the first is similar (deduce an inequality,  take $\varphi^{-1}$ on both sides of the inequality, use $\varphi^{-1}(\epsilon x)\sim \epsilon \varphi^{-1}(x)$ as $x\to 0^+$, and finally let $\epsilon\to 0^+$). The equivalence of the statements in the case when $\mu=\infty$ can be obtained by an identical argument by considering the reciprocals of the functions just studied, and then taking limits.  
	
	\begin{theorem}  \label{thm.Xwithouttailmart}.
		Suppose that $f\in C(\mathbb{R};\mathbb{R})$ is increasing, odd and obeys \eqref{eq.fglobalstable}, \eqref{eq.fasypres}, and \eqref{eq.fasypresmu}. Suppose $F$ is defined by \eqref{def.F}. 
		Suppose $\sigma$ is continuous, $\sigma \in L^2(0,\infty)$ and obeys \eqref{d}, and that there exists $\mu \in [0,\infty]$ such that
		\[
		\mu := \lim_{t \to \infty} \frac{\int_t^\infty \sigma^2(s)\, ds }{F^{-1}(t)^2/\log\log(1/F^{-1}(t)^2)}\in [0,\infty]
		\]
		is well--defined. Then:
		\begin{itemize}
			\item[(i)] If $\mu = 0$, then there exists a $\mathcal{F}^B(\infty)$--measurable random variable $\lambda$ with $\mathbb{P}[\lambda\in \{-1,0,1\}]=1$ such that 
			\[
			\mathbb{P}\left[ \lim_{t \to \infty} \frac{X(t)}{F^{-1}(t)} =\lambda \right] = 1.
			\]
			\item[(ii)] If $\mu \in (0,\infty)$, then
			\[
			\mathbb{P}\left[ \lim_{t \to \infty} \frac{X(t)}{F^{-1}(t)} \in \{-1,0,1\} \right] = 0.
			\]
			\item[(iii)] If $\mu=+\infty$, then 
			\[
			\mathbb{P}\left[ \limsup_{t \to \infty} \frac{|X(t)|}{F^{-1}(t)} =+\infty \right] = 1.
			\]
		\end{itemize}
	\end{theorem}
	\begin{proof}
		We start by noticing that the asymptotic estimate on $\int_t^\infty \sigma^2(s)\,ds$ is equivalent to 
		\[
		\mu := \lim_{t \to \infty} \frac{\int_t^\infty \sigma^2(s)\, ds \; \log\log \left( \frac{1}{\int_t^\infty \sigma^2(s) \, ds} \right)}{F^{-1}(t)^2} \in [0,\infty], 
		\]
		by virtue of the discussion before the theorem. 
		
		We prove part (i) first, in which $\mu=0$. Define $\varphi(x)=x\log_2(1/x)$ for $0<x<e{-e}$. On this interval, $\varphi$ is increasing and also $\lim_{x\to 0^+} \varphi(x)=0$.
		
		By Lemma~\ref{lemma.tailmartasy}, we have 
		\[
		\limsup_{t \to \infty} \frac{|\Gamma(t)|}{\sqrt{2\varphi\!\left(\int_t^\infty \sigma^2(s)\, ds\right)}} = 1,\quad \text{a.s.}
		\]
		Since $\mu = 0$, we have
		\[
		\lim_{t \to \infty} \frac{\varphi\!\left(\int_t^\infty \sigma^2(s)\, ds\right)}{F^{-1}(t)^2} = 0.
		\]
		Hence,
		\[
		\limsup_{t \to \infty} \frac{|\Gamma(t)|}{F^{-1}(t)}
		= \limsup_{t \to \infty} 
		\left\{ \frac{|\Gamma(t)|}{\sqrt{2\varphi\!\left(\int_t^\infty \sigma^2(s)\, ds\right)}}
		\times
		\sqrt{\frac{2\varphi\!\left(\int_t^\infty \sigma^2(s)\, ds\right)}{F^{-1}(t)^2}}
		\right\}
		= 0, \quad \text{a.s.}
		\]
		Therefore, we have by Theorem \ref{thm.Xwithtailmart} that
		\[
		\lim_{t \to \infty} \frac{X(t)}{F^{-1}(t)} = \lambda \in \{-1,0,1\}, 
		\quad \text{a.s.},
		\]
		and the proof is complete.
		
		We now prove part (ii). Let
		\[
		A = \left\{ \omega : \lim_{t \to \infty} \frac{X(t,\omega)}{F^{-1}(t)} = \lambda(\omega) \in \{-1,0,1\} \right\},
		\]
		and assume, by way of contradiction, that $\mathbb{P}(A) > 0$. Therefore, for all $\omega \in A$,
		\[
		X(t,\omega) \sim \lambda(\omega) F^{-1}(t), \quad t \to \infty.
		\]
		Since $f$ is odd and obeys \eqref{eq.fasypres} and \eqref{eq.fasypresmu}, we have that 
		\[
		f(X(t,\omega)) \sim f(\lambda(\omega) F^{-1}(t)), \quad t \to \infty,
		\]
		in the case $\lambda\neq 0$ and, in the case where $\lambda=0$, we have $f(X(t,\omega))=o(f\circ F^{-1}(t)$ as $t\to\infty$.
		
		We consider the case where $\lambda=0$ first. Suppose this occurs on the event $A_0$, which we suppose to  be a zero probability event. In that case, since $t\mapsto f\circ F^{-1}(t)$ is in $L^1(0,\infty)$. Moreover, we have that $X(t)\to 0$ as $t\to\infty$, and by hypothesis we have $\sigma$ in $L^2(0,\infty)$ and hence $U(t)$ tend to a finite limit also. Therefore, proceeding as before, we get on $A_0$ 
		\[
		\Gamma(t) = X(t)-\int_t^\infty f(X(s))\,ds
		\] 
		Now, since $f(X(t))=o((f\circ F^{-1})(t))$ as $t\to\infty$ on $A_0$, 
		\[
		\int_t^\infty f(X(s))\,ds = o(F^{-1}(t)), \quad t\to\infty,
		\]
		and so on $A_0$, we have $\Gamma(t)=o(F^{-1}(t))$ as $t\to\infty$. 
		
		By Lemma~\ref{lemma.tailmartasy}, we have that
		\[
		\limsup_{t \to \infty} 
		\frac{|\Gamma(t)|}{\sqrt{2 \, \varphi\!\left(\int_t^\infty \sigma^2(s)\, ds\right)}} 
		= 1, 
		\quad \text{a.s.}
		\]
		Let the event on which this holds be $\Omega_6$. By hypothesis we have  
		\[
		\mu = \lim_{t \to \infty} \frac{\varphi\!\left(\int_t^\infty \sigma^2(s)\, ds \right)}{(F^{-1}(t))^2} > 0.
		\]
		Therefore, for $\omega \in A_0 \cap \Omega_6$ (which is an event of positive probability),
		\[
		\limsup_{t \to \infty} \frac{|\Gamma(t,\omega)|}{F^{-1}(t)}
		= \limsup_{t \to \infty} \frac{|\Gamma(t,\omega)|}
		{\sqrt{2 \varphi\!\left(\int_t^\infty \sigma^2(s)\, ds\right)}} 
		\cdot
		\frac{\sqrt{2 \varphi\!\left(\int_t^\infty \sigma^2(s)\, ds\right)}}{F^{-1}(t)}=\sqrt{2\mu},
		\]
		%	Since
		%	\[
		%	\lim_{t \to \infty} 
		%	\frac{\sqrt{2 \varphi\!\left(\int_t^\infty \sigma^2(s)\, ds\right)}}{F^{-1}(t)} 
		%	= \sqrt{2\mu},
		%	\]
		%	and
		%	\[
		%	\limsup_{t \to \infty} \frac{|\Gamma(t,\omega)|}
		%	{\sqrt{2 \varphi\!\left(\int_t^\infty \sigma^2(s)\, ds\right)}} = 1,
		%	\]
		%	we get
		%	\[
		%	\sqrt{2\mu} \cdot \limsup_{t \to \infty} 
		%	\frac{|\Gamma(t,\omega)|}{\sqrt{2 \varphi\!\left(\int_t^\infty \sigma^2(s)\, ds\right)}} 
		%	= \sqrt{2\mu} > 0.
		%	\]
		where we interpret the righthand side as $+\infty$ in the case $\mu=+\infty$. But this contradicts the conclusion that $X(t)=o(F^{-1}(t))$ on $A_0$, so we have that $\mathbb{P}[A_0]=0$. 
		
		We now turn to the cases where $\lambda\neq 0$. Consider $A_+=\{\omega:\lambda(\omega)>0\}$ and $A-=\{ \omega:\lambda(\omega)<0\}$. We have assumed $\mathbb{P}[A_+\cup A_-]>0$. Suppose that $\mathbb{P}[A_+]>0$; we will show that this leads to a contradiction. The case for $A_-$ follows similarly, exploiting the fact that $f$ is odd, so we only give a proof for $A_+$. 
		
		%We have already on $A_+$ that $f(X(t))\sim f(\lambda F^{-1}(t))=f(F^{-1}(t))$, and the latter function is in $L^1$. Therefore $t\mapsto f(X(t))$ is in $L^1(0,\infty)$ on $A_+$. provided 
		%$t\mapsto f(\lambda F^{-1}(t))$ is in $L^1(0,\infty)$. As in earlier calculations, we have 
		%\begin{equation} \label{eq.intrhs}
		%\int_0^t f(\lambda F^{-1}(s))\,ds = \int_{F^{-1}(t)}^1 \frac{f(\lambda u)}{f(u)}\,du\leq \int_{0}^1 \frac{f(\lambda u)}{f(u)}\,du.
		%\end{equation}
		%We have shown already that there exist constants $c>1$ and $a>0$ such that for all $\lambda>1$ we have 
		%\[
		%\limsup_{x\to\infty} \frac{f(\lambda x)}{f(x)} \leq c\lambda^a,
		%\]
		%and by monotonicity, for $\lambda\in (0,1]$, $f(\lambda x)
		%\leq f(x)$. Therefore, either way we get that the righthand side of 
		%\eqref{eq.intrhs} is finite, and therefore $t\mapsto f(X(t))$ is in $L^1(0,\infty)$. 
		Since $X(t)\to 0$ as $t\to\infty$ and $\sigma\in L^2(0,\infty)$, on $A_+$ we can once again write
		\[
		\Gamma(t) = X(t) - \int_t^\infty f(X(s)) \, ds,
		\]
		Since $\lambda=+1$, we see that both terms on the righthand side are asymptotic to $F^{-1}(t)$ as $t\to\infty$, so the righthand side is $o(F^{-1}(t))$ as $t\to\infty$. However, by reusing the argument above in the case $\lambda=0$, we have that 
		\[
		\limsup_{t\to\infty} \frac{|\Gamma(t)|}{F^{-1}(t)}=\sqrt{2\mu}>0,
		\] 
		giving a contradiction to the supposition that $\mathbb{P}[A_+]>0$, so $\mathbb{P}[A_+]=0$.
		As we noted before, a symmetric argument deals with $A_-$ and so $\mathbb{P}[A_-]=0$. Therefore, we have that $A=A_0\cup A_+\cup A_-$ is a zero probability event, completing the proof of part (ii).  
		
		For part (iii), if we assume there is an event $A'$ of positive probability on which $X(t)=O(F^{-1}(t))$ as $t\to\infty$, we can prove as before on $A'$ that $t\mapsto f(X(t))$ is in $L^1(0,\infty)$ and once again that on $A'$ that $\Gamma$ is given by 
		\[
		\Gamma(t) = X(t) - \int_t^\infty f(X(s)) \, ds.
		\]
		Then, using the argument from Theorem \ref{thm.bigOconverse} applied pathwise in $A'$, we see that on $A'$
		\[
		\int_t^\infty f(X(s)) \, ds = O(F^{-1}(t)), \quad t\to\infty.
		\]
		Since $X(t)=O(F^{-1}(t))$ as $t\to\infty$ on $A'$, we have that 
		$\Gamma(t)=O(F^{-1}(t))$ as $t\to\infty$ on $A'$. However, since 
		$\varphi(\int_t^\infty \sigma^2(s)\,ds)/F^{-1}(t)^2\to \infty$ as $t\to\infty$, we get a.s. that 
		\[
		\limsup_{t \to \infty} \frac{|\Gamma(t,\omega)|}{F^{-1}(t)}
		= \limsup_{t \to \infty} \frac{|\Gamma(t,\omega)|}
		{\sqrt{2 \varphi\!\left(\int_t^\infty \sigma^2(s)\, ds\right)}} 
		\cdot
		\frac{\sqrt{2 \varphi\!\left(\int_t^\infty \sigma^2(s)\, ds\right)}}{F^{-1}(t)}=+\infty,
		\]
		which contradicts what we deduced earlier i.e., that $\Gamma(t)=O(F^{-1}(t))$ as $t\to\infty$ on $A'$, where $A'$ is an event of positive probability. Therefore, the original supposition that $A'$ is an event of positive probability must be false, and so 
		$X(t)=O(F^{-1}(t))$ as $t\to\infty$ can only happen on an event whose probability is zero. From this (iii) follows. 
	\end{proof}
	
	In the last Theorem, two side conditions on $\sigma$ were required, namely that $\sigma$ obeys \eqref{d} and that $\sigma\in L^2(0,\infty)$. These conditions do not restrict our analysis, but need to be treated separately from the main proof. 
	
	In the case that \eqref{d} does not hold, we have that there exists some $T\geq 0$ such that $\int_t^\infty \sigma^2(s)\,ds =0$ for all $t\geq T$. This implies that $\sigma(t)=0$ a.e. for $t\geq T$. 
	Note that $T$ is deterministic. Writing the SDE in integral form, we have 
	\[
	X(t)=X(T)-\int_T^t f(X(s))\,ds + \int_T^t \sigma(s)\,dB(s), \quad t\geq T.
	\]
	Since $\int_t^\infty \sigma^2(s)\,ds=0$ for all $t\geq T$,  we have that 
	\[
	\int_T^t \sigma(s)\,dB(s) = 0, \quad \text{for all $t\geq T$, a.s.}.
	\] 
	Hence for $t\geq T$, the stochastic differential equation reads
	\[
	X(t)=X(T)-\int_T^t f(X(s))\,ds, \quad t\geq T, \text{a.s.}
	\] 
	and so 
	\[
	X'(t)=-f(X(t)), \quad \text{a.e. on $t\geq T$, a.s.} 
	\]  
	Arguing pathwise, if $X(T)=0$, we have $X(t)=0$ for all $t\geq T$, and so $X(t)=o(F^{-1}(t))$ as $t\to\infty$. If $X(T)>0$, we have that $X(t)\sim F^{-1}(t)$ as $t\to\infty$, while if $X(T)<0$, we have $X(t)\sim -F^{-1}(t)$ as $t\to\infty$. Therefore, overall we get 
	\[
	\lim_{t\to\infty} \frac{X(t)}{F^{-1}(t)}=\text{sgn}(X(T)), \quad \text{a.s.}
	\]
	where $\text{sgn}$ denotes the sign (or signum) function. In this case therefore, $\lambda=\text{sgn}(X(T))$ which is a random variable which can assume only the values $\lambda\in \{-1,0,1\}$, as we have the same conclusion as in part (i) of Theorem~\ref{thm.Xwithouttailmart}. This is reasonable, as in this case (where \eqref{d} does not hold), the forcing term switches off entirely after a certain time. Therefore, we can think of this as a extreme case of the situation covered by part (i), in which the noise intensity tends to zero sufficiently quickly. 
	
	We have shown the following.
	
	\begin{theorem}
		\label{thm.Xwithouttailmart1}.
		Suppose that $f\in C(\mathbb{R};\mathbb{R})$ is increasing, odd and obeys \eqref{eq.fglobalstable}, \eqref{eq.fasypres}, and \eqref{eq.fasypresmu}. Suppose $F$ is defined by \eqref{def.F}. 
		Suppose $\sigma$ is continuous, $\sigma \in L^2(0,\infty)$ and does not obey \eqref{d}. 
		Then there exists a $\mathcal{F}^B(\infty)$--measurable random variable $\lambda$ with $\mathbb{P}[\lambda\in \{-1,0,1\}]=1$ such that 
		\[
		\mathbb{P}\left[ \lim_{t \to \infty} \frac{X(t)}{F^{-1}(t)} =\lambda \right] = 1.
		\]
	\end{theorem}
	
	The other restriction on $\sigma$ in Theorem~\ref{thm.Xwithouttailmart} is that $\sigma\in L^2(0,\infty)$. Suppose now that $\sigma\not\in L^2(0,\infty)$ (in which case the condition \eqref{d} does not make sense, so there is no overlap with the other idiosyncratic case dealt with above). 
	We now show that when $\sigma\not\in L^2(0,\infty)$, we either have that $X(t)$ does not tend to zero, or if it does, then it does so so slowly that 
	\[
	\limsup_{t\to\infty} \frac{|X(t)|}{F^{-1}(t)}=+\infty.
	\] 
	We therefore prove the following theorem. 
	\begin{theorem} \label{thm.Xwithouttailmart2}
		Suppose that $f\in C(\mathbb{R};\mathbb{R})$ is increasing, odd and obeys \eqref{eq.fglobalstable}, \eqref{eq.fasypres}, and \eqref{eq.fasypresmu}. Suppose $F$ is defined by \eqref{def.F}. 
		Suppose $\sigma$ is continuous and $\sigma \not \in L^2(0,\infty)$.  
		Then 
		\[
		\mathbb{P}\left[ \limsup_{t \to \infty} \frac{|X(t)|}{F^{-1}(t)} =+\infty \right] = 1.
		\]
	\end{theorem}
	\begin{proof}
		Suppose, by way of contradiction to the desired result that there exists an event $A''$ of positive probability given by  
		\[
		\limsup_{t \to \infty} \frac{|X(t)|}{F^{-1}(t)}=:L<+\infty.
		\]
		Therefore, for each $\omega\in A''$ there is a $L'(\omega)\in (0,\infty)$ such that $|X(t,\omega)|\leq L' F^{-1}(t)$ for $t\geq 0$. Arguing pointwise as in the proof of part (iii) in Theorem~\ref{thm.Xwithouttailmart} above, we see that $t\mapsto f(X(t,\omega))$ is in $L^1(0,\infty)$ for each $\omega\in A''$. Also, by the definition of $A''$, we have that $X(t,\omega)\to 0$ as $t\to\infty$ for all $\omega\in A''$. Now writing the SDE in integral form, and rearrange to get
		\[
		\int_0^t \sigma(s)\,dB(s)=X(t)-\xi + \int_0^t f(X(s))\,ds, \quad t\geq 0.  
		\] 
		Therefore, on $A'$, all the terms on the righthand side have finite limits as $t\to\infty$, and therefore we have shown that there exists an event $A'$ of positive probability such that 
		\[
		\lim_{t\to\infty} \int_0^t \sigma(s)\,dB(s) \text{ is finite on $A'$ a.s.}
		\]
		Now, since $\sigma\not\in L^2(0,\infty)$, we see that the martingale
		\[
		M(t)=\int_0^t \sigma(s)\,dB(s), \quad t\geq 0
		\] 
		has deterministic square variation that obeys $\langle M\rangle(t)\to\infty$ as $t\to\infty$. Therefore, by the martingale time change theorem, and the fact that one--dimensional Brownian motion is recurrent on $\mathbb{R}$ a.s., we have 
		\[
		\limsup_{t\to\infty} M(t)=+\infty, \quad \liminf_{t\to\infty} M(t)=-\infty,
		\]
		so $M$ tends to a finite limit with probability zero. However, our supposition lead to the conclusion that $M$ has a finite limit on the event $A''$, which has a positive probability. These two conclusions are in contradiction to one another, and therefore, the original supposition that $A''$ is a positive probability event must be false. Thus $A''$ is a zero probability event, so its complement must be an a.s. event. But the complement of $A''$ is the event for which $\limsup_{t\to\infty} |X(t)|/F^{-1}(t)=+\infty$, and the proof is complete.
	\end{proof}
	
	\bibliographystyle{unsrt}

\begin{thebibliography}{10}
		
		\bibitem{tahanithesis}
		Alansari, Tahani. Highly nonlinear stochastic and deterministic differential equations with time-varying shocks: asymptotic behaviour and numerical analysis. PhD thesis, Dublin City University, 2019
		
		%	\bibitem{AP:2002(ECP)}
		%	J.~A.~D. Appleby.
		%	\newblock Almost sure stability of linear {I}t\^{o}-{V}olterra equations with damped stochastic perturbations.
		%	\newblock {\em Electron. Comm. Probab.}, 7:223--234, 2002.
		
		%	\bibitem{AP:2004(SubExpItoVol)}
		%	J.~A.~D. Appleby.
		%	\newblock Subexponential solutions of linear {I}t\^{o}-{V}olterra equations with a damped perturbation.
		%	\newblock volume~11, pages 5--10. 2004.
		
		%						\bibitem{AP:2021}
		%						J.~A.~D. Appleby.
		%						\newblock Mean square characterisation of a stochastic {V}olterra integrodifferential equation with delay.
		%						\newblock {\em Int. J. Dyn. Syst. Differ. Equ.}, 11(3-4):194--226, 2021.
		
		\bibitem{apppatt} 
		J.~A.~D.~Appleby and D.~D.~Patterson, Classification of convergence rates of solutions of perturbed ordinary differential equations with regularly 
		varying nonlinearity, Electron. J. Qual. Theory Differ. Equ., Proc. 10th Coll. Qualitative Theory of Diff. Equ. 2016, No. 3, 1--38.
		
		\bibitem{apppatt2} 
		J.~A.~D.~Appleby and D.~D.~Patterson, On necessary and sufficient conditions for preserving
		convergence rates to equilibrium in deterministically and stochastically perturbed differential equations with regularly varying nonlinearity, in: Recent advances in delay differential
		and difference equations, Springer Proceedings in Mathematics and Statistics, Vol. 94, eds.
		M. Pituk, F. Hartung, Springer, Cham, 2014, 1--85. 
		
		\bibitem{appbuck}
		J.~A.~D. Appleby and E.~Buckwar, A constructive comparison technique for determining the asymptotic behaviour of linear functional differential equations with unbounded delay, \emph{Differ. Equ. Dynam. Syst.}, 18 (3), 271--301, 2010.
		
		\bibitem{JAJC:2011szeged}
		J.~A.~D.~Appleby and J.~Cheng, On the asymptotic stability of a
		class of perturbed ordinary differential equations with weak
		asymptotic mean reversion, \emph{E. J. Qualitative Theory of Diff.
			Equ.}, Proc. 9th Coll., No. 1 (2011), pp. 1-36.
		
		\bibitem{ACR:2011(DCDS)}
		J.~A.~D. Appleby, J.~Cheng, and A.~Rodkina.
		\newblock Characterisation of the asymptotic behaviour of scalar linear differential equations with respect to a fading stochastic perturbation.
		\newblock {\em Discrete Contin. Dyn. Syst.}, pages 79--90, 2011.
		
		%		\bibitem{AF:2003(EJP)}
		%		J.~A.~D. Appleby and A.~Freeman.
		%		\newblock Exponential asymptotic stability of linear {I}t\^{o}-{V}olterra equations with damped stochastic perturbations.
		%		\newblock {\em Electron. J. Probab.}, 8:no. 22, 22, 2003.
		%		
		%			\bibitem{AL:2023(AppliedNumMath)}
		%			J.~A.~D. Appleby and E.~Lawless.
		%			\newblock Mean square asymptotic stability characterisation of perturbed linear stochastic functional differential equations.
		%			\newblock {\em Applied Numerical Mathematics}, 2023.
		
		%			\bibitem{AL:2023(AppliedMathLetters)}
		%			J.~A.~D. Appleby and E.~Lawless.
		%			\newblock Solution space characterisation of perturbed linear volterra integrodifferential convolution equations: The ${L}^p$ case.
		%			\newblock {\em Applied Mathematics Letters}, 146:108825, 2023.
		%		
		%		\bibitem{ALpanto2024}
		%		J. A. D. Appleby and E. Lawless, Characterisation of asymptotic behaviour of perturbed deterministic and stochastic pantograph equations, arXiv, 2024. 
		%		
		%		\bibitem{AL:2023(AppliedNumMath)}
		%		J.~A.~D. Appleby and E.~Lawless, Mean square asymptotic stability characterisation of perturbed linear stochastic functional differential equations, {\em Applied Numerical Mathematics}, 200, 80--109, 2024. 
		%		
		%		\bibitem{AL:2023(AppliedMathLetters)}
		%		J.~A.~D. Appleby and E.~Lawless, Solution space characterisation of perturbed linear Volterra integrodifferential convolution equations: the ${L}^p$ case,
		%		\newblock {\em Applied Mathematics Letters}, 146:108825, 2023.
		%		
		%		\bibitem{AL:2024Cesaro}
		%		J.~A.~D. Appleby and E.~Lawless,
		%		\newblock Solution space characterisation of perturbed linear functional and integrodifferential Volterra convolution equations: Ces\`aro limits, arXiv:2407.08416, 2024, 21pp. 
		%		
		%		\bibitem{AL:2024Lp}
		%		J.~A.~D. Appleby and E.~Lawless,
		%		Solution space characterisation of perturbed linear discrete and continuous stochastic Volterra convolution equations: the $\ell^p$ and $L^p$ cases, arXiv:2407.07767, 2024, 24pp.
		
		\bibitem{AppRey02}
		J.~A.~D. Appleby and D. W. Reynolds, Subexponential solutions of linear
		integro-differential equations and transient renewal equations,
		{\em Proc. R. Soc. Edinb.} 132A, 521--543, 2002.		
		
		%				\bibitem{ApRie:2006(SAA)}
		%				J.~A.~D. Appleby and M.~Riedle.
		%				\newblock Almost sure asymptotic stability of stochastic {V}olterra integro-differential equations with fading perturbations.
		%				\newblock {\em Stoch. Anal. Appl.}, 24(4):813--826, 2006.
		
		\bibitem{ARS06}
		J.~A.~D. Appleby, A.~Rodkina and H.~Schurz, Pathwise non-exponential decay rates of solutions of
		scalar nonlinear stochastic differential equation, \emph{Disc. Con. Dynam. Sys. Ser. B.},
		6(4), 667--696, 2006.
		
		\bibitem{BGT}
		N.~H.~Bingham, C.~M~.Goldie and J.~L.~Teugels, {\em Regular Variation},  Encyclopedia of Mathematics and its Applications. Cambridge University Press, 1991.
		
		%		\bibitem{Choveretal} 
		%		J. Chover, P. Ney, and S. Wainger, Functions of probability measures, J. Analyse.
		%		Math. 26 (1972), 255–302.
		
		\bibitem{EvSam:2011}
		V.~Evtukhov and A.~Samoilenko, Asymptotic representations of solutions of nonautonomous ordinary
		differential equations with regularly varying nonlinearities, \emph{Differential Equations}, 47 (5), 627--649, 2011. 
		
		\bibitem{GLS}
		G.~Gripenberg, S.-O. Londen, and O.~Staffans.
		\newblock {\em Volterra Integral and Functional Equations}.
		\newblock Encyclopedia of Mathematics and its Applications. Cambridge University Press, 1990.
		
		\bibitem{Kozma:2012}
		A.~A.~Koz\'ma, Asymptotic behavior of one class of solutions of nonlinear nonautonomous second--order differential equations, \textit{Nonlinear Oscillations}, 14 (4), 497--511, 2012. 
		
		\bibitem{Liu01}
		K.~Liu, Some remarks on exponential stability of stochastic differential
		equations, {\em Stochastic Anal. Appl.}, 19(1):59--65, 2001.
		
		\bibitem{LiuMao98}
		K.~Liu, X.~Mao, Exponential stability of non-linear stochastic evolution equations, {\em Stochastic Process. Appl.}, 78:173--193, 1998.
		
		\bibitem{LiuMao:01a}
		K.~Liu, X.~Mao, \newblock Large time behaviour of dynamical equations with random perturbation
		features, {\em Stochastic Anal. Appl.}, 19(2):295--327, 2001.
		
		\bibitem{MaoOx}
		X.~Mao,  Almost sure polynomial stability for a class of stochastic
		differential equation, {\em Quart. J. Math. Oxford Ser. (2)}, 43(2):339--348, 1992.
		
		\bibitem{Mao92}
		X.~Mao, Polynomial stability for perturbed stochastic differential
		equations with respect to semimartingales, {\em Stochastics Process. Appl.}, 41:101--116, 1992,
		
		\bibitem{Maric2000}
		V. Mari\'c, Regular Variation and Differential Equations, Lecture Notes in Mathematics 1726,
		Springer-Verlag, Berlin, 2000.
		
		\bibitem{Rehak14}
		P. \v{R}eh\'ak, Nonlinear Differential Equations in the Framework of Regular Variation,
		http://users.math.cas.cz/~rehak/ndefrv/ndefrv.pdf, 2014.
		
		\bibitem{SY67}
		A.~Strauss and J.~A. Yorke.
		\newblock Perturbation theorems for ordinary differential equations.
		\newblock {\em Journal of Differential Equations}, 3:15--30, 1967.
		
		\bibitem{SY67b}
		A.~Strauss and J.~A. Yorke.
		\newblock On asymptotically autonomous differential equations.
		\newblock {\em Mathematical Systems Theory}, 1:175--182, 1967.
		
		\bibitem{ZhangTsoi96}
		B.~Zhang, A.~H.~Tsoi.
		\newblock Lyapunov functions in weak exponential stability and
		controlled stochastic systems.
		\newblock {\em J. Ramanujan Math. Soc.}, 11(2):85--102, 1996.
		
		\bibitem{ZhangTsoi97}
		B. Zhang, A.~H.~Tsoi, Weak exponential asymptotic stability of
		stochastic differential equations, {\em Stochastic Anal. Appl.}, 15(4):643--649, 1997.
	\end{thebibliography}

\end{document}